\documentclass[a4paper,11pt,twoside,reqno]{article}

\usepackage{hyperref}
\usepackage[english]{babel}
\usepackage{fancyhdr}
\usepackage{fullpage}
\usepackage[dvips]{graphicx}
\usepackage[T1]{fontenc}
\usepackage{fouriernc}
\usepackage[latin1,applemac]{inputenc}
\usepackage{amsmath}
\usepackage{amsfonts}
\usepackage{amssymb}
\usepackage{mathrsfs}
\usepackage{amsthm}
\usepackage{latexsym}
\usepackage{array}
\usepackage{enumerate}
\usepackage{mathrsfs}
\usepackage{amsxtra}
\usepackage{amscd}
\usepackage{mathtools}
\usepackage[usenames]{color}
\usepackage{rotating}
\definecolor{red}{rgb}{1.0,0.0,0.0}
\def\red#1{{\textcolor{black}{#1}}}
\definecolor{blu}{rgb}{0.0,0.0,1.0}

\definecolor{violet}{rgb}{0.5,0.0,0.4}
\def\violet#1{{\textcolor{black}{#1}}}

\newtheorem{Theorem}{Theorem}[section]

\newtheorem{Proposition}[Theorem]{Proposition}

\newtheorem{theorem}[Theorem]{Theorem}
\newtheorem{proposition}[Theorem]{Proposition}
\newtheorem{lemma}[Theorem]{Lemma}
\newtheorem{corollary}[Theorem]{Corollary}
\newtheorem{definition}[Theorem]{Definition}
\theoremstyle{definition}
\theoremstyle{definition}
\newtheorem{Assumption}[Theorem]{Assumption}
\newtheorem{Remark}[Theorem]{Remark}

\newtheorem{assumption}[Theorem]{Assumption}
\newtheorem{remark}[Theorem]{Remark}

\numberwithin{equation}{section}

\def\cali{\mathcal{I}}

\renewcommand{\phi}{\varphi}

\usepackage{wasysym}
\newcommand{\eyedx}{{\mathrlap{\kern 0.37em\bullet}\ocircle}}

\newcommand{\lipcz}[1]{\operatorname{Lip}_{\mbox{\small{loc}},c_0}(#1)}

\def\R{\mathbb R}
\def\N{\mathbb N}
\def\E{\mathbb E}

\def\cala{{\mathcal A}}

\def\call{{\mathcal L}}

\def\<{\left\langle }
\def\>{\right\rangle }

\def\P{\mathfrak{P}}

\def\N{\mathcal{N}}

\def\norm{{\| \kern -.05em | }}

\def\R{\mathbb R}
\def\N{\mathbb N}

\def\E{\mathbb E}
\def\P{\mathbb P}

\def\F{\mathbb F}

\def\cala{{\mathcal A}}

\def\calc{{\mathcal C}}

\def\cali{{\mathcal I}}

\def\call{{\mathcal L}}
\def\calm{{\mathcal M}}

\def\to{\rightarrow}

\def\<{\left\langle }
\def\>{\right\rangle }

\usepackage{amssymb}
\usepackage{float}
\usepackage{epsfig}
\usepackage{amsmath}
\usepackage[english]{babel}
\usepackage{a4}
\usepackage{amstext}
\usepackage{mathrsfs}
\usepackage{fancyhdr}
\usepackage{latexsym}
\usepackage{mathtools}
\usepackage{amssymb}
\usepackage{fancyhdr}
\usepackage{graphicx}
\usepackage[usenames]{color}
 \definecolor{blu}{rgb}{0.0,0.0,1.0}
 \definecolor{violet}{rgb}{0.5,0.0,0.4}

\usepackage{enumerate}

\usepackage{etoolbox}

\usepackage{xcolor}
\hypersetup{
    colorlinks, % presumo che imponga di non usare i boxes, e di usare invece i colori
    linkcolor={green!35!black},
    citecolor={orange!80!black},
    urlcolor={red!30!black}
}
\definecolor{rose}{rgb}{1.0, 0.33, 0.64}

\renewcommand\epsilon{\varepsilon}
\def\F{{\cal F}}

  \def \R {{\mathbb {R} }}
  \def \extR {{\overline{\mathbb {R}} }}
    \def \N {{\mathbb {N} }}
    
    \def \F {{\mathbb {F} }}
    \def \E {{\mathbb {E} }}

\def \0{{\textbf{0}}}
\def\norm{{\| \kern -.05em | }}

    \newcommand{\be}{\begin{equation}}
    \newcommand{\ee}{\end{equation}}

    \newcommand{\bp}{\begin{proposition}}
    \newcommand{\ep}{\end{proposition}}
    \newcommand{\bl}{\begin{lemma}}
    \newcommand{\el}{\end{lemma}}
    \newcommand{\ba}{\begin{aligned}}
    \newcommand{\ea}{\end{aligned}}
    \newcommand{\bde}{\begin{definition}}
    \newcommand{\ede}{\end{definition}}
    \newcommand{\bc}{\begin{corollary}}
    \newcommand{\ec}{\end{corollary}}
    \newcommand{\br}{\begin{remark}}
    \newcommand{\er}{\end{remark}}
    \newcommand{\bi}{\begin{itemize}}
    \newcommand{\ei}{\end{itemize}}

\begin{document}

\title{Irreversible investment with
fixed adjustment costs:\\ a stochastic impulse control approach}

\author{Salvatore Federico\thanks{\thinspace 
Dip.\ di Economia Politica e Statistica, Universit\`a di Siena (Italy). E-mail:  \texttt{salvatore.federico@unisi.it}.} \and 
Mauro Rosestolato\thanks{\thinspace 
CMAP, Ecole Polytechnique, Paris (France). E-mail: \texttt{mauro.rosestolato@gmail.com}.} \and
 Elisa Tacconi\thanks{\thinspace  Dipartimento di Finanza, Bocconi University,  Milan (Italy). E-mail: \texttt{elisa.tacconi@unibocconi.it}.}}
\maketitle
%\vspace{2cm}
\begin{abstract}
We consider an optimal stochastic impulse control problem over an infinite time horizon motivated by a model of irreversible investment choices with fixed adjustment costs.  By employing techniques of viscosity solutions and relying on  semiconvexity arguments, we prove that the value function is a classical solution to  the associated quasi-variational inequality. This enables us to characterize the structure of the continuation and action regions  and construct an optimal control. Finally, we focus on the linear case, discussing, by a numerical analysis, the sensitivity of the solution with respect to the relevant parameters of the problem.
\end{abstract}


\bigskip

\noindent \textbf{A.M.S.\ Subject Classification}: 93E20 (Optimal stochastic control); 35Q93 (PDEs in connecton woth control and optimization); 35D40 (Viscosity solution); 35B65 (Smoothness and regularity of solutions). \\\\
\noindent \textbf{J.E.L.\ Subject Classification}: C61	(Optimization Techniques; Programming Models; Dynamic Analysis); 
D25	(Intertemporal Firm Choice: Investment, Capacity, and Financing); E22	(Investment; Capital; Intangible Capital; Capacity).
\\\\

\noindent \textbf{Acknowledgements.} \violet{The authors are sincerely grateful to the Associate Editor and to two anonymous Referees for their careful reading and very valuable comment that improved the final version of the paper.} They also thank  Giorgio Ferrari for his very valuable comments and suggestions. Mauro Rosestolato thanks the Department of Political Economics and Statistics of the University of Siena for the kind hospitality in March 2017 and
the grant Young Investigator Training Program financed by Associazione di Fondazioni e Casse di Risparmio Spa supporting this visit.
He also thanks the ERC 321111 Rofirm
 for the financial support.

% \begin{center}
%   \begin{color}
%     {cyan}
%     MR:\ ho commentato le figure, perch\'e mi davano noia nella compilazione lavorando localmente. Guardare il file \texttt{2015-11-05.tex} per capire quali figure decommentare e dove.
%   \end{color}
% \end{center}

\tableofcontents

\section{Introduction}

In this paper we consider a one dimensional stochastic impulse optimal control problem modeling the economic problem of irreversible investment with fixed adjustment cost. 

Let   $X=\{X_t\}_{t\geq 0}$ be a  real valued positive process representing an economic indicator (such as the  GDP of a country, the production capacity of a firm, and so on) on which a planner/manager can intervene.
When no intervention is undertaken, it is  assumed that the process $X$ evolves autonomously according to a time-homogeneous It\^o diffusion.
On the other hand, the planner may act on this process,  increasing its value, by choosing a sequence of interventions dates $\{\tau_n\}_{n\geq 1}$ and of intervention amplitudes $\{i_n\}_{n\geq 1}$, with $i_n>0$\,(\footnote{The fact that only positive intervention, i.e.\ $i_n>0$, is allowed  is expressed in the economic literature of Real Options by saying that the investment is \emph{irreversible}.}).   
%When, at a  random (stopping) time $\tau_n$, the firm decides to undertake an investment $i_n>0$, exercising an impulse control, the capacity is increased correspondingly.
%{\color{red}(sposterei o casserei questa frase)} In order to increase the capacity, the firm can undertake an investment of arbitrary size $i>0$ at any time $t\geq 0$.  
 Hence, the control is represented by a sequence of couples $\left\{(\tau_n,i_n)\right\}_{n\geq 1}$:\ the first component represents the intervention time, the second component the size of intervention. 
The goal of the controller is to maximize over the set of all admissible controls, the expected total discounted  income 
\[\E\left[\int_0^\infty e^{-\rho t} f(X_t)dt- \sum_{n\geq 1}e^{-\rho \tau_n} (c_0i_n+c_1)\right],\]
where $f$ is a reward function, $c_0>0$ and $c_1>0$ represent, respectively, the proportional and the fixed cost of intervention, and $\rho>0$ is a discount factor. 

From the modeling side, our problem is the ``extension'' to the case $c_1>0$ of the same problem already treated in the literature in the case $c_1=0$ (see, e.g., \cite[Ch.\ 4, Sec.\ 5]{P}.  In this respect, it applies to economic problems of capacity expansion, notably irreversible investment problems (\footnote{Other than in \cite[Ch.\ 4, Sec.\ 5]{P}, irreversible and reversible  investment problems with no fixed investment costs are largely treated in the mathematical economic literature, both over finite and infinite horizon. We mention, among others, 
\cite{AE,AFPV, A1, A2, BK, B, CF, CH2,  DDSV, FP, FS, GSZ, SPA0, MOR, AAP, GP, M, MZ, O, RS, W}.}). 
 
From the theoretical side, the introduction of a fixed  cost of control is relevant, as it leads from a problem well posed (in the sense of existence of optimal controls) as a singular control problem  to a problem well posed as an impulse control problem (\footnote{The stochastic impulse control setting has been widely employed in several other applied fields:\ e.g., exchange rate \cite{CZ,JP}, portfolio optimization with transaction costs \cite{K,MO}, inventory  and cash management \cite{CoRi, S1, S2}, and  real options \cite{HT,MT}.}). Such a change is not priceless at the theoretical level.
% --- and probably {\color{cyan}(\`e una pretenziosa supposizione soggettiva, non una constatazione fattuale; io casserei questo commento)} this justifies also the fact that the there is almost no available literature on irreversible investment problems with fixed investment cost\footnote{To this regard, we must mention, as exception, at least the paper \cite{BSZ}. However, despite the (meaningful and relevant)  issue of the delay,  in that paper the dynamics (purely Brownian) and the functional (piecewise linear) are easier to handle.}.
Indeed, the introduction of a fixed cost of control has two unpleasant effects. Firstly, it destroys the concavity of the objective functional even if the revenue function is concave. Secondly, when approaching the problem by dynamic programming techniques (as we do),  the  dynamic programming equation has a nonlocal term and takes the form of  a \emph{quasi-variational inequality} (QVI, hereafter), whereas it is a \emph{variational inequality} in the singular control case.    
%Our approach to the problem relies on dynamic programming methods\footnote{Also the method described in \cite{E} seems employable here.}, which, at the analytical level, lead to the study of a quasi-variational inequality.

\paragraph*{Related literature.} 
 First of all, it is worth noticing that the stochastic impulse control setting has been widely employed in several applied fields:\ e.g., exchange and interest rates  \cite{CZ,  JP, MFM}, portfolio optimization with transaction costs \cite{EH, K,MO}, inventory  and cash management \cite{BS,  CLP, CoRi, DY, DY2, HST, HYZ, MSW, ODV, S1, S2, KY},   real options \cite{HT,MT}, reliability theory \cite{Anderson}. More recently,  games of stochastic impulse control have been investigated with application to pollution \cite{FTo}.
 
From a modeling point of view, the closest works to ours can be considered \cite{Al,AlLe,CS, E, K}. On the theoretical side,  starting from the classical book \cite{BeLio}, several works investigated  QVIs associated to stochastic impulse optimal control in $\R^n$. Among them, we  mention the recent \cite{GuoWu} in a diffusion setting and  \cite{BEM, DGW} in a jump-diffusion setting. In particular, \cite[Ch.\,4]{BeLio}  deals with Sobolev type solutions, whereas \cite{GuoWu} deals with viscosity solutions. These two works prove a $W^{2,p}$- regularity, with $p<\infty$, for the solution of QVI, which, by classical Sobolev embeddings, yields a $C^1$-regularity.
However, it is typically not easy  to obtain by such regularity 
information on the structure of the so called \emph{continuation and action regions}, hence on the candidate optimal control. If this structure is established, then one can try to prove a verificiation theorem to prove that the candidate optimal control is actually optimal. In a stylized one dimensional example, \cite[Sec.\,5]{GuoWu} successfully employs  this method by exploiting the regularity result proved in \cite[Sec.\,4]{GuoWu} to depict the structure of the continuation and action region for the problem at hand. Concerning verification, we need to mention the recent paper \cite{Belak}, which provides a non-smooth verification theorem in a quite general setting based on the stochastic Perron method to construct a viscosity solution to QVI; also this paper, in the last section, provides and application of the results to a one dimensional problem with an implementable solution. 
% In particular, the latter reference provides, in a fairly general framework, the viscosity characterization of the value function and its regularity. Suitably adapted to our context, it  provides ty the theoretical background of our solution. 
In dimension one other approaches, based on \emph{excessive mappings} and iterated optimal stopping schemes, have been successfully  employed in the context of stochastic impulse control (see \cite{Al, AlLe, E, HSZ}). More recently, these methods have been extended to Markov processes valued in metric spaces (see \cite{C}); again a complete description of the solution is shown in one dimensional examples.

\paragraph*{Contribution.}  From the methodological side our work is close to \cite{GuoWu}. As in the latter, we follow a \emph{direct analytical method} based on viscosity solutions and we do not employ a  \emph{guess-and-verify} approach(\footnote{See, e.g., \cite{BSZ,CoRi, K, JZ, MO} and, in a much more general context of jump-diffusion, \cite[Ch.\,6]{OS} for the guess-and-verify approach.}). Indeed, we directly provide \emph{necessary optimality conditions} that, by uniqueness, fully characterize the solution. In particular, we do not postulate the smooth-fit principle, as it is usually done in the guess-and-verify approach, but  we prove it directly(\footnote{The smooth-fit principle has also been established, when the diffusion is assumed to be transient, by techniques based on excessive function  (see \cite{ST})}.). 
To the best of our knowledge a rigorous analytical treatment as ours of the specific problem treated in this paper seems to be still missing in the literature. It is important to notice that our analysis yields a a complete and implementable characterization of the optimal control policy through the identification of the continuation and action regions. Since  the aforementioned techniques based on excessive mappings seems to be perfectly employable to our problem (even under weaker assumption), it is worth to point out that our contribution is \emph{methodological}. As it is well known, the (implementable) characterization of the optimal control in stochastic impulse control problems is a challenging task in dimension larger than one. Hence, it is important to have at hand an approach like ours that might be generalized to address impulse control problems in multi-dimensional setting. To this regard, it is worth to notice the following. 
\begin{itemize}
\item To the best of our knowledge, the only study providing a complete picture of the solution in dimension two --- through a two dimensional $(S,s)$-rule --- is the recent paper  \cite{BB}. The techniques used there are analytical and based on the study of QVI's. Unfortunately, in this paper, the authors are able to provide a complete solution only in a very specific case. 
\item In the presence of semiconvex data, our approach to prove $C^1$ regularity of the value function based on semiconvexity jointly with the viscosity property, unlike \cite{GuoWu}, might be successful to prove a directional regularity result just along nondegenerate directions (see \cite{FP} in a singular control context).     
\item The directional regularity result mentioned above might be sufficient to derive the right optimality condition to solve the control problem (see again \cite{FP} in a singular control context).
\end{itemize}
%We also stress, although one might invoke the results  of \cite{GuoWu} to establish the  $C^1$-regularity of the value function(\footnote{The only problem would be that our data (notably, the reward function $f$) are not globally Lipschitz, unlike in \cite{GuoWu}; but this would be not a major issue, as the arguments for the regularity in \cite{GuoWu} are local.}), we here follow another path of proof  relying on the semiconvexity issue, as we will describe  more in detail below.  

\paragraph{Contents.} In Section \ref{sec:formulation} we set up the problem. In Section \ref{sec:pre} we state some preliminary results on the value function $v$, in particular we show that it is semiconvex. In Section \ref{sec:DPE}
we derive  QVI  associated to  $v$ and show that it solves the latter in viscosity sense. After that, we prove that $v$ is of class $C^2$ in the continuation region (the region where the differential part of QVI holds with equality, see below) and  of class $C^1$ on the whole state space (Theorem \ref{th:viscosoluQVI}, our first main result), hence proving the smooth fit-principle.  We prove the latter result relying just on the semiconvexity of $v$ and exploting the viscosity supersolution property; unlike \cite{GuoWu}, this allows to avoid the use of a deep theoretical result such as the Calderon-Zygmund estimate. So, with respect to the aforementioned reference, our method of proof is cheaper from a theoretical point of view; on the other hand, it heavily relies on assumptions guaranteeing the semiconvexity of $v$. 
% ; however, we think that, when semiconvexity holds, this tecnique can be used to cover other cases even in the multi-dimensional case. {\red{SF: SU QUEST'ULTIMA AFFERMAZIONE: PENSARE QUANTO ESSA E? SUFFRAGATA DA UN CONVINCIMENTO TECNICO; ALTRIMENTI SOPPRIMERLA}}  
% After this formal passage, we invoke results from \cite{GuoWu} to state that the value function is a viscosity solution of QVI and that its suitably smooth ($C^1$ over the whole state space and $C^2$ on the continuation region). 
%and prove that the latter is a viscosity solution of the former.  
In Section \ref{sec:exp} we use the latter regularity to establish the structure  of the \emph{continuation} and \emph{action} regions --- the real unknown of the problem --- showing that they are both intervals. This allows to express explicitly $v$ up to the solution of a nonlinear algebraic system of three variables (Theorem \ref{theoremformvaluefunction}, our second main result).  In Section \ref{sec:opt}, relying on the results of the previous section, we are able to construct an optimal control policy (Theorem \ref{verificationtheorem}, our third main result). The latter turns out to be based on the so called $(S,s)$-\emph{rule} (\footnote{This is a  well known rule  in the economic literature of inventory problems, see  \cite{AHM, S1,S2}.}):\ the controller acts whenever the state process reaches a minimum level $s$ (the ``trigger'' boundary) and brings immediately the system at the level $S>s$ (the ``target'' boundary).
Finally, in Section \ref{sec:num}, we provide a numerical illustration of the solution  when $X$ follows  a geometric Brownian motion dynamics between intervation times, analyzing the sensitivity of the solution with respect to the volatility coefficient $\sigma$ and  to and the fixed cost $c_1$.

\section{Problem formulation}\label{sec:formulation}

We introduce some notation. We set  
 $$
%\overline{\R}\coloneqq [-\infty,+\infty],\qquad
\R_+\coloneqq [0,+\infty),\qquad
\extR_+\coloneqq [0,+\infty],
\qquad\R_{++}\coloneqq (0,+\infty).
% \qquad\overline{\R}_{++}\coloneqq (0,+\infty].
 $$ The set $\R_{++}$ will be the state space of our control problem. 
 Throughout the paper we  adopt the conventions $e^{-\infty}=0$ and  $\inf\emptyset =\infty$.
Moreover, we simply use the symbol $\infty$ in place of $+\infty$ when positive quantities are involved and no confusion may arise. 
Finally, the symbol 
$n$ will always denote a   natural number.

Let $(\Omega,\mathcal{F},% \mathbb{F}\coloneqq
 \{\mathcal{F}_t\}_{t\geq 0}, \mathbb{P})$ be a filtered probability space satisfying the usual conditions and supporting a a one dimensional  Brownian motion  $W=\{W_t\}_{t\geq 0}$.
We denote $
%{\displaystyle \mathcal{F}_\infty\coloneqq \sigma\left\{\mathcal{F}_t\colon t> 0\right\}}, \
\mathbb{F}\coloneqq \{\mathcal{F}_t\}_{t\in \overline{\mathbb{R}}_+}$, where we set
$\displaystyle{\mathcal{F}_\infty\coloneqq \bigvee_{t\in \mathbb{R}_+}\mathcal{F}_t}$.
We take $b,\sigma\colon \mathbb{R}\rightarrow \mathbb{R}$  satisfying the following
\begin{assumption}\label{eq:2017-03-23:00}
$b,\sigma\colon \mathbb{R}\rightarrow \mathbb{R}$
are Lipschitz continuous functions, with Lipschitz constants $L_b,L_\sigma$, respectively,  identically equal to $0$ on  $(-\infty,0]$,
and  with $\sigma>0$ on $\mathbb{R}_{++}$.
Moreover, $b,\sigma\in C^1(\R_{+})$,
% {\color{rose}(check che non si usi mai la derivata in 0)},
and $b',\sigma'$ are Lipschitz continuous on $\mathbb{R}_{++}$, with Lipschitz constants  $\tilde{L}_b, \tilde{L}_\sigma>0$, respectively.
 % $|b'|,|\sigma'|$ are bounded on $\mathbb{R}_{++}$ by $L_b$ and $L_\sigma$ respectively,
 % $b,\sigma\equiv 0$ on
%  {\color{rose}(in questo modo stiamo assumento che $b'(0)=\sigma'(0)=0$: lo usiamo? altrimento mettere qualcosa del tipo $C^1(\mathbb{R}\setminus \{0\}$ con boundedness: controllare))}, .
% {\color{rose}Aggiungere che $b'$ e $\sigma'$ sono $\tilde{L}$-Lipschitz-continuous su $\mathbb{R}_{++}$.}
% and
%  \begin{equation} 
%    |b(x_1)-b(x_2)|\leq L_b|x_1-x_2|,
%    \qquad
%    |\sigma(x_1)-\sigma(x_2)|\leq L_\sigma|x_1-x_2|,\qquad \forall x_1,x_2\in \mathbb{R},
%  \end{equation}
%      for some $L_b,L_\sigma>0$.
%Moreover, the diffusion $X^\circ$ is nondegenerate in $\R_{++}$, i.e.\  $\sigma(x)>0$ for every $x\in \mathbb{R}_{++}$. 
\end{assumption}
\begin{remark}
The requirement that $b',\sigma'$ are Lipschitz continuous is typical when one wants to prove  the semiconvexity/semiconcavity of the value function in stochastic optimal control problem (see, e.g., the classical reference \cite[Ch.\,4, Sec.\,4.2]{YongZhou} in the context of regular stochastic control; \cite{BEM} in the context of impulse control). We use this assumpton since,  as outlined in the introduction, in our approach the proof of the semiconvexity of the value function will be a crucial step  towards the proof of the $C^1$ regularity. 
\end{remark}
%
%
%Let $\tau$ be a finite $\mathbb{F}$-stopping time and denote 
%$$\Omega_\tau\coloneqq \Omega\times [\tau,\infty), \ \ \ \ \mathcal{P}_\tau\coloneqq \mathcal{P}|_{\Omega_\tau}.$$
%Let $\xi$ be an $\mathcal{F}_\tau$-measurable random variable.

Let $\tau$ be a (possibly not finite) $\mathbb{F}$-stopping time and let $\xi$ be an $\mathcal{F}_\tau$-measurable random variable.
By standard SDE's theory with Lipschitz coefficients, Assumption~\ref{eq:2017-03-23:00} guarantees 
that
 % for each $x\in\R$
 there exists a unique (up to  undistinguishability) $\mathbb{F}$-adapted process $Z^{\tau,\xi}=\{Z^{\tau,\xi}_t\}_{t\geq 0}$ with continuous trajectories on $[\tau,\infty)$, 
 such that 
% $\P\times \mathit{dt}$-null measure sets in $(\Omega_\tau,\mathcal{P}_\tau)$)  
% $$Z=Z^{\tau,\xi}:(\Omega_\tau, \mathcal{P}_\tau)\to \R \ \ \ \mbox{measurable}$$ verifying
 \begin{equation}
   \label{eq:SDE}
     Z^{\tau,\xi}_t=
     \begin{dcases}
       0& \mbox{for\ }t\in[0,\tau)\\
\xi +\int_\tau^t b(Z^{\tau,\xi}_s)ds
     +\int_\tau^t \sigma(Z^{\tau,\xi}_s)dW_s & \P\mbox{-a.s.}, \mbox{for\ }t\geq \tau.
   \end{dcases}
 \end{equation}
% morever $Z$ can be chosen (and we always do that) such that
%$$
% t\mapsto Z_{t+\tau}(\omega) \ \ \mbox{is continuous} \ \forall \omega\in\Omega.
%$$
Moreover, by
a straightforward adaptation of \cite[Sec.~5.2, Prop.~2.18]{KS} to random initial data, we obtain
 \begin{equation}\label{monotone}
\xi,\eta\ \mathcal{F}_\tau\mbox{-measurable random variables, }
\xi\leq \eta\ \mbox{$\mathbb{P}$-a.s.}
  \ \Longrightarrow \ Z^{\tau,\xi}_{t+\tau}\leq Z^{\tau,\eta}_{t+\tau} \  \P\mbox{-a.s.}, \ \forall t\geq 0.
\end{equation}
% \begin{equation}\label{monotonebis}
% x,y\in\R, \ x\leq y 
%   \ \Longrightarrow \ Z^{0,x}_t\leq Z^{0,y}_t \  \P\mbox{-a.s.}, \ \forall t\geq 0. 
% \end{equation}
%We use the convention $\int_{+\infty}^{+\infty}=0$.
%Without loss of generality we can assume that $Z^{\tau,\xi}$ is continuous on $[\tau,+\infty)$.
%  As a reference process on this probability space we consider a time-homoegeneous diffusion $X^\circ\coloneqq\{X^\circ_t\}_{t\in\R_+}$  solving an SDE
%  \begin{equation}
%  dX_t^\circ =\mu(X_t^\circ)dt+\sigma(X_t^\circ)dW_t, \ \ \ \ t\in\R_{++},
%  \end{equation}
% where $\mu,\sigma\colon\R\rightarrow \R$ satisfy the following.
Now fix $x\in \R_{++}$.
By
\eqref{monotone}
%\eqref{monotone}
 and Assumption~\ref{eq:2017-03-23:00}, it 
follows that  $Z^{0,x}$ takes values in $\R_+$. 
Due to the nondegeneracy assumption on $\sigma$ over $\R_{++}$, as a consequence of the results of \cite[Sec.~5.5.C]{KS},  the process  $Z^{0,x}$ is a (time-homogeneous) regular 
diffusion on  $\R_{++}$; i.e.,
setting 
$
\tau_{x,y}\coloneqq \inf\left\{t\geq 0\colon Z^{0,x}_t=y\right\},
$
one has 
$$
\P\{\tau_{x,y}<\infty\}>0  \ \ \ \forall y\in\R_{++}.
$$
In Appendix we show that  Assumption~\ref{eq:2017-03-23:00} guarantees that the boundaries $0$ and $+\infty$ are natural for $Z^{0,x}$  in the sense of Feller's classification.

\smallskip
We introduce now a  set of
% In this setting, we consider, as set of
 admissible controls and their corresponding controlled process.  As a  set of
% In this setting, we consider, as set of
admissible controls (i.e., feasible investment strategies) we consider the set $\mathcal{I}$ of all sequences of couples  $I=\left\{(\tau_n,i_n)\right\}_{n\geq 1}$ such that:
\begin{enumerate}[(i)]
\item
$\{\tau_n\}_{n\geq 1}$ is an increasing sequence of $\extR_+$-valued $\mathbb{F}$-stopping times
 % ()
 such that $\tau_n<\tau_{n+1}$ $\P$-a.s. over the set $\{\tau_n<\infty\}$ and  
 \begin{equation}
   \label{eq:2015-11-23:00}
\lim_{n\rightarrow \infty}\tau_n =\infty\ \ \ \mathbb{P}\mbox{-a.s.};
\end{equation}
%almost surely when $n\rightarrow+\infty$ (this latter condition 
\item $\{i_n\}_{n\geq 1}$ is a sequence of $\mathbb{R}_{++}$-valued random variables  % , taking values in $\R_{++}$, and 
  such that $i_n$ is $\mathcal{F}_{\tau_n}$-measurable for every $n\geq 1$;
\item The following integrability condition holds:
  \begin{equation}
    \label{eq:2015-12-18:01}
\sum_{n\geq 1}    \mathbb{E} \left[ e^{-\rho \tau_n}(i_n+1) \right] <\infty.
  \end{equation}
% and, for all $t\in \mathbb{R}_+$
%   \begin{equation}
%     \label{eq:2017-03-23:02}
% \mathbb{E}
%  \left[ 
%  \left( \sum_{n\geq 1}    
%  i_n\mathbf{1}_{[0,\tau_n]}(t) \right) ^{2\gamma}\right] <+\infty  
%   \end{equation}

\end{enumerate}
For $n\geq 1$, $\tau_n$
represents an  intervention time, whereas
$i_n$ represents 
the intervention size at the corresponding intervention time $\tau_n$.
Condition 
\eqref{eq:2015-11-23:00}
ensures that, within a finite time interval, only a finite number of  actions are executed.
We allow the case  $\tau_n=\infty$ definitively, meaning that only a finite number of actions are taken. Condition \eqref{eq:2015-12-18:01} ensures that the functional defined below is well defined. We call  \emph{null} control any sequence $\{(\tau_n,i_n)\}_{n\geq 1}$ such that $\tau_n=\infty$ for each $n\geq 1$ and denote any of  them by $\emptyset$.  
Notice that using the same  notation $\emptyset$ for the null controls is not ambiguous with regard to the control problem we are going to define, as any null control will give rise to the same payoff.

\smallskip
Given a control $I\in \mathcal{I}$, an initial stopping time $\tau\geq 0$ and a random variable $\xi>0$ $\P$-a.s. $\mathcal{F}_\tau$-measurable, we denote by $X^{\tau,\xi,I}=\{X^{\tau,\xi,I}_r\}_{r\in [0,\infty)}$ the unique (up to indistinguishability)
c\`adl\`ag process on $[\tau,\infty)$
solving the SDE (in integral form)
\begin{equation}
  \label{2017-09-17:01}
  X^{\tau,\xi,I}_t=
  \begin{dcases}
    0& \mbox{for\ }t\in [0,\tau)\\
\xi+ \int_\tau^t b(X^{\tau,\xi,I}_s) ds +\int_\tau^t\sigma(X^{\tau,\xi,I}_s)dW_s+
  \sum_{n\geq 1}\mathbf{1}_{[\tau,r]}(\tau_n)\cdot i_n& \mbox{for\ }t\in [\tau,\infty)
\end{dcases}
\end{equation}
If $t=0$ and $\xi\equiv x\in \R_{++}$ then we denote $X^{0,\xi,I}$ by $X^{x,I}$.
% {\red{COSA? CADLAG? A CHE SERVE?}} We assume that, for every $n\geq 1$ and $\omega\in \Omega$,
% $$
%  [\tau(\omega)\vee \tau_n(\omega),\tau(\omega)\vee \tau_{n+1}(\omega))\rightarrow \mathbb{R},\ 
%  t\mapsto X^{\tau,\xi,I}_t(\omega).
% %  \ \ \mbox{is continuous} \ \forall r\in
% % ), \ \forall\omega\in\Omega.
% $$
It is easily seen that, if $\tau'$ is another stopping time such that $\tau'\geq \tau$, then 
the following flow property holds true
\begin{equation}
  \label{eq:2017-09-17:02}
  X^{\tau,\xi,I}_t=X^{\tau',X^{\tau,\xi,I}_{\tau'^-},I}_t\ \forall t\geq \tau', \  \mathbb{P}\mbox{-a.e..}
\end{equation}
Note that, up to undistinguishability,  we have 
$
X^{x,\emptyset}= Z^{0,x}$.
 Moreover, setting by convention $\tau_0\coloneqq 0$, $ i_0\coloneqq 0$, and $X_{0^-}\coloneqq x$,  we have recursively on $n\in \N$
$$
X^{x,I}_t
=
Z^{\tau_n,
X^{x,I}_{\tau_n}
}_t
\ \forall t\in[\tau_n,\tau_{n+1}),
\  \mbox{$\P$-{a.s.}}.$$
Then, by \eqref{monotone}, we have
  the following monotonicity of the controlled  process with respect to the initial data 
\begin{equation}\label{ineq} 
X_t^{x,I}\leq  X_t^{x',I}\ 
 \P\mbox{-a.s.},
\forall t\geq 0, \  \forall I\in \mathcal{I}, \ \forall x,x'\colon 0<x\leq x'.
 \end{equation} 

Next, we introduce the optimization problem.
Given $\rho>0$,
$f\colon\R_{++}\to \R_{++}$ measurable,
$c_0>0$, $c_1>0$, we 
define the payoff functional 
$J$
 by 
\begin{equation}\label{fun}
J(x,I)\coloneqq\E\left[\int_0^{\infty}e^{-\rho t} {f}(X^{x,I}_t)dt-\sum_{n\geq 1}e^{-\rho\tau_n}(c_0 i_n+c_1)\right],\qquad
\forall x\in \R_{+},\ \forall I\in\mathcal{I}.
\end{equation}
%$C:\R_{++}\to \R_{++}$ is the affine function $i\mapsto C(i)\coloneqq c_0i +c_1$ 
% where $c_0>0$, $c_1>0$.
%Note that, for every $I\in\cal{I}$, $J(x,I)$ is well defined  and $J(x,I)>-\infty$ by \eqref{eq:2015-12-18:01}.

%$J>-\infty$.
%It is clear that $J$ is not well defined if both the integral and the sum appearing in \eqref{fun} are $+\infty$. In order to avoid this situation, % and in a consistent way with the maximization problem \eqref{eq:2015-11-23:01} below, for every $x\in \mathbb{R}_{++}$,
% we introduce the following class of admissible controls:
% \begin{equation}
%   \label{eq:2015-12-18:00}
%   \mathcal{C}_{\mathrm{ad}}(x)\coloneqq \left\{ I\in \mathcal{I}\colon \ \liminf_{T\rightarrow +\infty}
% \E\left[\int_0^{T}e^{-\rho t} {f}(X^{x,I}_t)dt-\sum_{n\geq 1}\mathbf{1}_{[0,T]}(\tau_n)e^{-\rho\tau_n} C(i_n)\right]>-\infty
%  \right\}.
% \end{equation}
% Then  $J\colon \mathbb{R}_{++}\times \mathcal{I}\rightarrow \overline{\mathbb{R}}$.

We notice that \eqref{eq:2015-12-18:01}
% {\color{rose}and the sublinearity of $f$}
 and the fact that $f$ is bounded from below ensure that   
$J(x,I)$   is well defined and takes values in $\R\cup\{\infty\}$.
% \begin{Remark} Note that $L_b$ in Assumption \ref{ass:rho} may be negative. In this case, negative values of $\rho$ are allowed. 
% \end{Remark}

We will make use of the
% We introduce the
following assumption on $f$.

\begin{assumption}\label{2016-11-06:00}
 $f\in C^1(\R_{++};\R_+)$,  $f'>0$,  $f'$ is strictly decreasing,
% Che la $f'$ sia concava ci serve. Che la $f''$ esista ci serviva/serve per il caso lineare, per la proof della semiconvessit\`a.
 and
$f$  satisfies the  Inada condition at $\infty$: 
 \begin{equation*}
%   f'(0^+)\coloneqq \lim_{x\rightarrow 0^+}f'(x)=+\infty
%\qquad
    f'(\infty)\coloneqq \lim_{x\rightarrow \infty}f'(x)=0.
 \end{equation*}
 Finally, without loss of generality, we assume that
 $f(0^+)\coloneqq {\displaystyle\lim_{x\rightarrow 0^+} f(x)=0}$.
\end{assumption}
Note that 
\begin{equation}\label{MB}
M_b:=\left(\sup_{x\in \R_{++}} b'(x)\right)^+<\infty
\end{equation} by Assumption \ref{eq:2017-03-23:00}.
The following assumption will ensure finiteness for the problem (Proposition~\ref{prop:preB}). 
\begin{assumption}\label{ass:rho}
 $\rho>M_b$.
%\hfill$\square$
\end{assumption}

\noindent 
Assumptions \ref{eq:2017-03-23:00}, \ref{2016-11-06:00}, and \ref{ass:rho}
will be standing  through the rest of the manuscript.

\medskip
 The optimal control problem that we address consists in maximizing the functional \eqref{fun} over $I\in\mathcal{I}$, i.e., for each $x\in \mathbb{R}_+$, we consider the maximization problem
 \begin{equation}
   \label{eq:2015-11-23:01}\tag{{\bf P}}
\sup_{I\in \mathcal{I}}
J(x,I).
 \end{equation}
%Due to the fact that any empty strategy belongs to every $\mathcal{C}_{\mathrm{ad}}(x)$, to \eqref{solutionstate}, and to the fat that $f>0$ in $\mathbb{R}_{++}$, in \eqref{eq:2015-12-18:00} we could replace $-\infty$ by $0$.

\smallskip
\begin{remark}\label{rem:cost}
The fact that $c_1>0$ means that there is a fixed cost when the investment occurs. This provides that  \eqref{eq:2015-11-23:01} is well posed as an \emph{impulse} control problem, i.e.\ optimal controls can be found within the class of impulse controls . If it was $c_1=0$ (only proportional intervention cost), the setting providing existence of optimal controls would be the more general \emph{singular} control setting (see e.g. \cite[Ch.~4]{P}). For comparison between impulse and singular control we refer to \cite{BLL}; for the relevance of the introduction of the fixed cost we refer to \cite{OUZ}, where the asymptotics for $c_1\to 0$ is investigated. In Subsection \ref{sec:fixedcost}, we comment this issue through the numerical outputs.

We also notice that one might consider more general intervention costs $C:\R_{++}\to\R_+$ increasing and convex, (e.g. $C(i)=\alpha i^2+\beta i+c_1$ with $\alpha,c_1>0$ and $\beta\geq 0$). We believe that, at least for a suitable subclass of such cost functons, the solution would depict the same structure as the one we provide here  in the affine case (i.e.\ $C(i)=c_0i+c_1$). On the other hand, we underline that at many points our proofs make use of the affine structure of the cost and the generalization seems to be not straightforward.
\end{remark}

\section{Preliminary results on the value function}\label{sec:pre}

In this section we introduce the value function associated with \eqref{eq:2015-11-23:01} and
establish some  basic properties of it.
 We define the value function $v$ by
\begin{equation}\label{value}
v(x)\coloneqq   \sup_{I\in\mathcal{I}}J(x,I), \ \ \ \forall x\in \R_{++}.
\end{equation}
%Indeed, we immediately
We notice that $v$ is $\overline{\mathbb{R}}_{+}$-valued,
 as by Assumption \ref{2016-11-06:00}
 \begin{equation}
  \label{eq:2017-04-07:03}
 v(x)\geq J(x,\emptyset)=\hat{ v}(x)\coloneqq \mathbb{E} \left[ \int_0^{\infty}e^{-\rho t} f(X^{x,\emptyset}_t)dt\right]\geq 0 \qquad \forall x\in \mathbb{R}_{++}.
\end{equation} 
Note that $\hat v$ is nondecreasing as $f'>0$ (Assumption \ref{2016-11-06:00}) and by \eqref{ineq}. 
% and that $v$ is  increasing, as $X^{x,I}_t(\omega)$ and then  $J(x,I)$ are  increasing in $x\in \mathbb{R}_{++}$, for every $t\in \mathbb{R}_+$, $\omega\in\Omega$, $I\in \mathcal{I}$.
%Notice that all the terms in \eqref{costgamma}-\eqref{estimate} below  are finite and positive, as $\rho>0$, $\beta\geq 0$, $\sigma>0$, and $\gamma\in(0,1)$.  {\color{cyan}MR: a che serve sta frase??}

\begin{Proposition}\label{prop:mono} $v$ is nondecreasing.
%$\gamma$-H\"older continuous {\color{rose}(ora non serve a nulla questo punto: togliere o lasciarlo per quanto scegliamo la $f$ particolare?)}.
%\item\label{2015-11-23:05B} {\color{red}SF: CREDO CHE QUI OCCORRA L'IPOTESI $\rho>K_p'$} For each $\epsilon>0$, $v$ is Lipschitz continuous in the interval $[\varepsilon,+\infty)$.
% with Lipschitz constant  $K_\varepsilon= C_\gamma^{-1}\varepsilon^{\gamma-1}$. 
\end{Proposition}
\begin{proof} Let $0<x\leq x'$. Since $f'>0$ (see Assumption \ref{2016-11-06:00}), from \eqref{ineq} we get $J(x;I)\leq J(x';I)$ for every $I\in\mathcal{I}$. The claim follows by taking the supremum over $I\in \mathcal{I}$.
\end{proof}

We denote by $f^*$ the Fenchel-Legendre transform of $f$ on $\mathbb{R}_{++}$:
\begin{equation}
  \label{eq:2017-04-07:02}
  f^*(\alpha)\coloneqq \sup_{x\in {\mathbb{R}_{++}}} \big\{ f(x)-\alpha x \big\},\qquad \forall \alpha\in \mathbb{R}_{++}.
\end{equation}
Nonnegativity and continuity of $f$ (see Assumption \ref{2016-11-06:00}) and the condition  $f'(\infty)=0$ (again Assumption \ref{2016-11-06:00}) guarantee that $0\leq f^*(\alpha)<\infty$ for all $x\in \mathbb{R}_{++}$.

% \smallskip
% Assumption~\ref{2016-11-06:04} will be standing for the remaingin part of {\color{rose}SEZIONE}.
%   \begin{assumption}\label{2016-11-06:04}
%       \begin{equation}
%     \label{2016-11-06:02}
%         \E\left[\int_0^{+\infty} e^{-\rho t} f'(X_t^{1
% ,\emptyset})X_t^{1,\emptyset}dt\right]<c_0.
%   \end{equation}
% \end{assumption}

\begin{proposition}\label{prop:preB}
% Let
% $C_\gamma\coloneqq  \left(\rho+\beta\gamma+\frac{1}{2}\gamma(1-\gamma)\sigma^2\right)^{-1}$ and  
% let 
% $\hat v$ be the function defined by
% \begin{equation}
%   \label{costgammaB}
% \hat v\colon \R_{++}\rightarrow \mathbb{R},\ x \mapsto   C_\gamma \frac{x^\gamma}{\gamma}.
% \end{equation}
% \begin{enumerate}[(i)]
% \item\label{2015-11-23:02B}
For all $\alpha\in \left(0,c_0\rho\right]$
% {\color{rose}(notare che la condizione per $\alpha$ si \`e indebolita rispetto alla precedente versione (che era scazzata dato che non sappiamo nemmeno gronwall))}
 % {\color{rose}($0$ incluso??? Pham scazza?)},
 we have
%If  $\tilde{f}$ denotes the Fenchel-Legendre transform of $f$ on $\mathbb{R}_{++}$ {\color{cyan}(MR: cos\`i la chiama Pham)} (\footnote{That is, $f^*(x)\coloneqq \sup_{y> 0} \left\{ f(y)-xy \right\} $, $x\in \mathbb{R}_{++}$.}), then
\begin{equation}
   \label{estimateB}
0\leq
\hat{ v}(x)\leq   v(x)\leq      \frac{f^*(\alpha)}{\rho}
     +
     \frac{\alpha x}{\rho}
%0   \leq v(x) \leq \frac{\tilde{f}((\rho+\beta)\mu)}{\rho} + \mu x
, \qquad \forall x\in \mathbb{R}_{++}
\end{equation}
and
\begin{equation}\label{limsup}
\limsup_{x\rightarrow \infty}\frac{v(x)}{x}=0.
\end{equation}
 \end{proposition}
 \begin{proof}
The fact  that $0\leq \hat{ v}\leq v$ was already noticed in \eqref{eq:2017-04-07:03}.
We show the remaining inequality.
% {\red{SF: CONTROLLARE BENE}}{\color{rose}SISTEMATA: dacci un occhio anche tu}
Let $x\in\R_{++}$ and $I\in\mathcal{I}$.
For $R>0$, define the stopping time
$
\hat\tau_R\coloneqq
 \inf\left\{t\geq 0\colon 
X^{x,I}_t\geq  R\right\}.
$ 
Notice that, since $b\in C^1(\R_{++};\R)$ and $b(0)=0$ by Assumption \ref{eq:2017-03-23:00},  mean value theorem yields \begin{equation}\label{qwe}
b(\xi)\leq b(0)+M_b \xi=M_b\xi, \ \ \ \ \forall \xi\in\R,
\end{equation}
where $M_b$ is defined in \eqref{MB}.
Set $\tau_0\coloneqq 0$ and let $t\in\R_{++}$.
Applying It\^o's formula 
 to $\varphi(s,X^{x,I}_s)\coloneqq e^{-\rho s}X^{x,I}_s$,
 $s\in [0,\hat\tau_R)$,
% on the interval
%  $[\tau_{n}\wedge \tau_{R},t\wedge \tau_{n+1}\wedge \hat\tau_R)$,
 taking expectations
after  considering that $X_s^{x,I}\in(0,R)$ for  $s\in[0,\hat\tau_R)$,  summing up over $n\in\N$, and using \eqref{qwe} 
% Assumptions~\ref{eq:2017-03-23:00}
 % and the definition of $L_b$
and \ref{ass:rho}, we get 
\begin{equation*}
  \begin{split}
      \mathbb{E}\left[e^{-\rho t} X^{x,I}_{t\wedge \hat\tau_R}\right]=
      &x- \rho
\int_0^t 
 e^{-\rho s} \mathbb{E} \left[ \mathbf{1}_{[0, \hat\tau_R]}(s) X_s^{x,I}\right]ds+\int_0^te^{-\rho s}
\mathbb{E} \left[ \mathbf{1}_{[0, \hat\tau_R]}(s)
 % -\rho X^{x,I}_s+
 b(X^{x,I}_s)  \right] ds
  +
  e^{-\rho t}\mathbb{E} \left[ 
    \sum_{n\geq 1, \,\tau_n\leq t\wedge \hat\tau_R}
%    \mathbf{1}_{[0,t\wedge \hat\tau_R]}(\tau_n)
i_n
     \right] \\
\leq&
x+(M_b-\rho)\int_0^t  e^{-\rho s}
\mathbb{E} \left[
\mathbf{1}_{[0, \hat\tau_R]}(s) X^{x,I}_s\right] ds
  +  e^{-\rho t}
\mathbb{E} \left[ \sum_{n\geq 1, \, \tau_n\leq t\wedge \hat\tau_R}
i_n
     \right]\\
\leq&
x
  +  e^{-\rho t}
\mathbb{E} \left[ \sum_{n\geq 1, \, \tau_n\leq t\wedge \hat\tau_R}
i_n
     \right].
   \end{split}
 \end{equation*}
By Fatou's lemma,  letting $R\rightarrow \infty$ and observing that $\tau_{R} \rightarrow \infty$ $\P$-a.s. ,  we get
\begin{equation}
  \label{eq:2017-09-18:00}
        \mathbb{E}\left[e^{-\rho t} X^{x,I}_t\right]\leq 
x
  +  e^{-\rho t}
\mathbb{E} \left[ \sum_{n\geq 1, \, \tau_n\leq t}
i_n
     \right].
\end{equation}
By integrating the second term on the right-hand side of
\eqref{eq:2017-09-18:00}, we have using Fubini-Tonelli's Theorem (as all the integrands involved are nonnegative)
\begin{equation}\label{2017-09-18:01}
    \mathbb{E} \left[ \int_0^\infty  \left( e^{-\rho t} \sum_{n\geq 1,\,
        \tau_n\leq t} 
      i_n  \right) dt
    \right]=
    \mathbb{E} \left[ 
      \sum_{n\geq 1}
      \left( 
      \int_{\tau_n}^\infty e^{-\rho (t-\tau_n)}
      dt
       \right)
       e^{-\rho \tau_n}
      i_n
    \right]
    =
      \frac{1}{\rho}
    \mathbb{E}
    \left[ 
      \sum_{n\geq 1}
e^{-\rho \tau_n }i_n
    \right] .
%     \leq
%       \frac{1}{\rho}
%     \mathbb{E}
%     \left[ 
%       \sum_{n\geq 1} 
% e^{-\rho \tau_n }i_n
%     \right].
\end{equation}
% \begin{equation*}
%       \mathbb{E}\left[e^{-\rho t} X^{x,I}_{t\wedge  \hat\tau_R}\right]\leq
%        \left( 
%          x
%          +
%          \mathbb{E} \left[ \sum_{n\geq 1,\, \tau_n\leq t}
%     e^{-\rho \tau_n} i_n
%      \right]
%      \right) e^{L_bt} \ \ \ \ \forall t\in\R_{++}.
% \end{equation*}
%       \mathbb{E}\left[e^{-\rho t}X^{x,I}_{t}\right]\leq
%        \left( 
%          x
%          +
%          \mathbb{E} \left[ \sum_{n\geq 1,\, \tau_n\leq t}
%   e^{-\rho \tau_n}i_n
%      \right]
%      \right) e^{L_bt}\qquad \forall t\in\R_{++}.
% \end{equation}
Therefore, taking into account
\eqref{eq:2017-09-18:00}, \eqref{2017-09-18:01} and
 % Assumption \ref{ass:rho} and 
 \eqref{eq:2015-12-18:01}, we have
\begin{equation}
  \label{eq:2017-03-25:03}
      \mathbb{E}\left[\int_0^\infty e^{-\rho t}X^{x,I}_t dt\right]\leq
\frac{1}{\rho}       \left( 
         x
         +
         \mathbb{E} \left[ \sum_{n\geq 1}
    e^{-\rho \tau_n}i_n
     \right]
     \right) <\infty.
\end{equation}
Now let $\alpha>0$. By definition of $f^*$
and by
\eqref{eq:2017-03-25:03}, we can write
\begin{equation*}
  \begin{split}
    \mathbb{E}
     \left[ \int_0^\infty e^{-\rho t}f(X^{x,I}_t)dt
       -
       \sum_{n\geq 1}e^{-\rho \tau_n}(c_0i_n+c_1) \right] 
   &  \leq
    \mathbb{E}
     \left[ \int_0^\infty e^{-\rho t}
      \left(   f^* (\alpha)+\alpha 
       X^{x,I}_t \right) 
       dt
       -
       \sum_{n\geq 1}e^{-\rho \tau_n}(c_0i_n+c_1) \right]      \\
     &
     \leq
     \frac{f^*(\alpha)}{\rho}
     +
     \frac{\alpha x}{\rho}
     + \left( \frac{\alpha}{\rho}-c_0 \right) 
         \mathbb{E} \left[ \sum_{n\geq 1}
    e^{-\rho \tau_n}i_n
     \right].
  \end{split}
\end{equation*}
By arbitrariness of $I\in\mathcal{I}$, if $\alpha\in \left(0,c_0\rho\right]$,  the latter provides the last inequality in \eqref{estimateB}.

Take now $\alpha\in (0,c_0\rho]$. By \eqref{estimateB} we have
$$
0\leq {\displaystyle{\limsup_{x\rightarrow \infty}\frac{v(x)}{x}}}\leq \alpha\, {\displaystyle{\limsup_{x\rightarrow \infty}\frac{v(x)}{\alpha x}\leq  \alpha \,\limsup_{x\to\infty}\left\{\frac{f^*(\alpha)}{\alpha \rho x}
     +
     \frac{1}{\rho}\right\}=\frac{\alpha}{\rho}}}
$$
By arbitrariness of $\alpha$ we get \eqref{limsup}.
\end{proof}

\begin{assumption}\label{ass:ass3}
% {\color{red}da sistemare con assunzioni opportune di integrabilit\`a}
%  $f|_{[\delta,\infty)}$ is semiconvex for each $\delta>0$;   that is, for each $\delta>0$
The following conditions hold true.
\begin{enumerate}[(i)]
%\item 
%$K\geq f'$ on $\mathbb{R}_{++}$;
  \item\label{2017-09-25:01} $\rho>\max\big\{B_0, C_0\big\}$ where $B_0,C_0$ are the constants defined in Lemma~\ref{2017-09-27:01}. % QQQQQQ \ref{lemma:a}.
% {\color{rose}CHECK dopo la correzione del lemma in appendice}
% { \red{SF: FATTO; CONTROLLA}}
% {\red{SF: forse intendi $q$? RIvedere la condizione che segue}}{\color{rose} [S\`i, intendo $q$, o meglio, il $q$ che fa funzionare la \eqref{eq:2017-09-15:01} (e che dipende da $\beta$) deve essere tale che il suo coniugato $p$ soddisfa alla condizione che segue]} as in 
% Assumption~\ref{ass:ass3},
% $d_p$ as in Lemma~\ref{2017-09-15:00}
%\begin{equation}
%  \label{eq:2017-09-20:00}
%  \rho>{\color{rose}B(2)\vee C(4) \mbox{ del Lemma~\ref{2017-09-27:01}: da calcolare esplicitamente}}
%% 3\ln 5+5L^2_b+40L_\sigma^2+ 5^3 L^4_b+
%% \frac{5^34^{11}}{3^6} L_\sigma^4.
%\end{equation}
%$d_2=5(L^2_b+8L_\sigma^2)$
%$d_4=5^3(L^4_b+4^{11}/3^6 L_\sigma^4)$
\item\label{all} For each $\beta>0$,
 % there exists $q\in[2,+\infty)$ such that {\color{rose}Perch\'e hai tolto $q$ da $M(\beta,q)$??}
  \begin{equation}
    \label{eq:2017-09-15:01}
  \  M(\beta)\coloneqq   
  \mathbb{E}
    \left[
      \int_0^\infty
      e^{-\rho t}
      \left( 
        f'(X_t^{\beta,\emptyset}) \right) ^2
      dt
    \right]     
  <\infty.
\end{equation}
\item\label{all2}
For each $\eta>0$,
the function  $f$ is semiconvex on $[\eta, \infty)$.
Precisely, there exists a nonincreasing function $K_0\colon \mathbb{R}_{++}\rightarrow \mathbb{R}_{++}$ such that
  \begin{equation}
    \label{eq:2017-09-15:02}
    f(\lambda x+(1-\lambda)y)- \lambda f(x)-(1-\lambda)f(y)
    \leq K_0(\eta) \lambda(1-\lambda) (y-x)^2, \ \ \ \ \forall \lambda\in[0,1], \ \forall x,y\in[\beta,\infty).
  \end{equation}
\item\label{2017-09-25:00} The function $K_0$ in  (\ref{all2}) is such that, for each $\beta>0$,
  \begin{equation}
    \label{eq:2017-09-15:01B}
   \hat{ M}(\beta)\coloneqq   
    \mathbb{E}
    \left[
      \int_0^\infty
      e^{-\rho t}
      \left( 
        K_0(X_t^{\beta,\emptyset}) \right) ^2
      dt
    \right]      
  <\infty.
\end{equation}
\end{enumerate}

\end{assumption}

\begin{Remark}\label{rem:semi}
Semiconvex functions are functions that can be written as difference of a convex function and  a quadratic one (see \cite[Prop.\,1.1.3]{CS} or \cite[Ch.\,4, Sec.\,4.2]{YongZhou}). Moreover, a  function $\varphi\in C^2([\beta,\infty);\R)$ verifies \eqref{eq:2017-09-15:02} with $K_0(\eta):= -2\,\inf_{[\eta,\infty)} f''$ (see again \cite[Prop.\,1.1.3]{CS}).
\end{Remark}
The following Proposition shows that power functions satisfy Assumption \ref{ass:ass3}(ii)--(iv). 
\begin{proposition}\label{prop:ass} Let
$f\in C^2(\R_{++};\R)$ such that $f'>0$, $f''<0$, and  
\begin{equation}\label{aqw}
f'(\xi)\leq C_0(1+|\xi|^{\gamma-1}), \ \ \  f''(\xi)\geq -C_0(1+|\xi|^{\gamma-2}) \ \ \ \forall \xi\in\R_{++}
\end{equation}
for some $C_0> 0$ and $\gamma\in(0,1)$,
%Let $f(x)=x^\gamma$, $\gamma\in(0,1)$, 
and let $\rho>L_b(1-\gamma)+\frac{1}{2}L_\sigma^2(1-\gamma)(2-\gamma)$. Then $f$ satisfies Assumptions \ref{ass:ass3}(ii)--(iv).
\end{proposition}
\begin{proof}
 Let $\beta\in\R_{++}$ and observe that, by Assumption \ref{eq:2017-03-23:00}, we have
 $$
 |b(\xi)|\leq L_b|\xi|, \ \ \ |\sigma(\xi)|\leq L_\sigma|\xi| \ \ \ \ \ \forall \xi\in\R.
 $$. With a localization procedure similar to the one of the prof of Proposition \ref{prop:preB} (now keeping the process $X^{\beta,\emptyset}$ away from $0$), we get from It\^o's formula
\begin{eqnarray*}
&&\E\left[e^{-\rho t}\big|X^{\beta,\emptyset}_t\big|^{\gamma-1}\right]\\&&= |\beta|^{\gamma-1}+\E\left[\int_0^t e^{-\rho s} \left[-\rho \big|X^{\beta,\emptyset}_s\big|^{\gamma-1}+(\gamma-1)\big|X^{\beta,\emptyset}_s\big|^{\gamma-2} b(X_s^{\beta,\emptyset}) +
 \frac{1}{2} (\gamma-1)(\gamma-2)\big|X^{\beta,\emptyset}_s\big|^{\gamma-3} \sigma^2(X_s^{\beta,\emptyset})
\right] ds\right]\\
&&\leq |\beta|^{\gamma-1}+\E\left[\int_0^t e^{-\rho s}\left[-\rho  \big|X^{\beta,\emptyset}_s\big|^{\gamma-1}+{L}_b(1-\gamma)\,\big|X^{\beta,\emptyset}_s\big|^{\gamma-1} +
 \frac{1}{2} {L}^2_\sigma (1-\gamma)(2-\gamma)\big|X^{\beta,\emptyset}_s\big|^{\gamma-1} 
\right] ds\right].
\end{eqnarray*}
Then Assumption \ref{ass:ass3}(ii) follows from \eqref{aqw} and  Gronwall's Lemma applied to the inequality above. 

Moreover, note that, since $\xi\mapsto -C_0(1+|\xi|^{\gamma-2})$ is negative and increasing, by Remark \ref{rem:semi} and \eqref{aqw} we obtain that $f$ verifies Assumption \ref{ass:ass3}(iii) with 
\begin{equation}\label{alal}
K_0(\eta):=-2\gamma(\gamma-1)\eta^{\gamma-2} \ \ \ \forall x\in\R_{++}.
\end{equation}

Finally, similarly as above,
 we have
\begin{eqnarray*}
&&\E\left[e^{-\rho t}\big|X^{\beta,\emptyset}_t\big|^{\gamma-2}\right]\\&&= |\beta|^{\gamma-2}+\E\left[\int_0^t e^{-\rho s} \left[-\rho \big|X^{\beta,\emptyset}_s\big|^{\gamma-2}+(\gamma-2)\big|X^{\beta,\emptyset}_s\big|^{\gamma-3} b(X_s^{\beta,\emptyset}) +
 \frac{1}{2} (\gamma-2)(\gamma-3)\big|X^{\beta,\emptyset}_s\big|^{\gamma-4} \sigma^2(X_s^{\beta,\emptyset})
\right] ds\right]\\
&&\leq |\beta|^{\gamma-2}+\E\left[\int_0^t e^{-\rho s}\left[-\rho  \big|X^{\beta,\emptyset}_s\big|^{\gamma-2}+{L}_b(1-\gamma)\,\big|X^{\beta,\emptyset}_s\big|^{\gamma-1} +
 \frac{1}{2} {L}_\sigma^2 (1-\gamma)(2-\gamma)\big|X^{\beta,\emptyset}_s\big|^{\gamma-1} 
\right] ds\right].
\end{eqnarray*}
Then Assumption \ref{ass:ass3}(iv) follows  from Gronwall's Lemma applied to the inequality above and from  \eqref{alal}. 
\end{proof}
\begin{Remark}\label{rem:ass}
Note that, if $\rho$ satisfies Assumption \ref{ass:ass3}(i), then it also satisfies  the requirement of Proposition \ref{prop:ass}.
\end{Remark}
\begin{proposition}\label{prop:semiconvex} Let Assumption \ref{ass:ass3} hold.
Then  $v$ is semiconvex on $[\beta,\infty)$ for each $\beta>0$, i.e., for each $\beta >0$  there exists $K_1(\beta)>0$ such that
    \begin{equation}
      \label{eq:2017-09-15:05}
      v(\lambda x+(1-\lambda)y)-\lambda v(x)-(1-\lambda) v(y)\leq K_1(\beta) \lambda(1-\lambda)(x-y)^2\qquad  \forall  \lambda\in [0,1],\ 
\forall x,y\in [\beta,\infty).
    \end{equation}
\end{proposition}
\begin{proof}
Fix $\beta>0$. Let
$x,y\in[\beta, \infty)$ with $x\leq y$, and $I\in \mathcal{I}$. For each  $\lambda\in [0,1]$ set
  $z_\lambda \coloneqq \lambda x+(1-\lambda)y$
%Let $\delta>0$ 
%and let
%
%satisfying
%\eqref{eq:2017-09-17:06}.
%be a $\delta$-optimal control for $v(z)$.
and $\Sigma^{\lambda,x,y,I}\coloneqq \lambda X^{x,I}+(1-\lambda)X^{y,I}$. We write
\begin{multline*}
J(z_\lambda,I)
-\lambda J(x,I)
-(1-\lambda)J(y,I)=
    \mathbb{E} \left[ \int_0^\infty
      e^{-\rho t}
       \left( 
         f(X^{z_\lambda,I}_t)
         -\lambda         f(X^{x,I}_t)
         -(1-\lambda)         f(X^{y,I}_t)
       \right) dt
    \right] =\\
  \begin{split}
=    \mathbb{E} \left[ \int_0^\infty
      e^{-\rho t}
       \left( 
         f(X^{z_\lambda,I}_t)
         -         f(\Sigma^{\lambda,x,y,I}_t)
       \right) dt
    \right]   
    &+   \mathbb{E} \left[ \int_0^\infty
      e^{-\rho t}
       \left( 
         f(\Sigma^{\lambda,x,y,I}_t)
         -\lambda         f(X^{x,I}_t)
         -(1-\lambda)         f(X^{y,I}_t)
       \right) dt
    \right].   \\
    % \hskip2.43cm%\mathbf{A}
    {
   \rotatebox[origin=c]{90}{$\coloneqq $}
   \atop\mathbf{A}
   }
\hskip2.43cm
&
\phantom{+}
% {\phantom{ \rotatebox[origin=c]{90}{$\coloneqq $}}
%     \atop
%     +}
\hskip4cm   
{
   \rotatebox[origin=c]{90}{$\coloneqq $}
   \atop\mathbf{B}
   }
  \end{split}
\end{multline*}
\noindent%  Take $q$ as in 
% Assumption~\ref{ass:ass3}(i).
  Applying H\"older's inequality,
%  and then 
% Minkowski's inequality for integrals (\cite[6.19]{Folland1999}), 
observing that $X^{\beta,\emptyset}\leq X^{z_\lambda,I}\wedge \Sigma^{\lambda,x,y,I}$,
 that $f'$ is decreasing,
 using 
Assumption~\ref{ass:ass3}(\ref{2017-09-25:01}),
and using
Lemma~\ref{2017-09-27:01}\emph{(\ref{2017-11-13:01})},
% {2017-11-13:00})
 we write
%{\color{red}$f'$ descrescente, oppure posso prendere $K$ al posto di $f'$}
\begin{equation*}
  \begin{split}
    \mathbf{A} &\leq \mathbb{E} \left[ \int_0^\infty e^{-\rho t}
      f'(X^{\beta,\emptyset}_t) \left| X^{z_\lambda,I}_t - \Sigma^{\lambda,x,y,I}_t \right|dt
    \right]\\
    & \leq
    % \rho^\frac{1-p}{p}
    \left( \mathbb{E} \left[ \int_0^\infty e^{-\rho t}
          \left(  f'(X^{\beta,\emptyset}_t)  \right)^2 dt  \right] \right)
    ^{1/2} \left( \mathbb{E} \left[ 
\int_0^\infty e^{-\rho t}
          \left| X^{z_\lambda,I}_t -\Sigma^{\lambda,x,y,I}_t \right|^2 dt  \right]
    \right) ^{1/2}\\
    &\leq
    % \rho^\frac{1-p}{p}
    M(\beta)^{1/2}
    % \left( \mathbb{E} \left[ \left( \int_0^\infty e^{-\rho t}
    %       f'(X^{\beta,\emptyset}_t) dt \right) ^q \right] \right)
    % ^{1/q}
    \left( \mathbb{E} \left[
        \int_0^\infty e^{-\rho t}
     \left|
          X^{z_\lambda,I}_t -\Sigma^{\lambda,x,y,I}_t \right|^2 
    dt\right] \right) ^{1/2}\\
    &\leq M(\beta)^{1/2}
     \left( \int_0^\infty e^{-\rho t}
       A_0e^{B_0t}
       dt \right) ^{1/2} \lambda(1-\lambda)|x-y|^2 \\
&=
\frac{A_0^{1/2}M(\beta)}{(\rho-B_0)^{1/2}} \lambda(1-\lambda)|x-y|^2.
  \end{split}
\end{equation*}
%where $A_0,B_0$ are the constants defined in Lemma~\ref{2017-09-27:01}.
% \ref{lemma:a}.
%where $M'(\beta)$ is a constant
% independent of $\lambda, x,y, I$,
%and $A(2,\lambda),B(2)$ are as in
%Lemma~\ref{2017-09-27:01}.
Moreover,
by Assumption~\ref{ass:ass3}(\ref{2017-09-25:01}),(\ref{all2}),(\ref{2017-09-25:00}), 
again using H\"older's inequality
and applying
%\eqref{2017-09-27:15} of
Lemma~\ref{2017-09-27:01}\emph{(\ref{2017-11-13:00})},
%with 
% $\xi=x$, $\xi'=y$, 
%  $p=4$,
we have 
\begin{equation*}
  \begin{split}
    \mathbf{B} &\leq \lambda(1-\lambda) \mathbb{E} \left[
      \int_0^\infty e^{-\rho t} K_0(X^{\beta,\emptyset}_t) \left| X_t^{y,I} -
        X_t^{x,I} \right|^2 dt
    \right] \\
    &\leq \lambda(1-\lambda) \left( \mathbb{E} \left[
          \int_0^\infty e^{-\rho t}
 \left( K_0(X^{\beta,\emptyset}_t)
  \right) ^2 dt
       \right] \right) ^{1/2} \left( \mathbb{E} \left[
        \int_0^\infty e^{-\rho t} \left| X^{y,I}_t - X^{x,I}_t
        \right|^4 dt \right]
    \right)^{1/2}\\
    &\leq
    % \rho^{\frac{1-p}{p}}
     \lambda(1-\lambda)
% \left( \mathbb{E} \left[\int_0^\infty e^{-\rho
    %       t} \left( K(X^{\beta,\emptyset}_t) \right) ^q dt \right]
    % \right) ^{1/q}
\hat{M}(\beta)^{1/2}
    \left(
 \int_0^\infty e^{-\rho t}
e^{C_0t}
  dt
    \right)^{1/2}|x-y|^2
\\
    &=
\frac{    \hat{ M}(\beta)^{1/2}}{(\rho-C_0)^{1/2}}\lambda(1-\lambda)|x-y|^2.
  \end{split}
\end{equation*}
Now let
$\delta>0$ and let $I$ be  
 such that $v(z)\leq J(z,I)+\delta$.
%Such a control exists by
%Corollary~\ref{2017-09-24:00}.
% % {\color{rose}(lemmino sul fatto che un tale $\delta$-optimal control esiste)}{\red{SF: Forse conviene intoridurre la classe $\cal{S}$ dei controlli che soddisfano \eqref{eq:2017-09-17:06} e dimostrare che $v=\sup_{I\in \cal{S}} J$}}.
The inequalities above provide
\begin{equation*}
 \begin{split} v(z)-\delta 
-\lambda v(x)-(1-\lambda)v(y)&\leq J(z,I) -\lambda J(x,I)-(1-\lambda)J(y,I)\\
&\leq  K_1(\beta)\lambda(1-\lambda)|x-y|^2\qquad \forall x,y\geq \beta, \ \forall \lambda\in [0,1],
\end{split}
\end{equation*}
where $K_1(\beta)\coloneqq \frac{    \hat{ M}(\beta)}{(\rho-C_0)^{1/2}}+\frac{   A_0^{1/2} \hat{ M}(\beta)}{(\rho-B_0)^{1/2}}$.
We then obtain
\eqref{eq:2017-09-15:05} by arbitrariness of $\delta$.
\end{proof}

\noindent  In view of the fact that the results which follow rely on the semiconvexity of $v$, Assumption~\ref{ass:ass3}
% and \ref{2017-09-15:04}
will be standing for the remaining of this section and in Sections \ref{sec:DPE}, \ref{sec:exp}, \ref{sec:opt}.
\bigskip

%{\red{SF: RICORDATI DI SISTEMARE LO SPAZIO CHE SEGUE: Lip su $\beta,\infty$ per ogni $\beta>0$? IN TAL CASO OCCORRE MOSTRARE LA PROPRIETA' FINO A $infty$. ALTRIMENTI LOCALLY LIP IN $\R_{++}$? ALLORA LA DIM. DI LEMMA 4.1 VA SISTEMATA}}

Define the  space
\begin{equation}\label{defL}
 \lipcz{\mathbb{R}_{++}}\coloneqq  \left\{
u\colon \mathbb{R}_{++}\rightarrow\mathbb{R}\ 
\mbox{locally Lipschitz continuous on } \R_{++}, \ 
%\in UC(\mathbb{R}_{++})\colon
%\sup_{x\in \mathbb{R}_{++}}\frac{|u(x)|}{c_0x}<\infty
%\sup_{x\in \mathbb{R}_{++}}
\mbox{s.t.}\
\limsup_{x\rightarrow \infty} \frac{u(x)}{x}<c_0
  \right\}. 
\end{equation}
We recall that semiconvex functions on open sets
%  can be written as difference of  a convex and a function of class $C^1$. In particular,
% semiconvex functions
are locally Lipschitz.
So, by Propositions~\ref{prop:preB} and \ref{prop:semiconvex}, we have 
$v\in \lipcz{\mathbb{R}_{++}}$.
The  space $ \lipcz{\mathbb{R}_{++}}$
 will be used in the next section.

\section{Dynamic Programming}\label{sec:DPE}
 The dynamic programming equation associated to our dynamic optimization problem is the quasi-variational inequality (see, e.g., \cite{BeLio})
\begin{equation}\label{QVI}\tag{\bf QVI}
\min\big\{\mathcal{L}u-{f},\ u-\mathcal{M}u\big\}=0,
\end{equation}
% \begin{color}
%   {cyan}
%   Non ho proprio capito da dove si deduce (anche solo euristicamente) la \eqref{QVI}.
%   Infatti, da \eqref{ito1} e da \eqref{DPP1} si ha
% $$
% \mathcal{L}v-f\geq 0 \qquad \mathrm{or}\qquad v-\mathcal{M}v\geq 0
% $$
% e io quindi deduco solo
% $$
% \max   \left\{ \mathcal{L}v-f, 
% v-\mathcal{M}v \right\} \geq 0.
% $$
% \end{color}
where $\mathcal{L}$ and $\mathcal{M}$ are  operators 
formally
 defined by
\begin{equation}\label{A}
\mathcal{L}u(x)\coloneqq \rho u(x)-b(x) u'(x)-\frac{1}{2}\sigma^2(x)u''(x),\qquad x\in \R_{++},
%{\color{cyan}\qquad \forall u\in C^2(\R_{++}),\ \forall x\in \R_{++},}
\end{equation}
\begin{equation}\label{M}
\mathcal{M}u(x)\coloneqq  \sup_{i>0}\left\{u(x+i)-c_0i-c_1\right\},\qquad x\in \mathbb{R}_{++}.
%{\color{cyan}\qquad \forall u\in C(\R_{++}),\ \forall x\in \R_{++}.}
\end{equation}
%{\color{cyan}(controllare a posteriori che il dominio di $\mathcal{M}$ sia sufficientemente grosso.)}
We note that $\mathcal{L}$ is a differential operator, so it has a local nature, while $\mathcal{M}$ is a functional operator having a nonlocal nature.

% \begin{remark}{\color{rose}(toglierei questo remark)} Also the weak DPP approach \cite{BTsicon} is employable here (CHECK). However, due to the fact that we already know a priori that $v$ is continuous, the latter approach is not needed.
% \end{remark}

\subsection{Continuation and action region}

Here we define and study the first properties of the  \emph{continuation} and \emph{action} region in the state space $\R_{++}$.
\begin{lemma}\label{lemma:M}
$\calm$ maps $\lipcz{\mathbb{R}_{++}}$ into itself.
%{\color{rose}(volendo si pu\`o anche mettere le $u$ che sono $UC$ su ogni $[\epsilon,\infty)$ e che soddisfano quella condizione di crescita.  Ma tanto noi non diamo risultato di unicit\`a per la QVI, quindi direi di restringerci $\lipcz$)}
\end{lemma}
\begin{proof}
Let $u\in \lipcz{\mathbb{R}_{++}}$.
Then there exists $\overline x ,\epsilon>0$ such that
\begin{equation}\label{2016-11-06:03}
  \frac{u(x)}{x}-c_0\leq -\epsilon\qquad \forall x\geq \overline x .
\end{equation}
By \eqref{2016-11-06:03}, for all $i>0$, $x\geq \overline x$, we have
\begin{equation*}
  u(x+i)-(c_0i+c_1)=(x+i) \left( \frac{u(x+i)}{x+i}-c_0 \right) +c_0x-c_1\leq (c_0-\epsilon) x.
\end{equation*}
Hence, by taking the supremum over $i>0$,
\begin{equation*}
  \frac{\mathcal{M}u(x)}{x}\leq c_0-\epsilon\qquad\forall x\geq \overline x,
\end{equation*}
which shows that ${\displaystyle\limsup_{x\rightarrow \infty}\frac{\mathcal{M}u(x)}{x}< c_0}$.

% {\red{SF: RIMANEGGIATO, CONTROLLARE MR}:{\color{cyan} dove \`e stata rimaneggiata? mi sembra ok} }
Now we show that $\mathcal{M}u$ is Lipschitz continuous on $[M^{-1},M]$ for each $M>0$. 
Using \eqref{2016-11-06:03} one can show that
\begin{equation}\label{2017-09-24:01} 
  \limsup_{i\rightarrow +\infty}\sup_{x\in[M^{-1},M]}
 \big\{ u(x+i)-c_0i\big\}=-\infty.
\end{equation}
Set
$$
U(x)\coloneqq \sup\big\{i\in  \mathbb{R}_{++}: u(x+i)-c_0i\geq u(x)-1\big\} \ \ \forall x\in [M^{-1},M].
$$
The limit \eqref{2017-09-24:01} provides that there exists $R>0$ such that
\begin{equation*}
U(x)\leq R \ \ \ \forall x\in[M^{-1},M].
\end{equation*}
Hence,  we have
\begin{equation}
  \label{eq:2017-09-24:02}
  \mathcal{M}u(x)=\sup_{i\in (0,R]} \{u(x+i)-c_0i-c_1\}\qquad \forall x\in [M^{-1},M].
\end{equation}
% {\red{SF: DA SISTEMARE MAURO: ORA NON E' PIU' LIP ALL'INFINITO}}
Now let $\hat L$ be the Lipschitz constant of  $u|_{[M^{-1},M+R]}$.
%, where 
%$$
%R=\sup\big\{i\colon (x,i)\in U,\ \mbox{for some}\ x\in [M^{-1},M]\big\}.
%$$
Then, if $M^{-1}\leq x<y\leq M$, $0<i\leq R$,
  we can write
% Let $w$ be a modulus of continuity for $u\in UC_{c_0}(\mathbb{R}_{++})$. We have
  \begin{equation}
    \label{eq:2017-11-13:02}
    u(x+i)-(c_0i+c_1)- \hat L(y-x) \leq  u(y+i)-(c_0i+c_1)\leq  u(x+i)-(c_0i+c_1)+ \hat L(y-x).
%, \ \forall i>0, \ \forall \va.
  \end{equation}
  Now
the claim follows
by taking the supremum over $i\in (0,R]$
on \eqref{eq:2017-11-13:02} and recalling~\eqref{eq:2017-09-24:02}.
\end{proof}

By definition of $v$ we have
\begin{equation}\label{DPPPP}
v(x)\geq v(x+i)-c_0i-c_1 \ \ \ \forall i>0,
\end{equation}
hence
%\begin{equation}\label{DPP1}
%v(x)\geq \sup_{i>0}\left\{v(x+i)-c_0i-c_1\right\}.
%\end{equation}
% and use \eqref{defL} and Lemma \ref{lemma:M} 
%to get the inequality
\begin{equation}
  \label{disM}
v\geq \mathcal{M}v.
\end{equation}
We define the \textit{continuation region} $\mathcal{C}$ and the \textit{action region} $\mathcal{A}$ by %respectively, as follows:
\begin{align}
 \label{conregion}
&\mathcal{C}\coloneqq \big\{ x\in \R_{++}\colon  \ \mathcal{M}v(x)<v(x)\big\} &&\qquad \mbox{(continuation region)}\\[5pt]
  \label{actregion}
&\mathcal{A}\coloneqq \R_{++} \setminus \mathcal{C}= \big\{ x\in \R_{++}\colon  \ \mathcal{M}v(x)=v(x)\big\} &&\qquad \mbox{(action region)}.
\end{align}
They will represent, respectively, the region where it will be convenient to let the system evolve autonomously and the region where it wil be convenient to undertake an action by exercising an impulse.
By
Proposition~\ref{prop:preB} and Lemma~\ref{lemma:M}, both members of \eqref{disM} are finite continuous functions.
%Taking into account \eqref{disM}, we can split the state space $\R_{++}$ in two regions. 
In particular, 
$\mathcal{C}$ is open and $\mathcal{A}$ is closed in $\R_{++}$.
% \begin{proof}
% It follows from continuity of $v- \calm v$ (see Lemma \ref{lemma:M}).
% \end{proof}

\smallskip
% In what follows  we will 
% % that the set $\mathcal{A}$ may be empty {\color{cyan}NO: noi non ``mostriamo'' che ``pu\`o'' essere vuoto}, whereas
 For $x\in\mathcal{A}$, let us introduce the set
$$
\Xi(x)\coloneqq \mathop{\operatorname{argmax}}_{i>0}\,\big\{v(x+i)-c_0i-c_1\big\}.
$$
Clearly $\Xi(x)$ is empty if $x\in\mathcal{C}$. In principle $\Xi(x)$ might be empty even if $x\in\mathcal{A}$,  but this is not the case as shown by the following.
\begin{proposition}\label{propempty} Let $x\in \mathcal{A}$. 
  \begin{enumerate}[(i)]
\item\label{2016-11-06:06} %For every $x\in\mathcal{A}$, the set
$\Xi(x)$ is not empty. % (i.e.\ the infimum is in fact a minimum), 
\item\label{2016-11-06:07}  %For every $x\in\mathcal{A}$ and every
For all $\xi\in\Xi(x)$, we have
$x+\xi\in\mathcal{C}.$
\end{enumerate}
\end{proposition}
\begin{proof}
\noindent\emph{(\ref{2016-11-06:06})}
Let $x\in\cala$ and take a sequence $\{i_n\}_{n\in \mathbb{N}\setminus\{0\}}\subset \R_{++}$ such that 
\begin{equation}\label{dsa}
\calm v(x)\geq v(x+i_n)-c_0i_n-c_1\geq \calm v(x)-\frac{1}{n}, \ \ \ \forall n\in \mathbb{N}\setminus\{0\}.
\end{equation}
Then, considering  that  ${\displaystyle{\limsup_{i\rightarrow \infty}\frac{v(x+i)}{x+i}=0}}$ by
Proposition~\ref{prop:preB} and that $\mathcal{M}v(x)$ is finite,
% \eqref{estimateB},
we easily see, arguing by contradiction, that, in order to fulfill \eqref{dsa}, the sequence $\{i_n\}_{n\in\N}$ must be bounded. Hence, 
% up to 
by considering a subsequence if necessary,
% relabeled in the same way,
we have $i_n\rightarrow i^*\in\R_+$. Let us show that $i^*>0$. Indeed, assume by contradiction that $i^*=0$. By \eqref{dsa}, taking into account that
$v$ is continuous and that
 $v(x)=\calm v(x)$ as $x\in\mathcal{A}$, we obtain
$v(x)=\calm v(x)\leq v(x)-c_1$, a contradiction.
 Then we have shown that $i^*>0$. 
From \eqref{dsa} we obtain, by continuity,
$\calm v(x) = v(x+i^*)-c_0i^*-c_1$ and the claim follows.

\noindent\emph{(\ref{2016-11-06:07})}  This part of the proof closely follows the proof of \cite[Prop.~2]{GuoWu}. We omit it for brevity.
% Let $x\in\mathcal{A}$ and  $\xi\in\Xi(x)$. We have 
%\begin{equation}
%\begin{split}
%v(x+\xi)-c_0\xi-c_1= &\calm v(x)=\sup_{i>0} \big\{v(x+i)-c_0i-c_1\big\}\\
%=&\sup_{i>0} \big\{v(x+\xi+(i-\xi))-c_0(i-\xi)-c_1\big\}-c_0\xi\\
%\geq&\sup_{i>\xi} \big\{v(x+\xi+(i-\xi))-c_0(i-\xi)-c_1\big\}-c_0\xi\\
%= &\calm v(x+\xi)-c_0\xi.
%\end{split}
%\end{equation}
%Hence $v(x+\xi)>\calm v(x+\xi)$, i.e.\ $x+\xi\in\calc$.  
\end{proof}

Note that, as a consequence of Proposition \ref{propempty}, we have $\mathcal{C}\neq \emptyset$. Indeed, either $\mathcal{A}=\emptyset$, thus $\mathcal{C}=\R_{++}$; or $\mathcal{A}\neq \emptyset$, thus $\mathcal{C}\neq \emptyset$ by Proposition \ref{propempty}(\ref{2016-11-06:07}). Formally, Proposition \ref{propempty}(\ref{2016-11-06:07}) says that, if the system is in a position $x\in\mathcal{A}$:  (i) an optimal control exists (part (\ref{2016-11-06:06})); (ii) this optimal control places the system in $\mathcal{C}$ (part (\ref{2016-11-06:07})). We will verify this fact rigorously afterwards.

\subsection{Dynamic Programming Principle and viscosity solutions}

\smallskip
The rigorous connection between $v$ and \eqref{QVI} passes through the dynamic programming principle (DPP). 
\begin{proposition} \label{Prop:DPP}
For every $x>0$ and every $\F$-stopping time $\tau\in\overline{\R}_+$,
    \begin{equation}\label{2015-11-24:00}\tag{\bf DPP}
v(x) = \sup_{I\in\mathcal{I}} \E\left[\int_0^{\tau}e^{-\rho s} {f}(X^{x,I}_s)ds-\sum_{n\geq 1,\,\tau_n\leq  \tau} e^{-\rho\tau_n}(c_0i_n+c_1)+e^{-\rho \tau} v(X^{x,I}_\tau)\right].
    \end{equation}
\end{proposition}
%In order to deduce heuristically from \eqref{2015-11-24:00} an HJB variational inequlaity associated to $v$, we argue as follows.
\begin{proof}
We refer to \cite{ChenGuo} (for the finite horizon case; our formulation is the usual one for time homogeneous infinite horizon problems). 
%{\color{cyan}(controllare se altre citazioni needed)}
%{\color{rose}(Credo che dovremmo dare la proof, almeno per il verso in cui usiamo il DPP.)}
\end{proof}
%Since we do not have any information about the regularity of $v$ beyond H\"older continuity, we rely on the concept of viscosity solution to study \eqref{QVI}.

Here we study \eqref{QVI} by means of viscosity solutions.
\begin{definition}
  [Viscosity Solution]\label{vs}
Let $u\in \lipcz{\R_{++}}$.
\begin{enumerate}[(i)]
\item  $u$ is a viscosity subsolution to \eqref{QVI} if for every $(x_0,\varphi)\in \R_{++}\times C^2(\mathbb{R}_{++})$  such that $u-\varphi$ has a local maximum at $x_0$ and $u(x_0)=\varphi(x_0)$ we have
$$\min\big\{\mathcal{L}\varphi(x_0)-f(x_0),u(x_0)-\mathcal{M}u(x_0)\big\}\leq0;$$

\item $u$ is  a viscosity supersolution to \eqref{QVI} if for every $(x_0,\varphi)\in \R_{++}\times C^2(\mathbb{R}_{++})$  such that $u-\varphi$ has a local minimum at $x_0$ and $u(x_0)=\varphi(x_0)$ we have
$$\min\big\{\mathcal{L}\varphi(x_0)-f(x_0),u(x_0)-\mathcal{M}u(x_0)\big\}\geq 0;$$
\item $u$ is a viscosity solution to \eqref{QVI} if it is both a viscosity subsolution and
a viscosity  supersolution of \eqref{QVI}.
\end{enumerate}
\end{definition}

\begin{proposition}
  \label{viscosoluQVI}
 The value function $v$ is a viscosity solution of \eqref{QVI}.

\end{proposition}
\begin{proof}
%We only sketch the proof. For a rigorous proof (over a finite time horizon) we refer to \cite{ChenGuo}. 
%%{\color{cyan}(non mi pare solo uno sketch)}
%

\noindent {\emph{Supersolution property.}}  
 Let $x_0\in \mathbb{R}_{++}$ and $\varphi\in C^2(\mathbb{R}_{++})$ be  such that $v-\varphi$ has a local minimum at $x_0$ and $v(x_0)=\varphi(x_0)$.
In particular,  $v\geq \varphi$ 
on $(x_0-\delta,x_0+\delta)$
for a suitable $\delta\in(0,x_0)$.
By  \eqref{disM}  we only need to show that 
$\mathcal{L}\varphi(x_0)-f(x_0)\geq 0$.
To this aim, consider the stopping time 
$\tau\coloneqq\inf\left\{t\geq 0\colon 
|X^{x_0,\emptyset}_t-x_0|>\delta\right\}$, and note that $\P\{\tau>0\}=1$ by continuity of trajectories. Then, from \eqref{2015-11-24:00} we get 
%for every $\varepsilon>0$
\begin{equation}
  \label{eq:2015-11-24:01tris}
  v(x_0)\geq\mathbb{E}\left[\int_0^{\tau\wedge \varepsilon} e^{-\rho t} {f}(X_t^{x_0,\emptyset})dt +e^{-\rho(\tau\wedge \varepsilon)}v(X^{x_0,\emptyset}_{\tau \wedge\varepsilon})\right]
\qquad \forall \epsilon>0.
  \end{equation}
  From this we derive
\begin{equation}
  \label{eq:2015-11-24:01bis}
  \varphi(x_0)\geq\mathbb{E}\left[\int_0^{\tau\wedge \varepsilon} e^{-\rho t} {f}(X_t^{x_0,\emptyset})dt +e^{-\rho(\tau\wedge \varepsilon)}\varphi(X^{x_0,\emptyset}_{\tau \wedge \varepsilon})\right]
\qquad \forall \epsilon>0.
\end{equation}
By applying Dynkin's formula, dividing by $\epsilon$, letting $\varepsilon \rightarrow 0^+$, and considering that
$X^{x,\emptyset}$ is right-continuous in $0$ and 
 $\P\{\tau>\varepsilon\}\rightarrow 1$ as $\varepsilon\rightarrow 0^+$, we obtain the desired inequality. 

\smallskip
\noindent{\emph{Subsolution property.}}
 Let $x_0\in \mathbb{R}_{++}$ and $\varphi\in C^2(\mathbb{R}_{++})$ be  such that $v-\varphi$ has a local maximum at $x_0$ and $v(x_0)=\varphi(x_0)$. If $v(x_0)=\calm v(x_0)$, then we are done.
Then assume $v(x_0)\geq \xi+ \calm v(x_0)$ for some $\xi>0$.
In this case, we need to show that $
\mathcal{L}\varphi(x_0)-f(x_0)\leq 0$. 
Assume by contradiction
 that $
\mathcal{L}\varphi(x_0)-f(x_0)\geq \varepsilon>0$. By continuity of $\call\varphi-f$ and 
of $v-\calm v$,
% (Lemma \ref{lemma:M}), 
 and in view of the fact that $v-\varphi$ has a local maximum at $x_0$ and $\varphi(x_0)=v(x_0)$, there exists $\delta\in(0,x_0/2)$ such that 
\begin{equation}\label{eq:sub}
\forall x\in B(x_0,2\delta]\ \ \ \ \ \begin{dcases}
\mbox{(i)} & \mathcal{L}\varphi(x)-f(x)\geq \varepsilon/2\\
\mbox{(ii)} &
\varphi (x) \geq v(x) \\
\mbox{(iii)} &  v(x) - \calm  v(x) \geq \xi/2.
\end{dcases} 
\end{equation}
%Consider an $\varepsilon/3$-optimal control $I_\varepsilon\coloneqq \{(\tau^\varepsilon_n,i^\varepsilon_n)\}_{n\geq 0}\in\cali$.
Now define the stopping time 
$
\tau\coloneqq \inf\{t\geq 0\colon 
| X_t^{x_0,\emptyset}-x_0|>\delta
\} 
$
and note that  $\P\{\tau>0\}=1$.
In view of
\eqref{eq:sub}(iii),
 undertaking an 
  investment in the region $B(x_0,2\delta]$ is not optimal.
%{\color{cyan}(a me non \`e immediatamente ovvio, e credo che questo sia un punto da argomentare)}.
 Hence \eqref{2015-11-24:00} can be rewritten limiting the ranging of $I$ to the set of controls such that $\tau_1>\tau$, yielding the simple equality
%{\color{cyan}(ho messo il $2\delta$ altrimenti qui sotto bisogna metterci il costo fisso, oppure scrivere $\tau^-$, ma noi, formalmente, il DPP non lo diamo con $\tau^-$.)}
\begin{equation}\label{DPP2}
v(x_0)= \E\left[\int_0^{\tau}e^{-\rho t} f(X^{x_0,\emptyset}_t)dt +e^{-\rho \tau} v (X_{\tau}^{x_0,\emptyset}) \right].
\end{equation}
% Now, given $I=\{(\tau_n^I,i_n^I)\}_{n\in \mathbb{N}}\in\cali$ {\color{cyan}with $\tau_1\geq\tau$},
Finally, we have, 
by  \eqref{DPP2}, Dynkin's formula, and \eqref{eq:sub}(i)--(ii),
\begin{equation}
\begin{split}
\frac{\varepsilon}{2}\E\left[\tau \right]&\leq \E\left[\int_0^\tau e^{-\rho t}\left( \mathcal{L}\varphi(X_t^{x_0,\emptyset})-f(X_t^{x_0,\emptyset})\right)dt \right]\\
& =
\varphi(x_0)- \E\left[\int_0^{\tau} e^{-\rho t}f(X^{x_0,\emptyset}_t)dt +e^{-\rho \tau} \varphi (X_{\tau}^{x_0,\emptyset}) \right]\\
 &\leq  
v(x_0)- \E\left[\int_0^{\tau} e^{-\rho  t}f(X^{x_0,\emptyset}_t)dt +e^{-\rho \tau} v (X_{\tau}^{x_0,\emptyset}) \right]=0.
\end{split}
\end{equation} 
This provide a contradiction as $\P\left\{\tau>0\right\}=1$.
 \end{proof}
 \subsection{Regularity of the value function}

Here we establish the regularity properties  of the value function.
Precisely, exploiting the semiconvexity provided by Proposition \ref{prop:semiconvex} and the viscosity property provided by Proposition \ref{viscosoluQVI}, we show that it is of class $C^1$ on $\R_{++}$ and of class $C^2$ on $\mathcal{C}$.

\begin{theorem}
  \label{th:viscosoluQVI} 
$
v\in C^1(\R_{++};\R)\,\bigcap\, C^2(\mathcal{C};\R)$.
\end{theorem}
\begin{proof}
Let $x_0\in\R_{++}$. As $v$ is semiconvex in a neighborhood of $x_0$  (Proposition~\ref{prop:semiconvex}), in such a neighborhood it can be written as difference of a convex function and  a quadratic one (see Remark \ref{rem:semi}). Hence, the one-side derivatives 
 $v'_+(x_0), v'_-(x_0)$ exist and $v'_-(x_0)\leq v'_+(x_0)$. To show that $v$ is differentiable at $x_0$, we need to show that the previous inequality is indeed an equality. Assume, by contradiction, that  $v'_-(x_0)< v'_+(x_0)$. Then we can construct a sequence of functions $\{\varphi_n\}_{n\in\N} \subset C^2(\R_{++})$ such that, for every $n\in \N$, 
 \begin{equation*}
    \varphi_n(x_0)=v(x_0),
\qquad
 \varphi_n\leq v,
\qquad
\varphi'_n(x_0)=\frac{v'_-(x_0)+v'_+(x_0)}{2},
\qquad
\varphi_n''(x_0)\geq n.
\end{equation*}
 Then $\call \varphi_n(x_0)-f(x_0)\rightarrow -\infty$ as $n\rightarrow \infty$, which is impossible as $v$ is a viscosity supersolution to \eqref{QVI},
by Proposition~\ref{viscosoluQVI}.
% 
% $v$ is also a supersolution to \eqref{QVI} in the sense of distribution (see \cite{Ishii}). Hence we have, in the sense of distribution
% $
% \call v\geq f.
% $
% From this, and taking into account that $v$ is locally Lipschitz continuous, so $v'$ is locally bounded (in the sense of distribution),  we get that $v''$ is locally bounded (in the sense of distribution). Hence, by \cite{CannarsaSinestrari}, $v$ is locally  semiconcave.
Hence
it must be $v'_-(x_0)=v'_+(x_0)$.
By arbitrariness of $x_0$, 
this shows that
 $v$ is differentiable on $\R_{++}$. By semiconvexity we deduce that $v\in C^1(\R_{++})$ 
(see \cite[Theorem~25.5]{Rockafellar1970}).

The fact that $v\in C^2(\calc;\R)$ follows from a  standard localization argument: in each interval $(a,b)\subset \calc$ the function $v$ is a viscosity solution to the linear equation $\call u-f=0$ with boundary conditions $u(a)=v(a)$ and $u(b)=v(b)$. By uniform ellipticity of $\call$ over $(a,b)$ (see, e.g., \cite[Ch.\,6]{Evans}), this equation admits a unique solution in $C^2((a,b);\R)$, which  must also be  a viscosity solution. By uniqueness of viscosity solutions  to the linear equation above with Dirichlet boundary conditions, we conclude that $v$ coincide with the classical solution, hence $v\in C^2((a,b);\R)$. As $\mathcal{C}$ is open, the claim follows by arbitrariness of $(a,b)$ .
% Let $x_0\in\cala$. By step 1, we know that $v'_+(x_0), v'_-(x_0)$ exist at each $x_0$ and $v'_+(x_0)\leq v'_-(x_0)$. To show that $v$ is differentiable at $x_0$, we need to show that the latter inequality is indeed an equality. Assume, by contradiction, that  $v'_+(x_0)< v'_-(x_0)$. Then we can construct a sequence of functions $(\varphi_n)_{n\in\N} \subset C^2(\R_{++})$ such that for every $n\in \N$  
% \begin{enumerate}
% \item $\varphi_n(x_0)=v(x_0)$;
% \item $\varphi_n\geq v$;
% \item $|\varphi_n|+|\varphi'_n|$ is locally bounded uniformly in $n\in \N$;
% \item $\varphi_\varepsilon''(x_0)\leq -n$;
% \end{enumerate}  
% Now, since $x_0\in $
\end{proof}

%\textbf{Proof.} SF: CITARE GLI ARGOMENTI DI GUO E WU E TAGLIARE CORTO
%\begin{remark}
%SF: DIRE CHE TH. \ref{viscosoluQVI} E' VERO PER TUTTE LE SOLUZIONI DI VISCOSITA' CON CERTE OPPORTUNE PROPRIETA'. MOLTO IMPORANTE PER DEDURRE L'UNICITA" A POSTERIORI IN TALE CLASSE
%\end{remark}

% \begin{color}
%   {cyan}ARRIVATO QUI
% \end{color}

\begin{corollary}\label{Lemma1}
We have
\begin{enumerate}[(i)]
%\item\label{v'leq}
%$v'(x)\leq c_0$ for every $x\in \R_{++}$.
%\item \label{v'eq}
% If $v'\equiv c_0$ in an interval $[a,b]\subset \R_{++}$ with $a<b$, then $[a,b]\subseteq \mathcal{A}$. 
\item
\label{ffa}
$v'(x+\zeta)=c_0$, for every $x\in\mathcal{A}, \ \forall \zeta\in\Xi(x)$.
\item 
\label{derv}
$v'(x)
=c_0$, for every $x\in\mathcal{A}.$
\end{enumerate}
\end{corollary}
\begin{proof}
The proof is the same as in \cite[Lemma.~5.2]{GuoWu} \red{and we skip it for the sake of brevity.}
\end{proof}

Corollary \ref{Lemma1}(\ref{ffa}) will be used in the next section to characterize the optimal target point, i.e.\ the point in the continuation region where it is optimal to place the system when it reaches the action region.

%\textit{\textbf{Proof of Theorem \ref{viscosoluQVI}}}\, The argument of this proof are quite standard. We refer the reader to  \cite[Th.3.2, Lemmma 4.1, Th. 4.2]{GuoWu}. 

%\emph{Supersolution Property}.  
% (A chi di dovere: commentare in riferimento all'articolo di Guo e Wu.)

\section{Explicit expression of the value function} \label{sec:exp}

In
 this section we characterize $\mathcal{C},\mathcal{A}$, and $v$ up to the decreasing solution of the homogeneous ODE $\mathcal{L}=0$ and to the solution of a nonlinear system of three algebraic equations.

\begin{lemma}\label{lemma:a}
 $\mathcal{A}$ does not contain any interval of the form $[a,\infty)$, with $a>0$. In particular
$\mathcal{C}\neq \emptyset$.
\end{lemma}
\begin{proof}
   Assume, by contradiction, that there exists  $a>0$ such that $\mathcal{A}\supset [a,\infty)$. Then,  due to
Lemma~\ref{Lemma1}\emph{(\ref{derv})}, we  have
$$v(x)=\red{c_0(x-a)}+v(a), \ \ \ \forall x\geq  a,$$
which contradicts Proposition~\ref{prop:preB}.
On the other hand  we should also  have
$$v(x)=\mathcal{M}v(x), \ \ \forall x\geq a.$$
So it must be
$$\red{c_0(x-a)}+v(a)=\sup_{i>0}\big\{\red{c_0(x+i-a)}+v(a)-c_0i-c_1\big\} \ \ \ \forall x\geq a,$$
which is impossible as $c_1>0$. 
\end{proof}
The following assumption ensures that the action region is an interval.
\begin{Assumption}\label{ass:bb}
$b|_{\R_+}$ is concave.
\end{Assumption}

\begin{lemma}\label{2017-04-04:00}
 Let Assumption \ref{ass:bb} hold. Then $\mathcal{A}$ is an interval.
\end{lemma}

% We will exploit heavily the knowledge of the structure of the value function on $\mathcal{C}$.
% We start by establishing the structure of the set %$\mathcal{C}$ and $\mathcal{A}$.
\begin{proof}
Since $\mathcal{A}$ is closed,
it is sufficient to
show that there do not exist points $x_0,x_1\in \R_{++}$, with $x_0<x_1$,  such that 
$x_0,x_1 \in \mathcal{A}$ and $(x_0,x_1)\subset \mathcal{C}.$
Arguing by contradiction, we assume that such points instead exist. 
%  Since $\mathcal{C}$ is open, we can find $x_0,x_1\in[s,z_0]$, such that $x_1>x_0$,  $x_0,x_1\in\mathcal{A}$, and $(x_0,x_1)\subset\mathcal{C}$.
% \noindent
Given  $x\in(x_0,x_1)$, set 
 $j\coloneqq i-(x_1-x)$  for every $i>0.$
  % and for $x_1\in\mathcal{A}$
Then, recalling that $x\in \mathcal{C}$, so $v(x)>\mathcal{M}v(x)$,  and  that $x_1\in\mathcal{A}$, hence  $v(x_1)=\mathcal{M}v(x_1)$, we can write
%\label{kj}
\begin{equation*}
\begin{split}
v(x)&> \mathcal{M}v(x)=\sup_{i>0}\big\{v(x+i)-c_0i-c_1\big\}\geq \sup_{i>x_1-x}\big\{v(x+i)-c_0i-c_1\big\}\\
&=\sup_{j>0}\big\{v(x_1+j)-c_0j-c_1\big\}+c_0(x-x_1)=v(x_1)+c_0(x-x_1),  \ \ \ \forall x\in(x_0,x_1).
\end{split}
\end{equation*}
Therefore 
\begin{equation}\label{sa4}
v(x)-v(x_1)>  c_0(x-x_1) \ \ \ \ \forall x\in(x_0,x_1).
\end{equation}
Due to Proposition \ref{propempty}\emph{(\ref{2016-11-06:06})}, we have 
 for some for some  $y_1>x_1$, $y_1\in\mathcal{C}$,
\begin{equation}\label{sa}
v(x_1)= v(y_1)-c_0(y_1-x_1)-c_1.
\end{equation}
On the other hand, $v\geq \mathcal{M}v$ implies
\begin{equation}\label{sa2}
v(x)\geq v(y_1)-c_0(y_1-x)-c_1 \ \ \ \forall x\in(x_1,y_1).
\end{equation}
Combining \eqref{sa} and \eqref{sa2} we get
\begin{equation}\label{sa3}
v(x)-v(x_1)\geq c_0(x-x_1) \ \ \ \forall x\in(x_1,y_1).
\end{equation}
Then \eqref{sa4} and \eqref{sa3} show that the function
$$\varphi(x)=v(x_1)+ c_0(x-x_1), \ \ \ \ x\in\R_{++},$$
is such that $\varphi(x_1)=v(x_1)$ and $v-\varphi$ has a local minimum at $x_1$. Since $v$ is a viscosity supersolution to \eqref{QVI}, this implies
\begin{equation}\label{11}
\rho v(x_1)-c_0b(x_1)\geq f(x_1).
\end{equation}
 Now, by \eqref{sa4}, there exists $\xi\in(x_0,x_1)$  such that $v'(\xi)<c_0$. Let
$$y_2\coloneqq\sup\left\{x\in [x_0,\xi)\colon\ v'(x)\geq c_0\right\}.$$
The definition above is well posed as $x_0\in\mathcal{A}$, so that by Corollary~\ref{Lemma1}\emph{(\ref{derv})} we have $v'(x_0)=c_0$. Moreover, by continuity of $v'$ and by definition of $y_2$ we have
\begin{equation}\label{pa}
y_2<\xi<x_1, \ \ \ v'(y_2)=c_0, \ \ \ v'(x)<c_0 \ \ \forall x\in (y_2,\xi).
\end{equation}
Therefore, considering that $v$ is twice differentiable in $(x_0,\xi)$ %(\red{ I think it's $(x_0,\xi)$, not $(y_0,\xi)$})
 as this interval is contained in $\mathcal{C}$, from \eqref{pa} and by continuity of $v'$ we see that % we can find a sequence $x_n\downarrow y_2$ such that
\begin{equation}\label{mbv}
v'(y_2)=c_0,\ v''(y_2)\leq 0.
%v'(x_n)\rightarrow c_0, \ \ \ v''(x_n)\leq 0.
\end{equation}
The equality $\mathcal{L}v=f$ holds
 % on this sequence
 in classical sense at $y_2$, hence \eqref{mbv} entails
% Taking the limit and taking into account \eqref{mbv}
\begin{equation}\label{12}
\rho v(y_2)- c_0b(y_2)\leq f(y_2).
\end{equation}
Combining \eqref{11} with \eqref{12}, we get
\begin{equation}\label{10}
\rho (v(x_1)-v(y_2))-  c_0(b(x_1)-b(y_2))\geq f(x_1)-f(y_2).
\end{equation}
On the other hand,
considering
 \eqref{sa4} 
with $x=y_2$,
%to the limit for $x\downarrow y_2$ 
and then combining it with \eqref{10}, we get
%\begin{equation}\label{13}
%v(y_1)\geq v(x_1)+c_0(y_1-x_1).
%\end{equation}
%Combining \eqref{10} and \eqref{13} we get
\begin{equation}\label{14}
\rho c_0(x_1-y_2)
-c_0(b(x_1)-b(y_2))> f(x_1)-f(y_2)
%(\rho+\beta)c_0(x_1-y_2)\geq f(x_1)-f(y_2).
\end{equation}
% \emph{Step 2.}   Applying to the point $z_0$ the same methodology as to the point $x_1$ %Arguing at $z_0$ as we have done for $x_1$,
% to obtain \eqref{11}, we get
% \begin{equation}\label{mz}
% \rho v(z_0)+\beta c_0 z_0\geq f(z_0).
% \end{equation}
Now, as $x_1\in \mathcal{A}$,
by \eqref{sa} we have
\begin{equation}\label{zzz}
v(y_1)-c_0(y_1-x_1)-c_1=\sup_{y>x_1} \big\{{v}(y)-c_0(y-x_1)-c_1\big\}.
\end{equation}
 % by Proposition~\ref{propempty}, there exists $z_1>_1$  such that
The function $v$ is twice differentiable at $y_1$ since $y_1\in\mathcal{C}$, so \eqref{zzz} yields
$$
v'(y_1)=c_0, \ v''(y_1)\leq 0.
$$
Therefore the equality$\mathcal{L}v(y_1)=f(y_1)$ yields the inequality
%\begin{equation}\label{mz2}
\begin{equation}
  \label{eq:2017-11-14:00}
  \rho v(y_1)- c_0 b(y_1)\leq f(y_1).
\end{equation}
% \end{equation}
Combining
\eqref{eq:2017-11-14:00}
with
\eqref{11},
we get
%\begin{equation}\label{101}
\begin{equation}
  \label{eq:2017-11-14:01}
  \rho (v(y_1)-v(x_1))-  c_0(b(y_1)-b(x_1))\leq f(y_1)-f(x_1).
\end{equation}
% \end{equation}
On the other hand, from \eqref{zzz} we get
%\begin{equation}\label{zz}
\begin{equation}
  \label{eq:2017-11-14:02}
  v(y_1)-v(x_1)\geq c_0(y_1-x_1).
\end{equation}
%\end{equation}
So, from 
\eqref{eq:2017-11-14:01}
and
\eqref{eq:2017-11-14:02}
% the last two inequalities
we get
%From \eqref{101} and \eqref{zz}
\begin{equation}\label{141}
\rho c_0(y_1-x_1)-c_0(b(y_1)-b(x_1))\leq f(y_1)-f(x_1).
%(\rho+\beta)c_0(z_1-z_0)\leq f(z_1)-f(z_0).
\end{equation}
To conclude, note that \eqref{14} and \eqref{141} are not compatible with the strict concavity of 
\begin{equation*}
\mathbb{R}_{++}\rightarrow \mathbb{R},\
x \mapsto   f(x)+c_0 b(x)-\rho c_0 x
\end{equation*}
which follows from Assumptions~\ref{2016-11-06:00} and \ref{ass:bb}.
% \emph{Step 2.}
% By Step~1, to conlcude the proof of the theorem, we only need to show there does not exist $\epsilon>0$ such that $(0,\epsilon)\subset \mathcal{C}$.
% {\color{cyan}CONCLUDERE}
\end{proof}

% Let $M\in (0,+\infty]$.
 % $(a,b)\subset \mathbb{R}_{++}$ be a non-empty opern interval.
 Under Assumption \ref{ass:bb}, 
Lemma~\ref{lemma:a} and
 Lemma~\ref{2017-04-04:00} provide
% $\mathcal{C}\neq \mathbb{R}_{++}$ and $\mathcal{C}$ has at most two connected components, one of which is unbounded, i.e.,
\begin{equation}\label{possibilities}
  \begin{split}
    \mbox{either (i)}\quad &\mathcal{C}= \mathbb{R}_{++}\\
\mbox{or  (ii)}\quad & \exists\  r, s, \ \ 0\leq r< s<\infty\colon
  \mathcal{C}=(0,r)\cup (s,\infty).
\end{split}
\end{equation}
Case (i) above corresponds to the case in which the continuation region invades all the state space and it is never convenent to undertake an action. In case (ii) the action region is not empty and there is convenience to undertake an action when the system reaches this region.

Consider the homogeneous ODE
\begin{equation}
  \label{2017-03-30:00}
  \mathcal{L}u=0  \quad \mbox{on}\quad \mathbb{R}_{++}.
\end{equation}
By \cite[Th.~16.69]{Breiman} its general solution is of the form
$$
u=A\psi+B\varphi, \ \ \ A,B\in\R,$$
where  
$\psi, \phi$ are, respectively,  the unique (up to a multiplicative constant) strictly increasing and strictly decreasing solutions to \eqref{2017-03-30:00} and, 
as $0$ and $\infty$ are not accessible boundaries for the reference diffusion $Z$, these fundamental solutions satifsy the
following boundary conditions
\begin{equation}\label{2017-04-05:01}
\psi(0^+)\coloneqq   \lim_{x\rightarrow 0^+}\psi(x)=0, \ \ \ 
  \varphi(0^+)\coloneqq \lim_{x\rightarrow 0^+}\varphi(x)=+\infty, \ \ \ \lim_{x\rightarrow\infty} \psi (x)=+\infty, \ \ \ \lim_{x\rightarrow\infty} \varphi (x)=0. 
\end{equation}  
 Other properties
of these functions can be found on \cite[Sec.~16.11]{Breiman}.
%The functions $\psi,\varphi$ are fundamental solutions of
%\eqref{2017-03-30:00}: every solution of \eqref{2017-03-30:00} is a linear combination of them.
On the other hand, 
the function $\hat v$ defined in \eqref{eq:2017-04-07:03} is the unique solution in $\R_{++}$, within the class of functions having at most linear growth,
to the nonhomogeneous ODE
$\mathcal{L}u=f$
(see \cite[Th.~16.72]{Breiman}: actually in the quoted result the function $f$ is required to be bounded, but the proof works as well in our context within the class of functions having at most linear growth).  
It follows that every classical solution to 
\begin{equation}\label{eq:homo}
\mathcal{L}u=f, \ \ \ \mbox{over} \ \mathcal{I}\subset \R_{++},
\end{equation} where $\mathcal{I}$ is an open interval, must have the form $u=A\psi+B \varphi +\hat v$. 
Therefore, as by Proposition \ref{viscosoluQVI} and Theorem  \ref{th:viscosoluQVI} the value function 
$v$ solves in classical sense \eqref{eq:homo}, according to the two possibilities of \eqref{possibilities}, in case (i) there must  exist
real numbers
$A,B$
 such that
\begin{equation}
  \label{2017-04-04:01}
  \begin{aligned}
  &   \ \ \ v=\hat{v}+A\psi+B\varphi
    \ \ \ \mbox{on}\quad\mathbb{R}_{++};&&%\quad\mbox{if}\quad\mathcal{C}=\mathbb{R}_{++};
    \end{aligned}
    \end{equation}
    in case (ii) there must  exist
real numbers
    $A_r,B_r,A_s,B_s$
    \begin{equation}\label{2017-04-04:01bis}
    \begin{aligned}
   & \begin{dcases}
      v=\hat{v}+A_r\psi+B_r\varphi& \quad\mbox{on}\quad (0,r),
      \\
      v=\hat{v}+A_s\psi+B_s\varphi& \quad\mbox{on}\quad
      (s,\infty).
    \end{dcases}&&%\quad\mbox{if}\quad\mathcal{C}=(0,r)\cup  (s,\infty)\neq \mathbb{R}_{++}.
  \end{aligned}
\end{equation}
%Notice that, since $v,\hat{v},\psi$ are bounded on bounded sets, it must be $B=B_r=0$,
%due to 
%\eqref{2017-04-05:00}.
%By $v\geq \hat{ v}$, by  
%\eqref{2017-04-05:01}, and recalling that $\psi$ is increasing, 
% we have also $A\geq 0$, $A_r\geq 0$.

\begin{Proposition}\label{pps}   Let  Assumption \ref{ass:bb} hold. According to the cases \emph{(i)} and \emph{(ii)} of \eqref{possibilities} we have, respectively:
\begin{enumerate}[-]
\item
if case \emph{(i)}
holds,
then
 $v\equiv\hat v$, hence $A=B=0$ in \eqref{2017-04-04:01};   
\item 
if case \emph{(ii)}
holds,
then
${\displaystyle\lim_{x\rightarrow\infty}} (v(x)-\hat v(x))=0$ and $A_s=B_r=0$, $A_r,B_s\geq 0$ in \eqref{2017-04-04:01bis}.
\end{enumerate}
\end{Proposition}
\begin{proof}
% {\color{red}(riguardare dopo la modifica alla proof di 
% Proposition\ref{prop:preB}\eqref{2015-11-23:02B}}

Assume that case (i) holds.
As $\mathcal{L}v=f$ on $\mathcal{C}=\R_{++}$,
by a standard localization procedure we get (see, e.g., the proof of Proposition \ref{prop:preB})
%\eqref{estimateB},
%and
%\eqref{eq:2017-03-25:03},
%we can apply
% Dynkin formula
%to get
\begin{equation}
  \label{eq:2017-11-14:03}
  v(x)
=\E\left[\int_0^{t} e^{-\rho s}f(X^{x,\emptyset}_s)ds\right] 
+\E\left[e^{-\rho t} v (X_t^{x,\emptyset})\right]\qquad\forall t\in\R_+.
\end{equation}
% By \eqref{estimateB}
% and
% \eqref{eq:2017-03-25:03},
We
 pass to the limit
$t\rightarrow \infty$
on the
first addend of the right hand side
by using
the monotone convergence theorem.
%{\red{SF: RIMANEGGIATO QUESTO PASSAGGIO AL LIMITE CONTROLLARE MR}}
 As for the second addend, 
we use
 \eqref{estimateB}
and
 \eqref{eq:2017-09-18:00} %\eqref{eq:2017-03-25:03}
with $I=\emptyset$ 
to write
$$
0\leq \E\left[e^{-\rho t} v (X_t^{x,\emptyset})\right]\leq e^{-\rho t}\frac{f^*(\alpha)}{\rho} + \frac{\alpha}{\rho} x \qquad\forall \alpha\in (0,c_0\rho].
$$ 
Then 
$$0\leq \limsup_{t\to\infty} \ \E\left[e^{-\rho t} v (X_t^{x,\emptyset})\right]\leq \frac{\alpha}{\rho} x \qquad\forall \alpha\in (0,c_0\rho].$$
By arbitrariness of  $\alpha$ we conclude that 
$$\lim_{t\to\infty}  \E\left[e^{-\rho t} v (X_t^{x,\emptyset})\right]=0.$$
Hence
% \eqref{estimateB},
\begin{equation}
  \label{eq:2017-11-14:05}
  v(x)=\E\left[\int_0^\infty e^{-\rho s}f(X^{x,\emptyset}_s)ds\right].
\end{equation}
By definition of $\hat v$ and by the inequality $v\geq \hat{ v}$, this proves the claim.

Now assume that case (ii) holds.
For each $x>s$  set 
$\tau_{x}\coloneqq \inf\left\{t\geq 0\colon X^{x,\emptyset}_t\leq s\right\}$. As $\infty$ is a natural boundary  for $Z^{0,x}=X^{x,\emptyset}$, by \eqref{natural} we have  
\begin{equation}\label{lima2}
 \lim_{x\rightarrow\infty}\P\{\tau_{x}\geq M\}=1 \ \ \ \forall M>0.
\end{equation}
If $0<x<x'$, \red{by \eqref{ineq} with $I=\emptyset$} we get
$$
\mathbb{P}\mbox{-a.s.},\
 X^{x,\emptyset}_t\leq X_t^{x',\emptyset}\ \mbox{for all $t\geq 0$},
$$
so, we also have $\tau_x\leq \tau_{x'}$ $\mathbb{P}$-a.s..
If $\{x_n\}_{n\in \mathbb{N}}$ is a sequence diverging to $\infty$,
we then have
\begin{equation}\label{lima}
 \lim_{n\to\infty}\tau_{x_n}=\infty\qquad \mathbb{P}\mbox{-a.s..}
\end{equation}
% Then,
% by monotonicity
% of $x \mapsto X^{x,\emptyset}_t$,
% hence  of $x\mapsto \tau_x$,
As $\mathcal{L}v=f$ on $(s,\infty)$, 
as
for \eqref{eq:2017-11-14:03},
we get
\begin{equation}
  \label{eq:2017-05-18:04}
  v(x_n)=\E\left[\int_0^{\tau_{x_n}\wedge t} e^{-\rho \zeta}f(X^{x_n,\emptyset}_\zeta)d\zeta\right]
+
\E\left[e^{-\rho (\tau_{x_n}\wedge t)} v (X^{x_n,\emptyset}_{t\wedge \tau_{x_n}})\right]
 \qquad \forall t\in\R_+,\ n\in \mathbb{N}.
\end{equation}
%By considering the two cases $\tau_x<t$ and $\tau_x\geq t$,
% {\red{SF: CORRETTO A SEGUIRE}} 
Therefore, splitting over $\{\tau_{x_n}< t\}$ and  $\{\tau_{x_n}\geq t\}$ the second addend on the right hand side, 
 \begin{equation*}
  \begin{split}
        v(x_n)
&=\E\left[\int_0^{\tau_{x_n}\wedge t} e^{-\rho \zeta}f(X^{x_n,\emptyset}_\zeta)d\zeta\right]
+ \E\left[\mathbf{1}_{\{\tau_{x_n}\geq t\}}
e^{-\rho  t} v (X^{x_n,\emptyset}_t)\right]
+\E\left[\mathbf{1}_{\{\tau_{x_n}<t\}}e^{-\rho (\tau_{x_n}\wedge t)} v (X^{x_n,\emptyset}_{t\wedge \tau_{x_n}})\right]\\
&\leq
\E\left[\int_0^{\tau_{x_n}\wedge t} e^{-\rho \zeta}f(X^{x_n,\emptyset}_\zeta)d\zeta\right]
+ \E\left[
% \mathbf{1}_{\{\tau_x=\infty\}}
\mathbf{1}_{\{\tau_{x_n}\geq t\}} e^{-\rho  t} v (X^{x_n,\emptyset}_t)\right]
+
\mathbb{E}[e^{-\rho \tau_{x_n}}\mathbf{1}_{\{\tau_{x_n}<t\}}]v(s).
\end{split}
\end{equation*}
for all $t\geq 0$.
Now we pass to the limit $t\rightarrow \infty$
% in \eqref{eq:2017-11-14:04}
by using the same arguments 
used to obtain
\eqref{eq:2017-11-14:05},
and we get
%  Passing to the limit
% for $t\rightarrow \infty$ and applying  dominated  by means of  the same ingredients used above
% we get
\begin{equation*}
  v(x_n)
\leq
\E\left[\int_0^{\tau_{x_n}} e^{-\rho \zeta}f(X^{x_n,\emptyset}_\zeta)d\zeta\right]
+\E[e^{-\rho \tau_{x_n}} \mathbf{1}_{\{\tau_{x_n}<\infty\}}]v (s).
\end{equation*}
Then, the definition of $\hat v$ provides
% and the flow property of $X^{x,\emptyset}$ yield
\begin{equation*}
 % \begin{split}
    v(x_n)-\E[e^{-\rho \tau_{x_n}} \mathbf{1}_{\{\tau_{x_n}<\infty\}}]v (s)
\leq\hat  v(x_n)-\E\left[\mathbf{1}_{\{\tau_{x_n}<\infty\}}\int_{\tau_{x_n}}^{\infty} e^{-\rho \zeta}f(X_\zeta^{x_n,\emptyset})d\zeta\right]\leq 
\hat{ v}(x_n).
%&= \hat  v(x)-\E\left[\mathbf{1}_{\{\tau_x<\infty\}}e^{-\rho \tau_{x}}\int_{0}^{\infty} e^{-\rho s}f(X^{s,\emptyset})ds\right].
%\end{split}
\end{equation*}
%1
Using \eqref{lima}
and recalling that $v \geq \hat{v}$,
 we 
conclude
$ \displaystyle\lim_{n\rightarrow\infty} (v(x_n)-\hat v(x_n))=0$.
Since the sequence $\{x_n\}_{n\in \mathbb{N}}$ was arbitrary, we conclude
 \begin{equation}\label{lil}
 \displaystyle\lim_{x\rightarrow \infty}(v(x)-\hat{ v}(x))=0.
 \end{equation}
 % by dominated (or monotone) convergence. 
From \eqref{2017-04-05:01} and \eqref{lil} we have $A_s=0$ and $B_s\geq 0$. 
Finally, since $v\geq \hat v$ and $v$ is finite in $(0,r)$, from  \eqref{2017-04-05:01} we have $A_r\geq 0$ and $B_r=0$.
\end{proof}
Set  $$\hat{v}^*(z)\coloneqq \sup_{x>0} \big\{\hat v(x)-zx\big\}, \ \ \ \ \ z\in\R_{++}.$$
We are going to introduce an assumption, requiring that $c_1$ is not too large, that guarantees, at once, 
that the action region is not empty and that the structure of the continuation and action regions are
$$
\mathcal{A}=(0,s]  \quad \mbox{and} \quad \mathcal{C}=(s,\infty) 	\quad \mbox{for some} \ s>0.
$$   
Under this nice structure, it turns out that it is convenient to undertake an action when the system lies below a given threshold and lat it evolve autonomously when the system lies above this threshold. Henceforth, we will call this threshold \emph{trigger boundary}.
\begin{Assumption}\label{Assumption:c0c1}
 $c_1<{ \hat{ v}^*}(c_0)$.
\end{Assumption}
The following result provides a way to check explicitly the validity of Assumption \ref{Assumption:c0c1}.
\begin{proposition}\label{prop:ass2} Let 
$f(x)\geq K x^\gamma$ for some $K>0$, $\gamma\in(0,1)$, and set $K':= \frac{\gamma K}{\rho +\gamma L_b+\frac{1}{2}\gamma(1-\gamma)L_\sigma^2}$. Then
$$\hat{v}^*(c_0)=  K'\frac{1-\gamma}{\gamma}\left(\frac{c_0}{K'}\right)^{\frac{\gamma}{\gamma-1}}.
$$
\end{proposition}
\begin{proof}
Let $x\in\R_{++}$. With a localization procedure similar to the one of the proof of Proposition \ref{prop:preB}, we get from It\^o's formula
\begin{eqnarray*}
&&\E\left[e^{-\rho t}\big|X^{x,\emptyset}_t\big|^{\gamma}\right]\\&&= x^{\gamma}+\E\left[\int_0^t e^{-\rho s} \left[-\rho \big(X^{x,\emptyset}_s\big)^{\gamma}+\gamma\big(X^{x,\emptyset}_s\big)^{\gamma-1} b(X_s^{x,\emptyset}) +
 \frac{1}{2} \gamma(\gamma-1)\big(X^{x,\emptyset}_s\big)^{\gamma-2} \sigma^2(X_s^{x,\emptyset})
\right] ds\right]\\
&&\geq x^{\gamma}+\E\left[\int_0^t e^{-\rho s}\left[-\rho  \big(X^{x,\emptyset}_s\big)^{\gamma}-{L}_b(1-\gamma)\,\big(X^{x,\emptyset}_s\big)^{\gamma} -
 \frac{1}{2} {L}^2_\sigma \gamma(1-\gamma)\big(X^{x,\emptyset}_s\big)^{\gamma} 
\right] ds\right].
\end{eqnarray*}
Then we get  
$$
\E\left[e^{-\rho t}\big(X^{x,\emptyset}_t\big)^{\gamma}\right]\geq x^\gamma e^{-\left(\rho +\gamma L_b+\frac{1}{2}\gamma(1-\gamma)L_\sigma^2\right)t}, \ \ \ \ \forall t\in\R_+.
$$
From that and from the assumption on $f$, we obtain
$$
\hat{v} (x)\geq \frac{K}{\rho +\gamma L_b+\frac{1}{2}\gamma(1-\gamma)L_\sigma^2} x^\gamma=\frac{K'}{\gamma} x^\gamma, \ \ \ \ \forall x\in\R_{++}.
$$
Hence, 
$$\hat{v}^*(c_0):=\sup_{x>0}\big\{\hat{v}(x)-c_0 x\big\} \geq \sup_{x>0}\Bigg\{\frac{K'}{\gamma}x^\gamma -c_0 x\Bigg\}=  K'\frac{1-\gamma}{\gamma}\left(\frac{c_0}{K'}\right)^{\frac{\gamma}{\gamma-1}}.\vspace{-.8cm}
$$
\end{proof}
\medskip
 \begin{proposition}
   \label{2017-04-07:00}
   Let  Assumptions \ref{ass:bb} and  \ref{Assumption:c0c1} hold.
%  then
% $v(0^+)>0$.
% In particular,
Then there exists $s>0$ such that $\mathcal{C}=(s,\infty)$  and, consequently, $\mathcal{A}=(0,s]$.
%\eqref{eq:2017-04-07:02}
 \end{proposition}
 \begin{proof}
   First, notice that, as  $\hat{ v}$
satisfies \eqref{estimateB}, it follows that  ${ \hat{ v}^*}$ is finite on $\mathbb{R}_{++}$.
Considering that $v\geq \hat{ v}$ and that $\hat{ v}$ is nondecreasing, we have
\begin{equation}\label{2017-04-07:04}
\begin{split}
\lim_{x\rightarrow 0^+}v(x)&\geq 
\lim_{x\rightarrow 0^+}\mathcal{M}v(x)
\geq
\lim_{x\rightarrow 0^+}\mathcal{M}\hat{ v}(x)= \lim_{x\rightarrow 0^+}\sup_{i>0}\big\{\hat{ v}(x+i)-c_0i-c_1\big\}\\
& \geq \lim_{x\rightarrow 0^+}
\sup_{i>0}
\big\{\hat{v}(i)-c_0i-c_1\big\}
=\hat{v}^*(c_0)-c_1>0.
% = \lim_{x\rightarrow 0^+}\sup_{i>0}\big\{\hat{ v}(x+i)- \hat v(i)-c_0i-c_1+\hat v(i)\big\}\\
% &\geq 
% %=\mathcal{M}\hat v (0)
% {
% \hat{ v}^*(c_0)-c_1
% -\lim_{x\rightarrow 0^+}\sup_{i>0}\big\{\hat{ v}(x+i)-\hat{ v}(i) \big\}
% }
% = \hat{ v}^*(c_0)-c_1>0,
\end{split}
\end{equation}
% {\color{cyan}
% where the last equality comes from \eqref{estimateB} and from the concavity (and continuity) of $\hat{ v}$ on $\mathbb{R}_{++}$.
% }
% {\red{SF: QUESTO PASSAGGIO CHW DICI NON MI E' CHIARO; MI PARE INVECE CHE SI POSSA CONCLUDERE SUBITO DALL'ULTIMO TERMINE DELLA PRIA RIGA PER MONOTONIA (MAGARI PRENDENDO IL LIMINF INVECE DEL LIM SE NECESSARIO)}}
Now assume by contradiction  that $(0,r)\subset \mathcal{C}$, for some $r>0$.
By Proposition \ref{pps} we have
 $$v(x)=\hat{ v}(x)+A_r \psi(x), \ \ \ \ x\in (0,r),$$ for some $A_r\geq 0$.
Then, as $\psi(0^+)=0$, we must have  $v(0^+)=\hat{v}(0^+)=0$. The latter  contradicts 
\eqref{2017-04-07:04}, hence we conclude.
\end{proof}

Under Assumptions \ref{ass:bb} and \ref{Assumption:c0c1},
  the structure of $\mathcal{C}$ and $\mathcal{A}$ established by Proposition \ref{2017-04-07:00} joined 
with Proposition~\ref{pps} provides the following  structure for $v$:  for some $B= B_s\geq 0$
\begin{equation}\label{vbis}
 v(x)=\begin{dcases}
B\varphi(x)+\hat{v}(x),& \mbox{if} \ x\in (s,\infty),\\
 B\varphi(s)+\hat{v}(s)-c_0(s-x), &\mbox{if} \ x\in (0,s].
 \end{dcases}
\end{equation}
%
%\begin{Assumption}\label{ass:bbss}
%$b,\sigma\in C^1(\R_{++};\R)$.
%\end{Assumption}
%\noindent $\eyedx\eyedx$ Assumptions \ref{ass:bb}, \ref{ass:bbss}, and  \ref{Assumption:c0c1} will be standing
%for the rest of this Section~{\color{red}REF}.   
\begin{lemma}\label{2017-05-22:01}
Let Assumption \ref{ass:bb} hold.  Let  $a\geq 0$ and let $u\in C^2((a,\infty);\mathbb{R})$  satisfy $\mathcal{L}u=f$
% and $u'\geq 0$ 
on $(a,\infty)$.
  If $x_0\in (a,\infty)$ is a local minimum point for $u'$, 
then
  $u'(x_0)> 0$
 and
there is no local maximum 
% or minimum
point for $u'$ in $(x_0,\infty)$.
\end{lemma}
  \begin{proof}
 As $b,\sigma, \red{f}\in C^1(\R_{++};\R)$, from 
 \begin{equation}
   \label{eq:2017-05-29:00}
   \rho u(x)=b(x)u'(x)+\frac{1}{2}\sigma^2(x)u''(x)+f(x), \ \ \forall x\in (a,\infty),
 \end{equation}
 we obtain $u''\in C^1((a,\infty);\R)$, \red{i.e.\ $u\in C^3((a,\infty);\R)$. We differentiate \eqref{eq:2017-05-29:00} getting} 
\begin{equation}\label{ppp}
\rho u'(x)= b'(x)u'(x)+b(x)u''(x)+\frac{1}{2} \sigma^2(x)u'''(x)+\sigma\sigma'(x)u''(x)+f'(x), \qquad\forall x\in(a,\infty).
\end{equation}  
Let $x_0\in (a,\infty)$ be a local minimum point for $u'$.
Then $u''(x_0)=0$ and $u'''(x_0)\geq 0$ so, by \eqref{ppp}, we have
\begin{equation}
  \label{eq:2017-05-29:01}
  \rho u'(x_0)\geq b'(x_0)u'(x_0)+f'(x_0).
\end{equation}
Note that from \eqref{eq:2017-05-29:01}, using 
Assumptions~\ref{2016-11-06:00}
and \ref{ass:rho}, we obtain 
$u'(x_0)>0$.
%\eqref{eq:2017-05-29:01}
%We now show
% that there is no local maximum point 
%for $u'$
%on $(x_0,\infty)$.
% First of all,   we recall that $v\in C^2((s,\infty);\R^+)$ and that $v'\geq 0$
%  over $\mathcal{C}$ as $v$ is nondecreasing (Proposition~\ref{prop:preB}\emph{(\ref{2015-11-23:03B})}).
Now, arguing by contradiction, assume that $x_1\in (x_0,\infty)$  is local maximum point for $u'$. Then  $u''(x_1)=0$ and $u'''(x_1)\leq 0$, so,
by \eqref{ppp}, we have
% We have $v''(x_0)=v''(x_1)=0$, $v'''(x_0)\geq 0$, and $v'''(x_1)\leq 0$. So, from \eqref{ppp} we get
% \rho v'(x_0)\geq  b'(x_0)v'(x_0)+f'(x_0)
% \end{equation}
%    and 
\begin{equation}\label{max}
\rho u'(x_1)\leq  b'(x_1)u'(x_1)+f'(x_1).
\end{equation}
    % Assume now, by contradiction, that there exist $x_0,x_1$ with $s<x_0<x_1$ such that $x_0$ and $x_1$ are, respectively, a local minimum point and a local maximum point for $v'$. Moreover, assume
Without loss of generality, we can assume that
\begin{equation}\label{primo}
  u'(x_0)\leq u'(x_1).
\end{equation}
Combining \eqref{eq:2017-05-29:01} and \eqref{max} and taking account that  $f'$ is strictly decreasing, we get
\begin{equation}\label{aas}
(\rho-b'(x_1))u'(x_1)\leq f'(x_1)< f'(x_0)\leq (\rho-b'(x_0))u'(x_0).  
\end{equation}
Now, by Assumption \ref{ass:bb} we have  $b'(x_0)\geq b'(x_1)$. So, the fact that $u'(x_0)> 0$
 and \eqref{aas} yield
\begin{equation}\label{aasbis}
(\rho-b'(x_1))u'(x_1)<  (\rho-b'(x_1))u'(x_0).  
\end{equation}
By Assumption~\ref{ass:rho}, we have the $\rho-b'(x_1)>0$. Hence,
from  \eqref{aasbis} we obtain $u'(x_1)< u'(x_0)$, contradicting \eqref{primo}.    
% Now we show that there is no local mimum point on $(x_0,\infty)$.
% Again reasoning by contradiction, let $x_1\in (x_0,\infty)$ be a local minimum point for $u'$.
% Observe that $x_0,x_1$ must provide strict local minima for $u'$ on $[x_0,x_1]$, otherwise there would
% exist local maxima on $(x_0,x_1)$, contradicting the first part of the proof.
% But then the maximum point of $u'$ on $[x_0,x_1]$ lies in the interior $(x_0,x_1)$, and we have again a contradiction with the first part of the proof.
% We conclude that we cannot assume the existence of a local minimum point for $u'$ on $(x_0,\infty)$.
\end{proof}
{\red{Recall that a function $\varphi:\mathcal{O}\to\R$, with $\mathcal{O}$ open interval, is said  quasiconcave if
$$
\varphi (\lambda x+(1-\lambda) x')> \min\left\{\varphi(x),\varphi(x')\right\} \ \ \ \ \forall x,x'\in\mathcal{O}, \ \forall \lambda\in(0,1).
$$
Strictly quasiconcave functions can be characterized as functions that are either strictly increasing, or strictly decreasing, or strictly increasing on the left of a point $x^*\in\mathcal{O}$ and strictly decreasing on the right of $x^*$.  
}}
\begin{lemma}\label{2017-05-29:02}
Let Assumption \ref{ass:bb} hold. Let  $a\geq 0$, let $u\in C^2((a,\infty);\mathbb{R})$  satisfy $\mathcal{L}u=f$ on $(a,\infty)$, and  assume
  that $\displaystyle \liminf_{x\rightarrow \infty}u'(x)\leq 0$.
Then $u'$ is strictly quasiconcave.
% Then, for any $c> 0$, there are at most two solutions to the equation
  % \begin{equation}\label{2017-05-29:04}
  %   u'(x)=c\qquad x\in (a,\infty).
  % \end{equation}
%Moreover, there exists $A\in [0,\infty)$ such that $u'$ is streactly increasing on $(0,A)$ and streactly decresaing on $(A,\infty)$.
\end{lemma}
\begin{proof}
By virtue of \cite[Proposition~3.24]{Avriel},
it is sufficient to show that $u'$ does not admit any local minimum.
Argue by contradiction and assume that
 $x_0\in (a,\infty)$ is a local minimum point for $u'$. The proof of 
 Lemma~\ref{2017-05-22:01} shows then that $u'(x_0)>0$. Hence, since $\displaystyle \liminf_{x\rightarrow \infty}u'(x)\leq 0$, there must exists a local maximum point $x_1\in (x_0,\infty)$. This  contradicts  Lemma~\ref{2017-05-22:01} and we conclude.
  % The proof follows immediately by Lemma~\ref{2017-05-22:01}.
  % Indeed, assuming by contradiction that there exist $x_1,x_2,x_3\in (a,\infty)$, $a<x_1<x_2<x_3$, such that $u'(x_i)=c$, $i=1,2,3$, then there must exist $x_*,x^*\in (x_1,x_3)$ being local minimum and local maximum points for $u'$, respectively. Without loss of generality we can assume $x_*\neq x^*$.
  % By Lemma~\ref{2017-05-22:01}, it must be $x_*>x^*$.
  % Now, since $x_*<x_3$, $u'(x_*)\leq u'(x_3)=c>0$ and $\displaystyle \liminf_{x\rightarrow \infty}u'(x)\leq 0$, there must exist a local maximum point $x_4\in (x_*,\infty)$, which is impossible by Lemma~\ref{2017-05-22:01}.
\end{proof}

\begin{Proposition}\label{structv'} Let  Assumptions \ref{ass:bb} and  \ref{Assumption:c0c1} hold.
{${}$}
\begin{enumerate}[(i)]
  \item\label{2017-05-30:01} There exists a unique $S\in \mathcal{C}=(s,\infty)$ such that $v'(S)=c_0$.
  \item \label{2017-05-30:02} There exists (a unique) $x^*\in (s,S)$ such that $v'$ is strictly increasing in $(s,x^*]$ and strictly decreasing in $[x^*,\infty)$.
  \item \label{2017-05-30:03}
$      \displaystyle\lim_{x\rightarrow \infty} v'(x)= 0$.
  \end{enumerate}
% Let $S$ be as 
% There exists $x^*\in (s,S)$ such that $v'$ is strictly increasing in $(s,x^*]$ and strictly decreasing in $[x^*,+\infty)$. Moreover,
\end{Proposition}
\begin{proof}
\emph{(\ref{2017-05-30:01})}
Corollary~\ref{Lemma1}\emph{(\ref{ffa})} and Proposition~\ref{propempty}\emph{(\ref{2016-11-06:06})}
yield the existence of  $S\in \mathcal{C}=(s,\infty)$ such that $v'(S)=c_0$.
Regarding uniqueness, observe
first that $v$ satisfies the requirements of Lemma \ref{2017-05-29:02} (plugging 
$v$ in place of $u$) with $a=s$
and where
\begin{equation}
  \label{eq:2017-05-22:00}
  \liminf_{x\rightarrow \infty}v'(x)\leq 0
\end{equation} holds by \eqref{estimateB}. Then the fact that
 $v'(s)=c_0$ by  Corollary~\ref{Lemma1}\emph{(\ref{derv})} yields the uniqueness.
% Hence we can apply 
% Lemma~\ref{2017-05-29:02} to $v$ on $(s,\infty)$,
% and this, recalling that $v'(s)=c_0$, forces $S$ to be unique.

% By
% \eqref{eq:2017-05-22:00},
% there must exists 
% $M>0$ such that $v'$ has a local maximum point on $[s,M]$.
% Let $x^*$ denote such a point.

{\red{\noindent\emph{(\ref{2017-05-30:02})}
By \eqref{vbis} we have $v'(s)=c_0$. By \emph{(\ref{2017-05-30:01})} above we have $v'(S)=c_0$ and $v'(x)\neq c_0$ for each $x\in(s,S)$. Then 
the claim follows by 
Lemma~\ref{2017-05-29:02}.}}
\noindent\emph{(\ref{2017-05-30:03})}
This follows immediately
by  monotonicity of $v'$ on $[x^*,+\infty)$,
\eqref{eq:2017-05-22:00}, and Proposition \ref{prop:mono}, which provides $v'\geq 0$.
\end{proof}

\begin{theorem}\label{theoremformvaluefunction}  
Let  Assumptions \ref{ass:bb} and  \ref{Assumption:c0c1} hold. The 
 value function  has the form
  \begin{equation}\label{v}
    v(x)=\begin{dcases} B\varphi(x)+\hat{v}(x),
      & \mbox{if} \ x\in (s,\infty),\\
      B\varphi(S)+\hat{v}(S)-c_0(S-x)-c_1, &\mbox{if} \ x\in (0,s],
    \end{dcases}
  \end{equation}
and the triple $(B,s,S)$ 
is the unique solution
in  \red{$\R_+\times \R_{++}^2$}
to  the system
% the system
% is the unique solution in  $(0,\infty)^3$, with $S>s$, of the algebraic system
  \begin{equation}\label{smoothbase}
 \begin{dcases}
 \mbox{\textbf{(i)}} & B\varphi(s)+\hat{v}(s)=B\varphi(S)+\hat{v}(S)-c_0(S-s)-c_1,\\[0.3em]
 \mbox{\textbf{(ii)}}  & B\varphi'(s)+\hat{v}'(s)=c_0,\\[0.3em]
 \mbox{\textbf{(iii)}} & B\varphi'(S) +\hat{v}'(S)=c_0.
% \\[0.3em]
% \mbox{\textbf{(iv)}}&S>s.
 \end{dcases}
 \end{equation}
% Moreover, 
% % Conversely, 
% \eqref{smoothbase} has a unique solution $(B,s,S)\in(0,\infty)^3$
% if  we add the contraint
% \begin{equation}\label{2017-05-23:00}
% \mbox{\textbf{(v)}\ \ \ $B\varphi+\hat{ v}$ is nondecreasing}
% \end{equation}
% % % there exists a unique triple
% % % $(B,S,s)\in (0,\infty)^3$ 
% % % solving
% % % to
% % % \eqref{smoothbase} and
% % \begin{enumerate}[(a)]
% % \item $B\varphi+\hat{ v}$ is nondecreasing, or 
% % \item 
% or if we assume $b(x)=bx$, $b\in \mathbb{R}$.
% then $v$
%  defined by
% \eqref{v}
%  is the value function.
\end{theorem}

\begin{proof}
{\red{Consider \eqref{vbis}. The expression of $v$ over $(s,\infty)$ in \eqref{v} and \eqref{vbis} is the same. As for the expression of $v$ over $(0,s]$, 
we note that, by definition of $\Xi(s)$, Proposition~\ref{propempty}, Corollary~\ref{Lemma1}, and 
 Proposition~\ref{structv'}(\ref{2017-05-30:01}),  we have 
\begin{equation}\label{maz}
0<S-s=\mathop{\operatorname{argmax}}_{i>0}\big\{v(s+i)-c_0i-c_1\big\}.
\end{equation}
Since $s\in\mathcal{A}$, we have $v(s)=[\mathcal{M}v](s)$; so, from \eqref{maz} we get 
$$
v(s)=v(S)-c_0(S-s)-c_1,
$$
from which we get the expression of $v$ over $(0,s]$ in \eqref{v}.
 Then the three equations of \eqref{smoothbase} follow, respectively, by imposing the continuity of $v$ at $s$, the smooth-fit at $s$ (as $v\in C^1(\R_{++};\R)$), and the condition of  Proposition~\ref{structv'}(\ref{2017-05-30:01})  defining $S$.  }}
% 
%we note
%\emph{(\ref{2017-05-30:01})}, Proposition \ref{pps}
%% Proposition \ref{structv'}, 
%and Corollary~\ref{Lemma1} yield \eqref{v}--\eqref{smoothbase} with $B\geq 0$ and 
%   $$0<S-s=\mathop{\operatorname{argmax}}_{i>0}\big\{v(s+i)-c_0i-c_1\big\}.$$
%\red{Now, if $B=0$, then $v=\hat{ v}$ on $[s,\infty)$ and, since $\hat{v},v\in C^1(\R_{++};\R)$, we would have $\hat{v}'_+(s)=v'(s)=c_0$. On the other hand, since $\hat{v}$ is concave, these facts would contradict Proposition~\ref{structv'}{\eqref{2017-05-30:01}}.}
% and, by 
% \eqref{smoothbase},
%$\frac{ v(S)-v(s)}{S-s}<c_0$.
%Then, there exists $\hat{ x}\in (s,S)$ such that $v'(\hat x)<c_0$, 
%which contradicts
%Proposition~\ref{structv'}{\eqref{2017-05-30:01}}-{(\ref{2017-05-30:02})}.
%Hence, it must be $B>0$.

%  Moreover,  $B>0$ in \eqref{v}. 
% Indeed, if $B=0$, then
%  the last two equations
% \emph{\textbf{(ii)-(iii)}}
%  in \eqref{smoothbase} provide $\hat{v}'(s)=\hat v' (S)=c_0$, contradicting the first equation in \eqref{smoothbase}.

To show that 
\eqref{smoothbase} has a unique solution in \red{$\R_+\times \R_{++}^2$}, we consider  the function
$$
h(\hat{B},x)= \hat{B}\varphi(x)+\hat v(x), \qquad (\hat{B},x)\in \R_{++}\times \R_{++}.
$$
For each \red{$\hat{B}\geq 0$}, 
 $\mathcal{L}h(\hat{B},\cdot)=0$ in $\R_{++}$ and $\displaystyle\liminf_{x\rightarrow \infty}\,h_x(\hat{B},x)\leq 0$ by \eqref{estimateB} and 	\eqref{2017-04-05:01}.
By 
Lemma~\ref{2017-05-29:02}
$h_x(\hat{B},\cdot)$ is strictly quasiconcave; hence, 
 there exist at most two solutions $\hat{s},\hat{S}$ %($s<S$)
 to  $h_x(\hat{B},\cdot)=c_0$ in $\R_{++}$. If such solutions exist,  we have
$h(\hat{B},\cdot)-c_0> 0$ on 
$(\hat{s}\wedge \hat{S},\hat{s}\vee \hat{S})$.
% Now, % if $s,S$ 
% % are two solutions of $h_x(B,\cdot)=c_0$ and if
Therefore, if
$(\hat{B},\hat{s},\hat{S})\in  \red{\R_+\times \R_{++}^2}$ solves
 \eqref{smoothbase},
then
% , by strict
% quasiconcavity of
% $h_x(B,\cdot)$,
% %(Lemma~\ref{2017-05-29:02}),
% we must have $h(B,\cdot)-c_0> 0$ on 
% $(s\wedge S,s\vee S)$,
% hence,
\eqref{smoothbase}\emph{\textbf{(i)}} yields
\begin{equation*}
  0<c_1=[\hat{B}\varphi(\hat{S})+\hat{ v}(\hat{S})]
  -
  [\hat{B}\varphi(\hat{s})+\hat{ v}(\hat{s})]-c_s(\hat{S}-\hat{s})=\int_{\hat{s}}^{\hat{S}} (h_x(\hat{B},r)-c_0)dr.
\end{equation*}
This forces $\hat{s}=\hat{s}\wedge \hat{S}$, $\hat{S}=\hat{s}\vee \hat{S}$, $\hat{s}\neq \hat{S}$. By the argument above we see that, 
% These remarks tells us that
if
 $(B_1,s_1,S_1)$ and $(B_2,s_2,S_2)$ are two different solutions to
\eqref{smoothbase} in  \red{$\R_+\times \R_{++}^2$}, we need to have
$s_1<S_1$, $s_2<S_2$, and $B_1\neq B_2$.

Now assume, by contradiction, that
 $(B_1,s_1,S_1)$ and $(B_2,s_2,S_2)$
are two
different
 solutions
of \eqref{smoothbase} in  \red{$\R_+\times \R_{++}^2$}.
Without loss of generality, we can assume $B_1< B_2$.
% Because of these remarks, we can assume without loss of generality that $B_1<B_2$.
Recalling that $\varphi$ is strictly decreasing,
% hence $\varphi'\leq 0$
 we have
\begin{equation}\label{finiamo}
 h_x(B_1,\cdot)> h_x(B_2,\cdot).
 \end{equation}
The latter inequality,
%  Hence,
% by
Lemma~\ref{2017-05-29:02}, and \eqref{smoothbase}\textbf{\emph{(ii)-(iii)}}
provide
% (Lemma~\ref{2017-05-29:02}),
\begin{equation}\label{finiamo1}
%\begin{cases}
%h_x(B_1,s_1)=h_x(B_1,S_1)=h_x(B_2,s_2)=h_x(B_2,S_2)=c_0,\\
 (s_1,S_1)\supset (s_2,S_2), \ \ \ h_x(B_1,\cdot)-c_0> 0 \ \mbox{on} \ (s_1,S_1).
% \end{cases}
\end{equation}
%The quasiconcavity of $h_x(B_1,\cdot)$ also implies
% $h_x(B_1,\cdot)-c_0\geq 0$ on $(s_1,S_1)$.
We can then write, using \eqref{finiamo}-\eqref{finiamo1} and  \eqref{smoothbase}\emph{\textbf{(i)}},
\begin{equation*}
  \begin{split}
0&=c_1-c_1
=
 \left(  h(B_1,S_1)-h(B_1,s_1)- c_0(S_1-s_1) \right) 
-
 \left(  h(B_2,S_2)-h(B_2,s_2)- c_0(S_2-s_2) \right) \\
&=
  \int_{s_1}^{S_1} (h_x(B_1,\xi)-c_0)d\xi
-
  \int_{s_2}^{S_2} (h_x(B_2,\xi)-c_0)d\xi\\
&
\geq  \int_{s_2}^{S_2} (h_x(B_1,\xi)-
h_x(B_2,\xi))d\xi>0,
  % \\&\geq
  % \int_{s_2}^{S_2} (h_x(B_1,\xi)-c_0)d\xi
  % \geq  \int_{s_2}^{S_2} (h_x(B_2,\xi)-c_0)d\xi
% \\&  =
%   h(B_1,S_2)-h(B_1,s_2)- c_0(S_2-s_2)=c_1.
%  &>
%  h(B_2,S_2)-h(B_2,s_2).
  % =
  % \int_{s_2}^{S_2} (h_x(B_2,s)-c_0)ds
\end{split}
\end{equation*}
which is a contradiction.
\end{proof}

%\begin{center}
%\texttt{METTERE FIGURA IN TEX SULLA VALUE FUNCTION}
%\end{center}

\section{Optimal control}\label{sec:opt}

In this section, through Theorem \ref{verificationtheorem}, we describe the structure of an optimal control for our problem through a recursive rule.
In the economic literature  --- see the stream of papers on stochastic impulse control at the beginning of the paragraph on the related linterature in the Introduction and \cite{BB} --- this rule is known as $(S,s)$-rule. Informally, this rule, rigorously stated in Theorem \ref{verificationtheorem} below, can be described as follows.
	\begin{itemize}
	\item The point $s$ works as an \emph{optimal  trigger boundary}:\ when the state variable is at level $s$ or below such level (i.e., it is within the action region $\mathcal{A}$), the controller acts.
	\item  The point $S$ works as an \emph{optimal target boundary}:\ when the controller acts, she/he does that in such a way to place the state variable at the level $S\in \mathcal{C}$.
\item When the state variable lies in the region $\mathcal{C}$, the controller let it evolve autonomously without undertaking any action until it exits from this region.
\end{itemize}

%\begin{center}
%\texttt{METTERE FIGURA IN TEX SUL CONTROLLO OTTIMO}
%\end{center}

\noindent Such rule is made rigorous by the following	 construction.  Let $x\in \R_{++}$ and consider the control
  $I^*=\{(\tau_n,i_n)\}_{n\geq 1}$ defined
% , with the
%,
  as follows:
  \begin{equation*}
    \begin{dcases}
      \tau_1\coloneqq \begin{dcases} 0
        &   \mbox{if} \ x\leq s,\\
        \inf\left\{t\geq 0\colon {Z}_t^{0,x}\leq s \right\} & \mbox{if} \ x>s,
      \end{dcases}\\[0.4em]
      i_1\coloneqq \begin{dcases} S-x
        &  \qquad\qquad\qquad\quad\mbox{if} \ \tau_1=0 \ (\mbox{i.e.}\ x\leq 0),\\
        S-s & \qquad\qquad\qquad\quad\mbox{if} \ \tau_1>0 \        (\mbox{i.e.}\ x>s),
      \end{dcases}
    \end{dcases}
  \end{equation*}
and then, recursively for $n\geq 1$,
\begin{equation*}
  \begin{dcases}
  \tau_{n+1}\coloneqq
  \begin{dcases}
    \tau_n +
  \inf\left\{t>0\colon Z^{\tau_n,S}_{\tau_n+t}\leq s\right\} & \mbox{if } \tau_n<\infty\\
\infty &\mbox{otherwise}
\end{dcases}\\[0.4em]
i_{n+1}\coloneqq S-s.
\end{dcases}
\end{equation*}
Note that,
for
$\mathbb{P}$-a.e.\
 $\omega\in \{\tau_n<\infty\}$,
 by continuity of $\mathbb{R}_+\rightarrow \mathbb{R},\ t\mapsto Z^{\tau_n,S}_{\tau_n+t}(\omega)$ and since $S>s$, we have  $\tau_{n+1}(\omega)>\tau_n(\omega)$.
% if $\tau_n(\omega)<\infty$.

\begin{theorem}[Optimal control]\label{verificationtheorem}
Let  Assumptions \ref{ass:bb} and  \ref{Assumption:c0c1} hold.
  % Let $c_1<\alpha_\gamma$ and let $S\in \mathcal{C}$ such that
  % $v'(S)=c_0$.
  Let $x\in \R_{++}$ and consider the control
  $I^*=\{(\tau_n,i_n)\}_{n\geq 1}$ defined
% , with the
%,
 above. 
Then $I^*\in\mathcal I$ and it is optimal for the problem starting at $x$, i.e.,
$J(x,I^*)=v(x)$.
\end{theorem}
\begin{proof}
\emph{Admissibility.}
As noticed above,
$\tau_n< \tau_{n+1}$ $\mathbb{P}$-a.s.\ on $\{\tau_n<\infty\}$.
% $\{\tau_n\}_{n\geq 1}$ is a sequence of strictly increasing stopping times.
Moreover,  for each $n\geq 1$,  $i_n$ is constant; so, as a random variable,  it is trivially $\mathcal{F}_{\tau_n}$-measurable.

% Now, define the auxiliary sequencee $\hat{ \tau}_n\coloneqq \tau_n\wedge 1$, $n\geq 1$.
% Since $s>0$, it can be easily verified by induction that 
% \begin{equation*}
% \hat{\tau}_1\coloneqq
% \begin{dcases}
%   0&\mbox{if $x\leq s$}\\
%   \inf\left\{t\geq 0\colon Z^{0,x}_t\leq s\right\}\wedge 1&\mbox{if $x>s$}
% \end{dcases}
% \qquad
% \hat{ \tau}_{n+1}=\hat{ \tau}_n+\inf\left\{t\geq 0\colon Z^{\hat{ \tau}_n,S\mathbf{1}_{\hat{ \tau}_n<1}}_{\hat{ \tau}_n+t}\leq s\right\}
% \end{equation*}

% \bigskip

Now, for fixed  $\epsilon>0$ such that $S-\epsilon S^2>s$, define the auxiliary sequence $\{\tau^\epsilon_n\}_{n\geq 1}$ of stopping times by
\begin{equation*}
  \tau^\epsilon_1\coloneqq
  \begin{dcases}
    0& \mbox{if $x\leq s$}\\
    \inf
    \left\{ 
      t\geq 0
    \colon Z^{0,x}_t 
    -\epsilon  \left( 
      Z^{0,x}_t 
      +t
       \right) ^2
    \leq s\right\} &\mbox{if $x>s$}
  \end{dcases}
\end{equation*}
and
\begin{equation*}
\tau^\epsilon_{n+1}\coloneqq
 \tau^\epsilon_n+
    \inf
    \left\{ 
      t\geq 0
    \colon 
    Z^{\tau^\epsilon_n,S}_{ \tau^\epsilon_n+t} 
    -\epsilon 
 \left( 
   Z^{\tau^\epsilon_n,S}_{ \tau^\epsilon_n+t} 
   +t
 \right) ^2
    \leq s\right\}\mbox{ for $n\geq 1$}.
\end{equation*}
We notice that
$\tau^\epsilon_n$ is finite and  $\tau^\epsilon_{n+1}> \tau^\epsilon_n$ $\mathbb{P}$-a.s.. % if  $S-\epsilon S^2>  s$.
Moreover,
the random variables
 $\{ \tau^\epsilon_{n+1}- \tau^\epsilon_n\}_{n\geq 1}$
are identically distributed
and
$ \tau^\epsilon_{n+1}- \tau^\epsilon_n$
is independent on $\mathcal{F}_{ \tau^\epsilon_n}$.
Finally, it can be  verified by
 induction that
 \begin{equation*}
   \label{eq:2017-11-15:00}
\lim_{\epsilon\rightarrow 0^+}   \tau^\epsilon_n=\tau_n\qquad \mathbb{P}\mbox{-a.s.\ on\ }\{\tau_n<\infty\},
 \end{equation*}
 from which we obtain 
  \begin{equation}
   \label{eq:2017-11-15:00bis}
\liminf_{\epsilon\rightarrow 0^+}   e^{-\rho \tau^\epsilon_n}\geq e^{-\rho \tau_n}\qquad \mathbb{P}\mbox{-a.s.}.
 \end{equation}
\noindent Define $Y^\epsilon\coloneqq 
    \inf
    \left\{ 
      t\geq 0
    \colon 
    Z^{0,S}_t 
    -\epsilon 
 \left( 
   Z^{0,S}_t
   +t
 \right) ^2
    \leq s\right\}$.
Then $\tau^\epsilon_{n+1}-\tau^\epsilon_n\sim Y^\epsilon$ for all $n\geq 1$.
%  $\{\hat{ \tau}_{n+1}-\hat{ \tau}_n\}_{n\geq 1}$ have the same distribution.
% Moreover, $\hat{ \tau}_{n+1}-\hat{ \tau}_n>0$ $\mathbb{P}$-a.s., because $s<S$.
 % {\color{rose}tutto mi \`e chiaro, tranne: come cacchio si fa a mostrare che, se $X_\cdot$ \`e distribuito come $Y_{\tau +\cdot}$, allora $\int_0^\cdot X_tdW_t$ \`e distribuito come $\int_\tau^{\tau+\cdot}Y_t dW_t$??}
Observe that $Y^\epsilon$ 
increases as $\epsilon$ tends to $0^+$. Let ${\displaystyle{Y\coloneqq \lim_{\epsilon\rightarrow 0^+}Y^\epsilon}}$.
Since $S-\epsilon S^2>s$ entails $Y^\epsilon>0$, we have in particular $Y>0$.
% Notice that by {\color{red}???} we have $\tau_n<\infty$ $\mathbb{P}$-a.s., for $n\geq 1$. {\color{blue} NON CREDO SIA VERO; MA NON MI PARE NEANCHE CHE SERVA}
% Hence,
% by definition of $I^*$ and Markov property of $Z$, it  follows
% that 
% $e^{-\rho (\tau_{n+1}-\tau_n)}$ is independent of $\mathcal{F}_{\tau_n}$ and that the  random variables $e^{-\rho (\tau_{n+1}-\tau_n)}$ are identically distributed
%Let $Y^\epsilon\sim e^{-\rho (\tau}_{n+1}-\hat{\tau_n})}$.
%  and observe that,
% since $\tau_{n+1}-\tau_n>0$, it follows $\mathbb{E}[Y]<1$. 
We can then write, using \eqref{eq:2017-11-15:00bis} and Fatou's Lemma in the first  inequality below,   
\begin{equation}
  \label{eq:2017-11-15:01}
  \begin{split}
      \mathbb{E}
   \left[ 
     e^{-\rho \tau_{n+1}} \right] 
&\leq 
{\liminf_{\epsilon\rightarrow 0^+}}\ \mathbb{E}
 \left[ 
e^{-\rho \tau^\epsilon_{n+1}}
 \right] =
\liminf_{\epsilon\rightarrow 0^+}\ 
  \mathbb{E}
   \left[ 
     e^{-\rho 
       (\tau^\epsilon_{n+1}-\tau^\epsilon_n)}
e^{-\rho \tau^\epsilon_n} \right] 
 \\&=
\liminf_{\epsilon\rightarrow 0^+}
 \ \mathbb{E}
   \left[ 
\mathbb{E}
 \left[ 
\left.
     e^{-\rho 
       (\tau^\epsilon_{n+1}-\tau^\epsilon_n)}
e^{-\rho \tau^\epsilon_n}\right|\mathcal{F}_{\tau^\epsilon_n} \right] 
 \right] 
=
\liminf_{\epsilon\rightarrow 0^+}\
 \left( \mathbb{E}
   \left[ 
     e^{-\rho 
       (\tau^\epsilon_{n+1}-\tau^\epsilon_n)}
 \right] 
\mathbb{E}
 \left[ 
e^{-\rho \tau^\epsilon_n}
 \right] \right)
  \\&
=
% \lim_{\epsilon\rightarrow 0^+}
\liminf_{\varepsilon\to 0^+}\
 \left( 
\mathbb{E} \left[ e^{-\rho Y^\epsilon} \right]
  \mathbb{E}
   \left[ 
     e^{-\rho 
       \tau^\epsilon_n}
 \right] \right) \ \stackrel{\mbox{\footnotesize{(by induction)}}}{=}\ 
% \lim_{\epsilon\rightarrow 0^+}
\liminf_{\varepsilon\to 0^+}\  \left(\mathbb{E} \left[ e^{-\rho Y^\epsilon} \right]  \right) ^n
  \mathbb{E}
   \left[ 
     e^{-\rho 
       \tau^\epsilon_1}
 \right]  \\&
\leq
 \left(\mathbb{E} \left[ e^{-\rho Y} \right]  \right) ^n.
\end{split}
\end{equation}
% \begin{equation}
%   \label{2017-06-06:00}
%   \mathbb{E} \left[ e^{-\rho\tau_{n+1}} \right] =
%    \mathbb{E} \left[ \E\left[\left.e^{-\rho(\tau_{n+1}-\tau_n)}e^{-\rho \tau_n}\right|\mathcal{F}_{\tau_n} \right]\right]=\E\left[e^{-\rho(\tau_{n+1}-\tau_n)}\right]\ \E\left[e^{-\rho \tau_n}\right], \ \ \ \ \  n\geq 1. 
% %  \sum_{n\geq 2} \left( \mathbb{E}[Y] \right) ^{n-2} 
% \end{equation}
% By induction we get from \eqref{2017-06-06:00}
% \begin{equation}\label{rra}
% \mathbb{E} \left[ e^{-\rho\tau_{n+1}} \right]=
%  \mathbb{E}[Y]^{n} \E\left[e^{-\rho \tau_1}\right], \ \ \ \ n\geq 1.
%   \end{equation}
  Summing over $n\geq 1$ and taking into account that $\E[e^{-\rho Y}]<1$,
from \eqref{eq:2017-11-15:01}
 we get
  \begin{equation}\label{pppk}
  \mathbb{E} \left[ \sum_{n\geq 1}e^{-\rho\tau_{n+1}} \right]<\infty. 
  \end{equation}
Both conditions \eqref{eq:2015-11-23:00} and \eqref{eq:2015-12-18:01} follow from \eqref{pppk},
so the control $I^*$ is admissible.
\smallskip

\noindent \emph{Optimality.}
Set $X^*\coloneqq X^{x,I^*}$.
We observe  that,
by 
\eqref{estimateB},
\eqref{eq:2017-09-18:00},
% \eqref{eq:2017-03-25:03},
and  \eqref{pppk}, we have
\begin{equation}\label{lim}
\lim_{T\rightarrow\infty} \E\left[e^{-\rho T}v(X^*_{T})\right]=0.
\end{equation}
% and let
%$$N=\sup\{n\geq 0 \ | \ \tau_n<+\infty\}.$$
%Set $\tau_0\coloneqq 0$ and
 Let $T>0$ and set $\tau_0\coloneqq 0^-$. Observe that by definition $X^*\in [s,+\infty)$ and recall that
 $\mathcal{L} v=f$ on $\mathcal{C}=(s,\infty)$.
% The function $v$ is of  class  $C^1(\mathbb{R}_{++})\cap C^2(\mathbb{R}_{++}\setminus\{s\})$; moreover,
% the first and second derivative are uniformly continuous
% near $s$.
 % {\color{red}(perch\'e $\varphi$ e $\hat{ v}$ sono soluzioni $C^2$ su tutto $\mathbb{R}_{++}$)}.
For all $n\in\N$ we apply It\^o's formula to $v(X^*)$ in the interval $[\tau_n\wedge T,\tau_{n+1}\wedge T)$. Note that $v'$ is bounded in $[s,\infty)$ by Proposition \ref{structv'},
so 
$$\E\left[\int_{\tau_n\wedge T}^{\tau_{n+1}\wedge T} v'(X^*_t)dW_t\right]=0\qquad \forall n\in\N.$$
Hence, taking the expectation in the It\^o formula and taking into  account that $\mathcal{L}v(X^*)=f(X^*)$, we get 
\begin{equation}\label{Ito2}
  \begin{multlined}[c][.9\displaywidth]
      \E\left[e^{-\rho (\tau_{n+1}\wedge T)}v(X^*_{(\tau_{n+1}\wedge T)^-})\right]
  -
  \E\left[ e^{-\rho (\tau_n\wedge T)}v(
    X^*_{\tau_n\wedge T})\right]
  =
  -\E\left[\int_{\tau_n\wedge T}^{\tau_{n+1}\wedge T}e^{-\rho t}f(X^*_t)dt\right], \ \   \forall n\in\N.
\end{multlined}
\end{equation}
Now fix for the moment $\omega\in\Omega$, $n\geq 1$, and assume that $\tau_{n}(\omega)\leq T$.  By definition of $i_n(\omega)$ and considering that $X^{*}_{\tau_n^-}(\omega)\in \mathcal{A}$ 
we have (cf.\ also 
Corollary~\ref{Lemma1}, Proposition~\ref{propempty}\emph{(\ref{2016-11-06:06})},
and the definition of $S$
in
Proposition~\ref{structv'}\emph{(\ref{2017-05-30:01})})
 $$
i_n(\omega)= \mathop{\operatorname{argmax}}_{i>0}  \left\{v(X^*_{\tau_n^-}(\omega)+i)-c_0 i-c_1\right\}.
$$
Hence, considering that $\mathcal{M} v(X^{*}_{\tau_n^-}(\omega))=v(X^{*}_{\tau_n^-}(\omega))$, we have
%{\color{red}(scrivere prima l'ugaguaglianza per ogni $\omega$)}
\begin{equation}\label{Ito3bis}
\hskip-0.20em
e^{-\rho \tau_{n}(\omega)}v(
  X^*_{\tau_{n}(\omega)}
  ) -
e^{-\rho \tau_{n}}v(
  X^*_{\tau_{n}(\omega)^-}
  ) =e^{-\rho \tau_{n}(\omega)}(c_0i_{n}(\omega)+c_1).
\end{equation}
It follows that, for all $n\geq 1$,
\begin{equation}\label{Ito3}
 \begin{multlined}[c][.9\displaywidth]
    \E\left[e^{-\rho (\tau_{n}\wedge T)}
 \left( 
v(
  X^*_{\tau_{n}\wedge T}
  )
-v(
  X^*_{(\tau_{n}\wedge T)^-}
  )
 \right) \right] =\\
=\E\left[e^{-\rho (\tau_{n}\wedge T)}
 \left( 
v(
  X^*_T
  )
-v(
  X^*_{T^-}
  )\right) \mathbf{1}_{\{\tau_n>T\}} \right]+
\E\left[e^{-\rho \tau_{n}}(c_0i_{n}+c_1)\mathbf{1}_{\{\tau_{n}
\leq T\}}\right].
\end{multlined}
\end{equation}
Using \eqref{Ito2} and \eqref{Ito3}, we can then write, for $N\geq 1$,
\begin{equation*}
  \begin{split}
    \mathbb{E} &\left[ e^{-\rho (\tau_{N+1}\wedge T)} v(X^*_{\tau_{N+1}\wedge T}) \right] 
-
v(x)
=
\sum_{n=0}^N
\mathbb{E}
 \left[ 
e^{-\rho (\tau_{n+1}\wedge T)}
v(X^*_{\tau_{n+1}\wedge T})
-
e^{-\rho (\tau_n\wedge T)}
v(X^*_{\tau_n\wedge T}) \right] \\
=&
\sum_{n=0}^N
\mathbb{E}
 \left[ 
e^{-\rho (\tau_{n+1}\wedge T)}
 \left( 
v(X^*_{\tau_{n+1}\wedge T})
-
%e^{-\rho (\tau_{n+1}\wedge T)}
v(X^*_{(\tau_{n+1}\wedge T)^-}) \right) \right] \\
&
+ 
\sum_{n=0}^N
\mathbb{E} \left[ e^{-\rho (\tau_{n+1}\wedge T)} v(X^*_{
    (\tau_{n+1}\wedge T )^- }) - e^{-\rho (\tau_n\wedge T)}
  v(X^*_{\tau_n\wedge T}) \right]\\
=&\sum_{{{n=0}}}^N \left( \E\left[e^{-\rho (\tau_{n+1}\wedge T)} \left(
      v( X^*_T ) -v( X^*_{T^-} )\right) \mathbf{1}_{\{\tau_{n+1}>T\}}
  \right]+ \E\left[e^{-\rho
      \tau_{n+1}}(c_0i_{n+1}+c_1)\mathbf{1}_{\{\tau_{n+1}
      \leq T\}}\right] \right) \\
&- \sum_{n=0}^N \mathbb{E}\left[ 
  \int_{\tau_n\wedge
    T}^{\tau_{n+1}\wedge T}e^{-\rho t} f(X^*_t)dt \right].
\end{split}
\end{equation*}
By passing to the limit $N\rightarrow \infty$ and using \eqref{eq:2015-11-23:00}, we obtain
 \begin{align*}
   & \mathbb{E} \left[ e^{-\rho
 T} v(X^*_T) \right] 
-
v(x) 
+
 \mathbb{E}\left[ 
  \int_{0}^T e^{-\rho t} f(X^*_t)dt \right]
\\=&
 \sum_{{{n=0}}}^\infty \left( \E\left[e^{-\rho (\tau_{n+1}\wedge T)} \left(
      v( X^*_T ) -v( X^*_{T^-} )\right) \mathbf{1}_{\{\tau_{n+1}>T\}}
  \right]+ \E\left[e^{-\rho
      \tau_{n+1}}(c_0i_{n+1}+c_1)\mathbf{1}_{\{\tau_{n+1}
      \leq T\}}\right] \right).
\end{align*}
 We take now the ${\displaystyle\liminf_{T\rightarrow \infty}}$, using 
\eqref{lim} on the first addend of the left hand side, monotone convergence on the third addend of the left hand side, and 
 Fatou's lemma on the right hand side. We obtain 
\begin{equation}\label{2017-11-16:01}
-v(x)
+\mathbb{E}
 \left[  
   \int_{0}^\infty e^{-\rho t} f(X^*_t) dt \right] \geq
\sum_{n=0}^\infty 
\mathbb{E}
 \left[ 
   e^{-\rho \tau_{n+1}}
   (c_0i_{n+1}+c_1) \right],
\end{equation}
%Notice that,
% by the very definition of $\tau_1$,
% \begin{equation}
%   \label{eq:2017-11-16:04}
%   \mathbb{P}(\tau_1>0)>0
%\Longrightarrow 
%x>s
%\mbox{ and }
%\mathbb{P}(\tau_1>0)=1,
%\end{equation}
%hence,
%using again $\mathcal{L}v=f$ on $(s,\infty)$, we have
%\begin{equation}\label{2017-11-16:02}
%  \mathbb{E} 
% \left[  v(X^*_{\tau_1}) \right] = 
%      v(x)\mathbb{P}( \tau_1>0)+
%\mathbb{E}
% \left[ 
%   \int_0^{\tau_1}
%e^{-\rho t}f( X^*_t)dt
% \right] 
%+      v(S) \mathbb{P}( \tau_1=0).
%\end{equation}
%Moreover, by
%\eqref{Ito3} with $n=1$,
% % on $\{\tau_1=0\}$,
% we have
%\begin{equation}
%  \label{eq:2017-11-16:00}
%  v(S)=v(x)+(c_0i_1+c_1)
%\quad \mbox{on}\quad
%\{\tau_1=0\} .
%\end{equation}
%By
%\eqref{2017-11-16:01}
%%
%\eqref{eq:2017-11-16:04},
%\eqref{2017-11-16:02},
%and
%\eqref{eq:2017-11-16:00}, we conclude
%\begin{equation*}
%\mathbb{E}
% \left[  
%   \int_0^\infty e^{-\rho t} f(X^*_t) dt  -
%\sum_{n=1}^\infty 
%   e^{-\rho \tau_n}
%   (c_0i_n+c_1) \right] 
%\geq v(x),
%\end{equation*}
which shows that $I^*$ is optimal.
\end{proof}

\begin{figure}
\centering
\includegraphics[width=11cm, height=7cm]{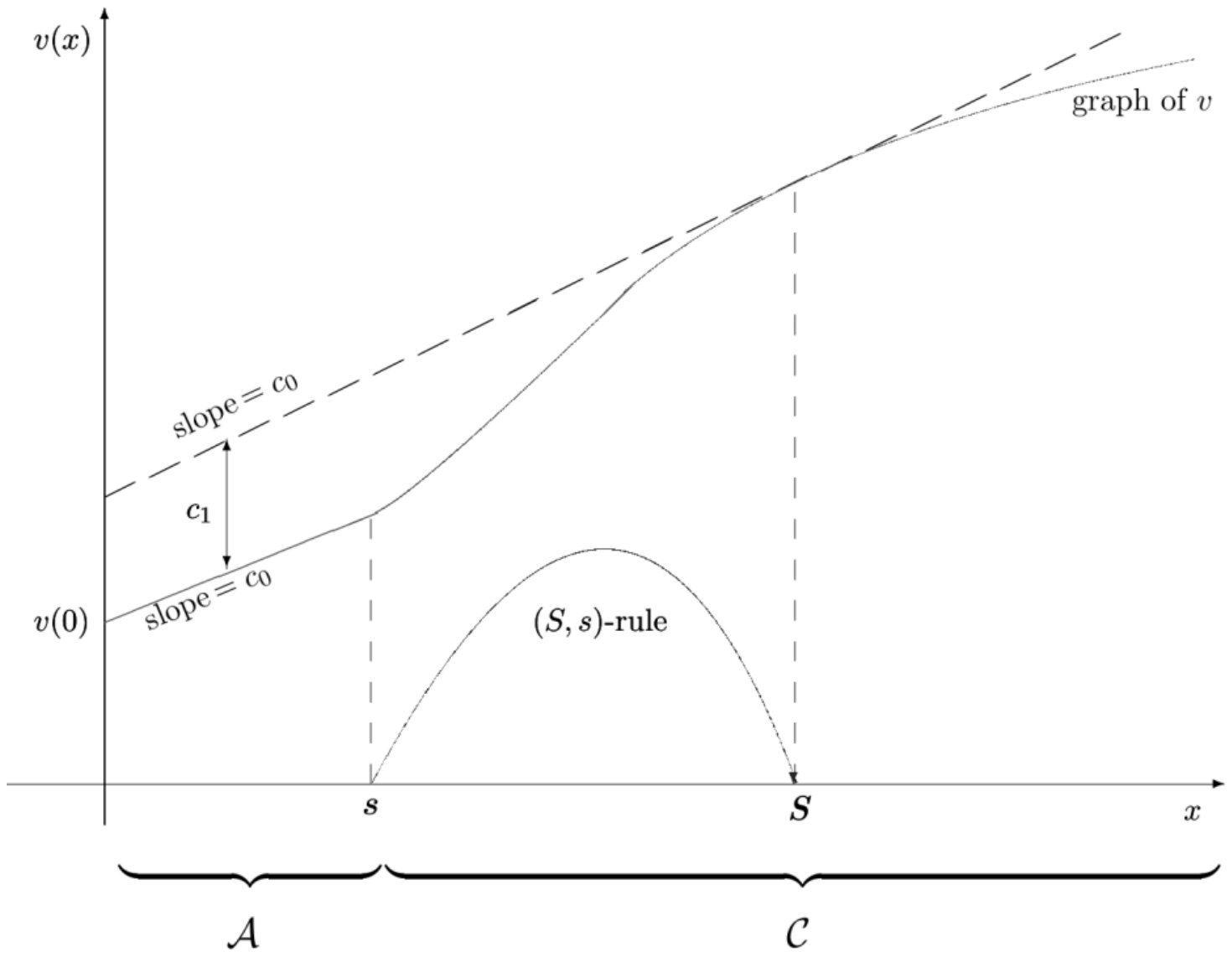}
\caption{An illustrative picture of the value function and of the $S-s$ rule.}\label{fig:1}
\end{figure}
\newpage
\section{Numerical illustration in the linear case}\label{sec:num}

In the  previous sections we have characterized the solution of the dynamic optimization problem through the unique solution of the nonlinear algebraic system \eqref{smoothbase} in the triple $(A,s,S)$.
In this section we specialize the study when the reference process $Z$ follows a 
 geometric Brownian motion dynamics, i.e.\ when
$    b(x)\coloneqq  
      \nu  x$,
$    \sigma(x)\coloneqq  
      \sigma x$,
% $b(x)=\nu  x\mathbf{1}_{\{x>0\}}$,  % (w
% ith abuse of notation)
% $\sigma (x)=\sigma x \mathbf{1}_{\{x>0\}}$,
with $\nu \in\R$,  $\sigma>0$, and when $f(x)=\frac{x^\gamma}{\gamma}$, with $0<\gamma<1$, 
assuming 
\begin{equation}\label{aasro}
\rho>\nu^+.
\end{equation}
In this way, %$f$ satisfies Assumption 
Assumptions~\ref{eq:2017-03-23:00},
\ref{ass:bb},
\ref{ass:rho}, \ref{2016-11-06:00}, 
\ref{ass:ass3}(\ref{all})--(\ref{2017-09-25:00}) are satisfied
 (\footnote{Actually, we should consider $b(x)=\nu x$ if $x>0$ and $b(x)=0$ otherwise, and similarly for $\sigma$, in order to fit
Assumption~\ref{eq:2017-03-23:00}. But this does not matter because 
our controlled process lies in $\mathbb{R}_{++}$.}).
% In this case, Assumption \ref{ass:rho} is verified with $\rho>\nu ^+$.
%Dealing under the latter condition
% and assuming also that $\nu <\frac{\sigma^2}{2}$, 
In the present case we have 
$$
\varphi (x)= x^m, 
$$
where $m$ is  the negative  root of the characteristic equation 
$$
 \rho-\nu  m-\frac{1}{2}\sigma^2 m(m-1)=0$$
associated with $\mathcal{L}u=0$,
 i.e.\
 %$$
%{\color{red} m_\pm=\, \left(\frac{1}{2}+ \frac{\rho}{\sigma^2}\right) \pm \sqrt{ \left(\frac{1}{2}+ \frac{\rho}{\sigma^2}\right)^2+\frac{2\rho}{\sigma^2}}\qquad \textrm{sbagliato, dato che manca $\nu $}}.
%$$
\begin{equation}\label{eq:m}
m= \left(\frac{1}{2}- \frac{\nu }{\sigma^2}\right) - \sqrt{ \left(\frac{1}{2}- \frac{\nu }{\sigma^2}\right)^2+\frac{2\rho}{\sigma^2}},
\end{equation}
 and 
% $\hat v$ be the function defined by
\begin{equation}
  \label{costgamma}
\hat v(x)= C_\gamma \frac{x^\gamma}{\gamma},  \ \ \ \ C_\gamma\coloneqq  \left(\rho-\nu \gamma+\frac{1}{2}\gamma(1-\gamma)\sigma^2\right)^{-1}.
\end{equation}   
The problem with no fixed cost, i.e.\ when $c_1=0$,  is investigated in the singular control setting (the right one to get existence of optimal controls, see Remark \ref{rem:cost}) in \cite[Sec.\,4.5]{P}. In this case, the value function $v$ and the optimal reflection boundary $s$ are characterized in \cite[Th.\,4.5.7]{P} through an algebraic system too. Such system can be solved providing, in our notation, 
\begin{equation}\label{ssing}
s=\left(\frac{c_0(m-1)}{C_\gamma(m-\gamma)}\right)^{\frac{1}{\gamma-1}},  \ \ \ \ 
 B= \frac{C_\gamma(1-\gamma)}{m(m-1)} s^{\gamma-m}.
\end{equation}
 We make Assumption~\ref{Assumption:c0c1}; the latter in the present case reads as
\begin{equation}\label{costgamma2}
  c_1<
C_\gamma^{\frac{1}{1-\gamma}}
c_0^{\frac{\gamma}{1-\gamma}}
 \left( 
   \frac{1}{\gamma}-1
 \right).
\end{equation}
 {\red{
Moreover, Assumption~\ref{ass:ass3}(\ref{2017-09-25:01}) would read as
 $$\rho>\max\left\{4|\nu|+6\sigma^2,\ 2|\nu|+2\sigma^2\right\}= 4|\nu|+6\sigma^2.$$ However,  as we show below, in the linear-homogeneous case under consideration here, we do not need to make
this assumption:
% (\ref{all})--(\ref{2017-09-25:00})
we can exploit  the linear dependence of the controlled process  on the  initial datum and the homogeneity of $f$ to show the result of semiconvexity stated,  for the general case, in 
Proposition~\ref{prop:semiconvex}.
% of Section \ref{sec:pre}
Consequently, the other results of the paper  hold under no further assumption.
Indeed,
 observing that the
terms $\{i_n\}_{n\geq 1}$ 
 enter in the dynamics of $X^{x,I}$ in additive form, we have
 \begin{equation}
   \label{flopro}
X_t^{x,I} - X_t^{y,I} = X_t^{x,\emptyset} - X_t^{y,\emptyset}=(x-y)e^{(\nu -\frac{\sigma^2}{2})t+\sigma W_t}, \ \ \ \forall I\in\mathcal{I}, \ \forall x,y\in \R_{++},
\end{equation}
that we can use to prove  the following result.}}
\begin{Proposition} In the above framework we have, for every $\lambda\in[0,1]$, and every $x,y\geq \varepsilon>0$
\begin{equation*}  
   v(\lambda x+(1-\lambda)y) -\lambda v(x)-(1-\lambda)v(y)  \leq  \lambda(1-\lambda)(1-\gamma)C_\gamma^{-1}\varepsilon^{\gamma-2}(y-x)^2.
\end{equation*}
\end{Proposition}  
\begin{proof}
Let $0< \xi\leq \xi'$. Then,  for suitable $\eta,\eta'\in[\xi,\xi']$ we have, by Lagrange's Theorem,
\begin{equation}\label{semconv}
\begin{split}
f(\lambda\xi+(1-\lambda)\xi')- \lambda f(\xi)-(1&-\lambda)f(\xi')=\\
=& -\lambda [f(\xi)-f(\xi+(1-\lambda)(\xi'-\xi))]-(1-\lambda)[f(\xi')-f(\xi'+\lambda(\xi-\xi'))]\\
=&\lambda(1-\lambda)  f'(\eta)(\xi'-\xi)-\lambda(1-\lambda)f'(\eta')(\xi'-\xi)\\
=&\lambda(1-\lambda)  \left( f'(\eta)-f'(\eta') \right) (\xi'-\xi)\\
\leq &\lambda(1-\lambda) |f''(\xi)|(\xi'-\xi)^2\\
= &\lambda(1-\lambda) (1-\gamma)\xi^{\gamma-2} (\xi'-\xi)^2.
\end{split}
\end{equation}
  Let now $0<\varepsilon\leq  x\leq y$, $\lambda\in[0,1]$, and set $z\coloneqq \lambda x+(1-\lambda)y$. Let $\delta>0$ and let $I_\delta\in\cali$ be a $\delta$-optimal control for $v(z)$. Then, using \eqref{semconv}, the fact that $X^{x,I}\geq X^{x,\emptyset}$, 
% \eqref{flopro} $X^{y,N}\geq X^{x,N}$, and $X^{x,I}\geq X^{x,N}$,  
 and recalling \eqref{flopro}, we get 
\begin{equation*}
  \begin{split}
   v(\lambda x+(1-\lambda)y)-\delta -\lambda v(x)-(1-\lambda)&v(y) \leq J(z,I_\delta)-  \lambda J(x,I_\delta)-(1-\lambda)J(y,I_\delta)\\
  = &\E\left[\int_0^{+\infty} e^{-\rho t} \left(f(X_t^{z,I_\delta})-	\lambda f(X_t^{x,I_\delta})-(1-\lambda) f(X_t^{y,I_\delta})\right)dt\right]\\
  \leq &\lambda(1-\lambda)(1-\gamma)\E\left[\int_0^{+\infty} e^{-\rho t} (X^{x,I_\delta}_t)^{\gamma-2}(X^{y,I_\delta}_t-X^{x,I_\delta}_t)^2 dt\right]\\
    \leq &\lambda(1-\lambda)(1-\gamma)\E\left[\int_0^{+\infty} e^{-\rho t} (X^{x,\emptyset}_t)^{\gamma-2}(X^{y,\emptyset}_t-X^{x,\emptyset}_t)^2 dt\right]\\
    =& \lambda(1-\lambda)(1-\gamma)C_\gamma^{-1}x^{\gamma-2}(y-x)^2 \leq  \lambda(1-\lambda)(1-\gamma)C_\gamma^{-1}\varepsilon^{\gamma-2}(y-x)^2,
\end{split}
\end{equation*}
the claim.
\end{proof}

\subsection{Numerical illustration}
We perform a numerical analysis of the solution solving  the nonlinear system \eqref{smoothbase}.
% by the \texttt{fminsearchcon} algorithm in Matlab.
% which finds the minimum of a multivariate function with bound constraints using a derivative-free method. {\color{red}CAPIRE COME SISTEMARE LA FRASE PRECEDENTE.}.
%{\color{cyan}Io resto dell'opinione che la specificazione del linguaggio e della funzione che usiamo per la risoluzione numerica \`e ridondante: la dovreemmo dare solo se ce lo chiedono. Anche perch\'e: 1) questa parte numerica \`e una ``illustrazione'', quindi il suo grado di verit\`a/attendibilit\`a \`e deel tutto indipendente da quanto prima sviluppato; 2) se il paper tocca in mano a un numerico, e questo ritiene che \texttt{fminsearchcon} sia una funzione del cacchio, magari ci rompe e ci chiede di farlo in altro modo. Ergo: sorvolerei del tutto sul codice usato.}
%\red{In all the simulations the conditions \eqref{aasro} and \eqref{costgamma2} are satisfied.}
%{\color{cyan}DOBBIAMO VERIFICARE CHE I PARAMETRI SCELTI VERIFICHINO L'ASSUNZIONE~\ref{Assumption:c0c1}}
 In Figure~\ref{dervalfun}, we provide the picture of the value function and its derivative when the parameters are set as follows:
{{$\rho = 0.08, \
\nu  = - 0.07, \ \sigma = 0.25, \ 
c_0 = 1, \ 
c_1 = 10, \ \gamma = 0.5.$}} 
Solving \eqref{smoothbase} with
 these entries and with $\varphi(x)=x^m$, where $m$ is given by \eqref{eq:m}, 
 yields 
$$(B,s,S)= (97.0479,  \ 8.7492, \ 56.9930).$$ 
%The figure shows that the smooth-fit principle holds at $s=8.7492$. 
\vspace{-1.2cm}
\begin{figure}[htbp]
    \centering
    \includegraphics[height=7cm,width=11cm]
{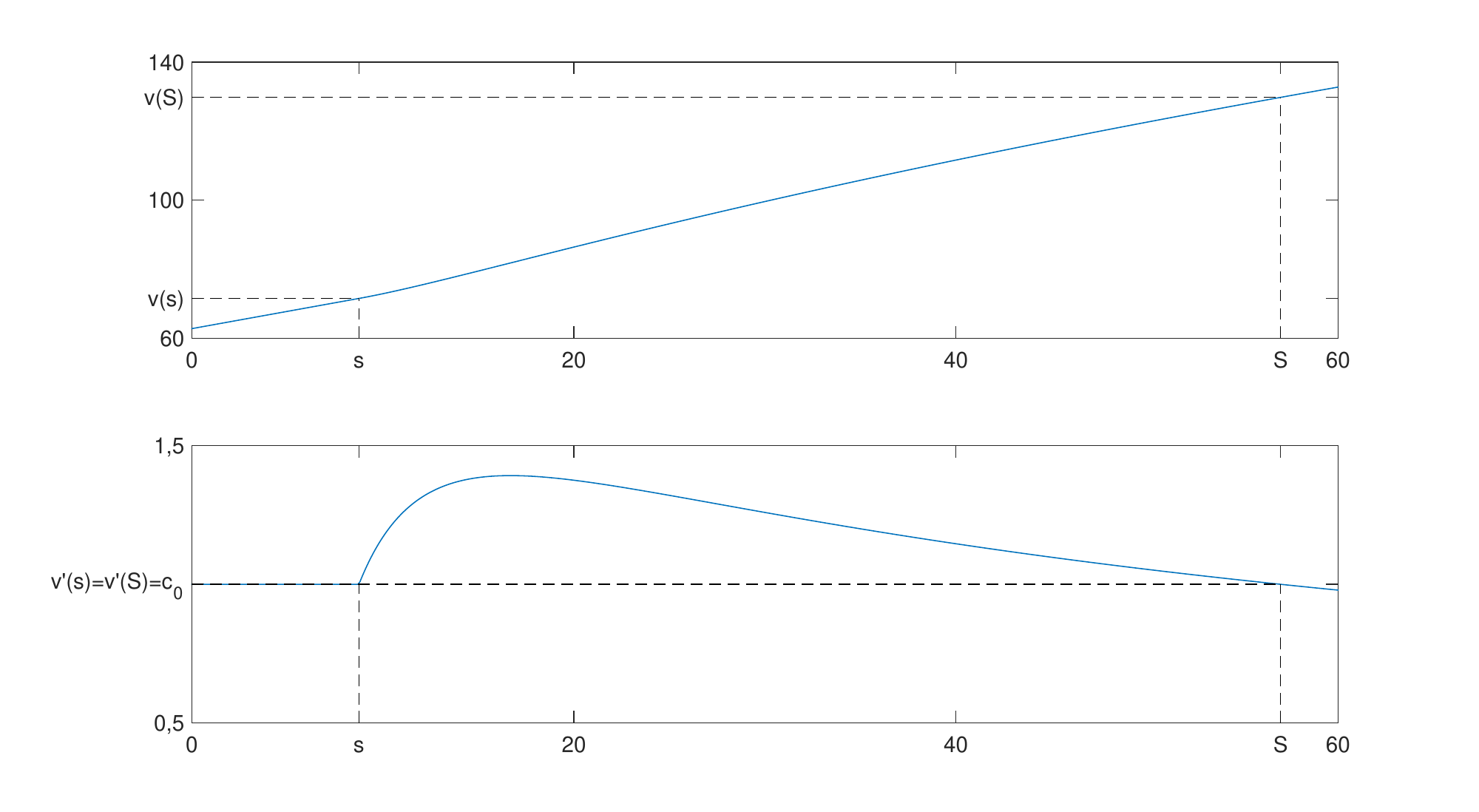}
%{valufunctionandderivativebis.eps}
    \caption{Value function (above) and its derivative (below)}\label{dervalfun}
\end{figure}

In the rest of this section we discuss numerically the solution, illustrating how changes in parameters affect the value function and the trigger and target  boundaries $s,S$, which describe the optimal control\,(\footnote{  
The simulations are done for negative values of $\nu$, thinking of it as a depreciation factor. We omit, for the sake of brevity, to report the simulations that we have performed for positive values of  $\nu$, as the outputs show the same qualitative behaviour as in the case of negative $\nu$.}).

%In the  previous sections we have characterized the solution of the dynamic optimization problem through the (unique) solution of the nonlinear algebraic system \eqref{smoothbase} in the triple $(A,s,S)$. In the present section we perform several numerical experiments  to show how our solution can be implemented numerically to analyze the impact on the value function and on the optimal control of the most relevant parameters: the volatility $\sigma$ and the fixed cost $c_1$. 

%In this section we want to analyze the behaviour of the value function with respect to small changes of the parameters of the model. We will study the equation that characterizes the value function in order to compute these comparative statics analytically. The effect of any parameter $\theta$ on the value function can be found by differentiating this expression with respect to $\theta$.
%{\color{blu} We start calculating} the rate of change of the
%value function with respect to the volatility of the fluctuations of the capacity, that is, $\sigma$. 
\subsubsection{Impact of volatility}

%{\color{cyan}MR: check numerical values fine}

In Table \ref{Conregion} we report the relevant values the solution for different values of the volatility $\sigma$. 
{{The other parameters are set as follows: $\rho= 0.08, \ \nu  = -0.07,\ \gamma = 0.5,\ c_0 = 1,\ c_1 = 10.$}}
\vspace{-.4cm}
{\small{
\begin{table}[htbp]
\centering \caption{Solution as function of $\sigma$.
%$\rho= 0.08,\nu  = 0.07,\gamma = 0.5,c_0 = 1,c_1 = 10.$
}\label{Conregion}
\medskip
\begin{tabular}{l c c c c c c c c}\hline\hline
\multicolumn{1}{c}{$\sigma$} &\textbf{$B$}& \textbf{$s$}
& \textbf{$S$} & \textbf{$S-s$}& \textbf{$v(0)$}& \textbf{$v(s)$}&\textbf{$v(S)$}\\ \hline 
1\% & 349.2820 & 14.6488& 69.1073 &  54.4584 &68.2325& 82.8813 &147.3398\\
5\%& 313.6460 &14.2670 & 68.4774& 54.2104 &68.0298 &   82.2968 &146.5072\\
10\%& 238.6460 &  13.2168&   66.6426& 53.4258 & 67.3856  & 80.6024 &144.0282\\
15\%&172.6459 &  11.8029 &   63.9264&  52.1235 &66.2914&  78.0943&  140.2178\\
20\%& 126.9781& 10.2646& 60.6291 & 50.3644&64.7453&75.0099&135.3743\\
25\%& 97.0479 & 8.7492 &  56.9930&48.2438&62.7645& 71.5137&  129.7575\\
30\%& 77.1043 & 7.3358&  53.2006&45.8648& 60.3826& 67.7184&123.5832
 \\
\hline\end{tabular}
\end{table}
}}\\
%
%In Figure \ref{vfdiffsigma1} we show the different plots (corresponding to different colours) of the value function corresponding to the data of Table~\ref{Conregion}.
%%{\color{blue} Bisognerebbe mettere anche le $s$ in colore (e far capire che sono minuscole.}
%% corresponding to the different choice of the volatility parameter $\sigma$, and the other parameters set as in the table.
%%Finally figure \ref{lengthActCon} shows that the length of $\mathcal{A}$ decreases faster than the length of $(s,S).$
%
%% 	\begin{figure}[htbp]
%%    \centering
%%    \includegraphics[height=8cm,width=14cm]{volatilitydiffzoom.eps}\label{vfdiffsigmaconzoom}
%%    \caption{Value function as a function of the volatility $\sigma$.}
%% \end{figure}
%			
%			
%\begin{figure}[htbp]
%    \centering
%    \includegraphics[height=8cm,width=14cm]{ultimaVOL.eps}
%    \caption{Value function for several values of the volatility parameter $\sigma$.}\label{vfdiffsigma1}
%\end{figure}
% {\color{cyan}
% ATTENZIONE l'assuzione~\ref{Assumption:c0c1} 
% non viene verificata coi dati della }\\
\noindent Figure \ref{sSdifvola}, {{drawn imposing the same values of parameters}},
 represents the trigger level $s$, the target level $S$, and their difference $S-s$ as functions  of the volatility $\sigma$. The figure and the table show that, when uncertainty increases,
 the action region $\mathcal{A}$ shrinks and
the investment size $S-s$ shrinks.  The first effect is well-known in the economic literature of irreversible investments without fixed costs as \emph{value of waiting to invest}: an increase of uncertainty leads to postpone the investment (see \cite{MS}). We can see that, in our fixed cost context, also the size of the optimal investment is negatively affected by an increase of uncertainty.

\begin{figure}[htbp]
    \centering
    \includegraphics[height=7cm,width=12cm]
{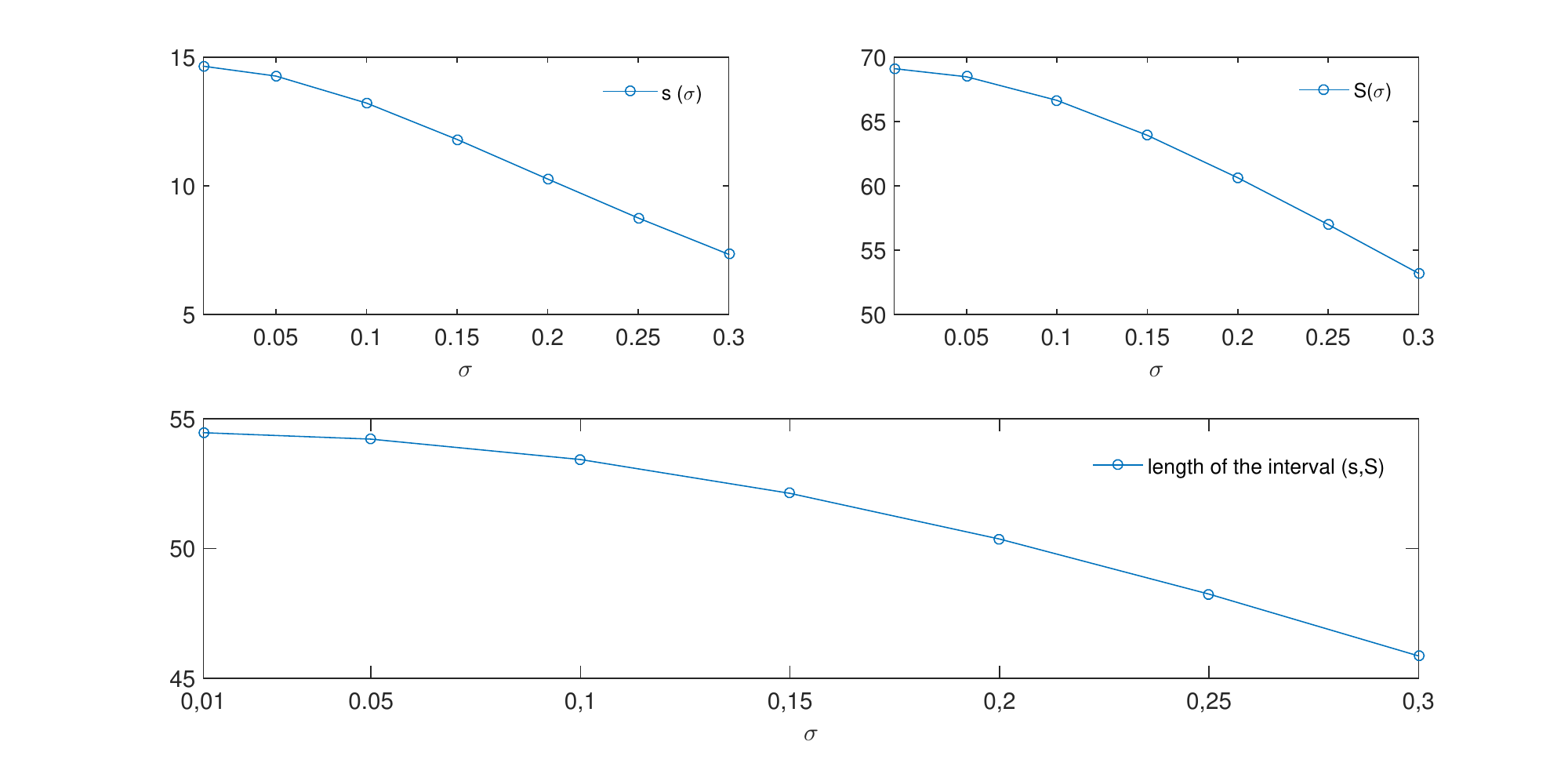}
%{sandSanddiffdiffvol.eps}
\vspace{-.2cm}    \caption{The trigger level $s$, the target level  $S$, and the difference $S-s$ as functions of $\sigma$.}\label{sSdifvola}
\end{figure}

\subsubsection{Impact of fixed cost}\label{sec:fixedcost}
In Table \ref{valudifc1} we report the relevant values of the solution for different values of the fixed cost $c_1$, when the other parameters are set as follows: $\sigma=0.1,\ \rho= 0.08,\ \nu  = - 0.07,\ \gamma = 0.5,\ c_0 = 1.$
In the row corresponding to $c_1=0$, there are reported the outputs of the corresponding singular control problem, computed  according to the values of $s$ and $B$ expressed by \eqref{ssing}(\footnote{In this case the  optimal control consists in a reflection policy at a boundary; in other terms the interval $[s,S]$ degenerates in a singleton $\{s\}=\{S\}$.}). It can be observed that the convergence as $c_1\to 0^+$ is pretty slow; this is consistent with the theoretical result of \cite{OUZ}, which would state, in our case, $\frac{\partial v (\cdot;c_1)}{\partial c_1}(0^+)=-\infty$. 
% {\color{cyan}
% ATTENZIONE, con riferimento alla
% tabella~\ref{valudifc1} e ai grafici su di essa basati,
%  l'assuzione~\ref{Assumption:c0c1} 
% \`e verificata solo per $c_1\leq 35$.
% }
{\small{
\begin{table}[htbp]
\centering 
\caption{Solution as function of $c_1$.}\label{valudifc1}\medskip
\begin{tabular}{l c c c c c c c}\hline\hline
\multicolumn{1}{c}{\textbf{$c_1$}} & \textbf{$B$} & \textbf{$s$}
& \textbf{$S$} & \textbf{$S-s$}& \textbf{$v(0)$}& \textbf{$v(s)$}&\textbf{$v(S)$}\\ \hline 
0   &577.5165& 41.6233&  41.6233   & 0       &83.2470&  124.8703&124.8703\\
0.01&573.1240&38.6466  & 44.5649  & 5.9182  & 83.1362   &  121.7828  &  127.7110 \\
0.5 &519.9311&30.6195  & 52.1522 & 21.5328  &  81.5607   &  112.1802 & 134.2129\\
1   & 487.9211 &27.7903  & 54.7042 & 26.9139 & 80.4620    &  108.2523  &  136.1663 \\
10  & 238.6460 &13.2168  & 66.6426 & 53.4258  &  67.3856       &  80.6024  &  144.0282 \\
30  & 57.6611 & 4.2696  & 72.3953  &  68.1257	 & 44.7847  &  49.0543 & 147.1800 \\
50  & 7.9037  & 1.0275  & 73.7826  &  72.7551  & 24.1040   &  25.1315 & 147.8866\\
%70& 0.0326  & 0.0246  & 73.9961  &  73.9715  &  3.9974   &   4.0220 & 147.9935\\
\hline\end{tabular}
\end{table}
}}\\
%
%\begin{table}[htbp]
%\centering 
%\caption{
%Solution as function of $c_1$.%$\sigma= 8\,\ \rho= 0.02,\ \nu  = 0.07,\ \gamma = 0.5,\ c_0 = 1$
%}\label{valudifc1}\medskip
%\begin{tabular}{l c c c c c c c}\hline\hline
%\multicolumn{1}{c}{\textbf{$c_1$}} & \textbf{$A$} & \textbf{$s$}
%& \textbf{$S$} & \textbf{$S-s$}& \textbf{$v(0)$}& \textbf{$v(s)$}&\textbf{$v(S)$}\\ \hline 
%20& 120.3432& 7.4955  & 70.5439  &  63.0484  & 55.6617   &  63.1572 & 146.2055 \\
%30& 57.6611 & 4.2696  &72.3953   &  68.1257	 & 44.7847   &  49.0543 & 147.1800 \\
%40& 24.2324 & 2.2677  & 73.3346  &  71.0670  & 34.3259   &  36.5936 & 147.6606 \\
%50& 7.9037  & 1.0275  & 73.7826  &  72.7551  & 24.1040   &  25.1315 & 147.8866\\
%60& 1.4736  & 0.3228  & 73.9571  &  73.6343  & 14.0169   &  14.3397 & 147.9741\\
%70& 0.0326  & 0.0246  & 73.9961  &  73.9715  &  3.9974   &   4.0220 & 147.9935\\
%\hline\end{tabular}
%\end{table}

\noindent
 Figure \ref{vva}, {{drawn imposing the same values of parameters}}, shows that, as $c_1$ increases,
 the action region $\mathcal{A}$  shrinks
and the investment size $S-s$ expands. Both these effects are expected: the first one is the counterpart of the value of waiting to invest, now with respect to the fixed cost of investment, rather than with respect to  uncertainty; the second one expresses the fact that an increase of the fixed cost leads to invest less often, then to provide a larger investment size when the investment is undertaken.  
\begin{figure}[htbp]
    \centering
    \includegraphics[height=7cm,width=13cm]
{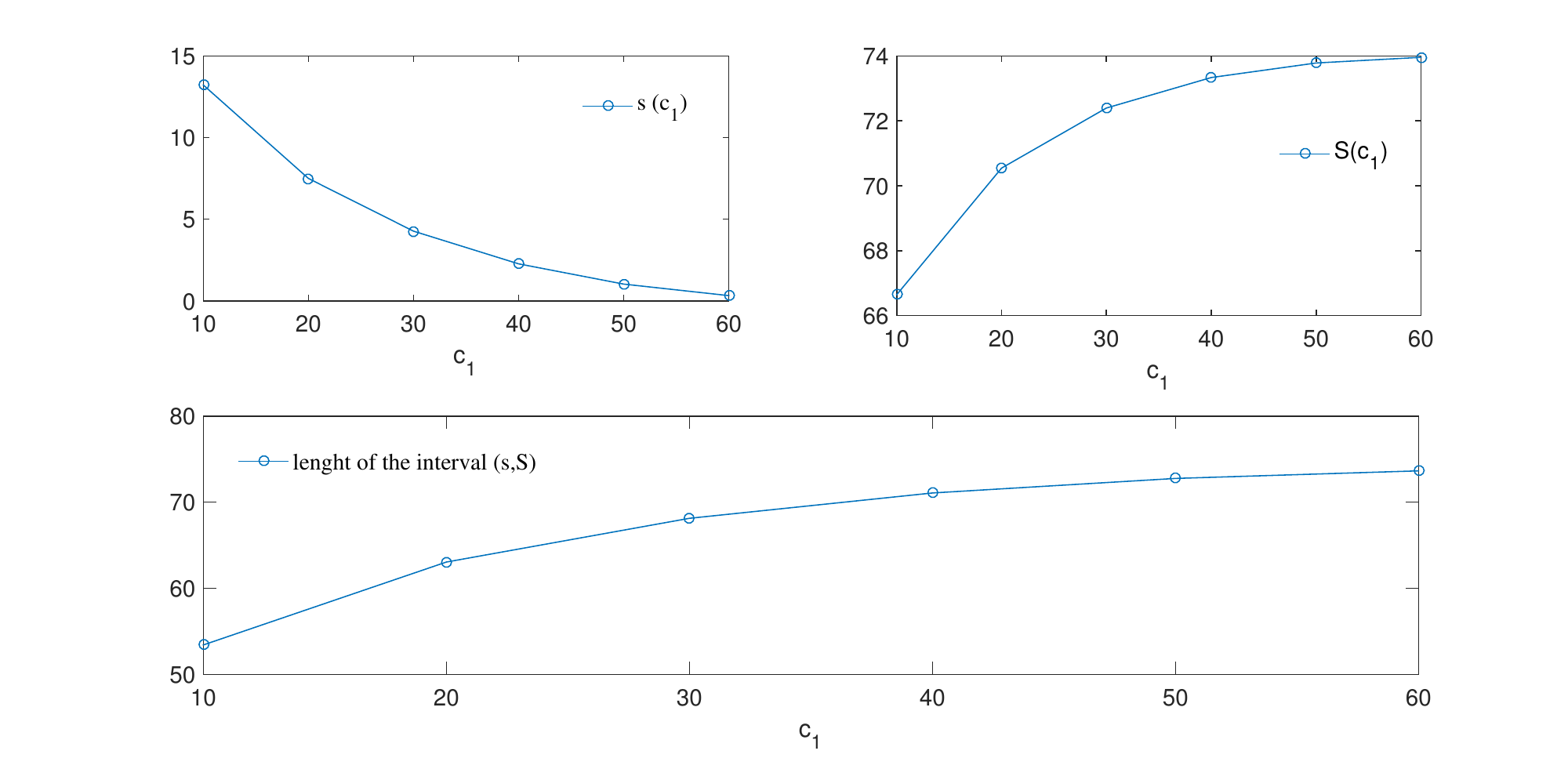}
%{sandSanddiffasdiffc1.eps}
\caption{The trigger level $s$, the target level  $S$, and the difference $S-s$ as functions of $c_1$.}\label{vva}
\end{figure}

\appendix
\section{Appendix} 

\begin{proposition}
Under 
Assumption~\ref{eq:2017-03-23:00} the boundaries  $0$ and $+\infty$ are \emph{natural} in the sense of Feller's classification for the diffusion  $Z^{0,x}$. 
\end{proposition}
\begin{proof}
 Clearly $+\infty$ is \emph{not accessible}, in the sense that $Z^{0,x}$ does not explode in finite time.
It remains to show  that $0$ is \emph{not accessible}, that is
\begin{equation}
  \label{eq:2017-05-18:00bis}
  x\in\R_{++}  \ \Longrightarrow  \
Z^{0,x}_{t}>0 \quad \mathbb{P}\mbox{-a.s.}\  \forall  t\geq 0;
\end{equation}
that both $0$ and $+\infty$ are \emph{not entrance}, that is
\begin{equation}\label{natural}
\lim_{x\downarrow 0}\P\{\tau_{x,y}<t\}=0, \ \ \ \lim_{x\uparrow \infty}\P\{\tau_{x,y}<t\}=0, \ \ \ \forall t,y\in \R_{++}.
\end{equation}
 To this end, 
we introduce the speed measure $m$ of the diffusion $Z^{0,x}$ transformed to natural scale (see \cite[Prop.~16.81, Th.~16.83]{Breiman}). 
Up to a multiplicative constant, we have
$$
m(dy)= \frac{2}{\sigma^2(y)} e^{\int_1^y\frac{2b(\xi)}{\sigma^2(\xi)}d\xi} dy, \ \ \ y\in\R_{++}. 
$$
Assumption~\ref{eq:2017-03-23:00} implies that for some $C_0,C_1>0$ we have 
$|b(\xi)|\leq C_0 \xi$ and $\sigma^2(\xi)\leq C_1\xi^2$ for every $\xi\in\R_+$. 
{\red{According to \cite[Prop.~16.43]{Breiman} we compute
$\int_0^1 ym(dy)
$.
We have
$$
\int_0^1 ym(dy)\geq \int_0^1 \frac{2y}{\sigma^2(y)} e^{\int_1^y\frac{-2C_0\xi}{\sigma^2(\xi)}d\xi} dy.
$$
Set $F(y):=\int_1^y\frac{-2C_0\xi}{\sigma^2(\xi)}d\xi$. We have
\begin{eqnarray*}
\int_0^1 \frac{2y}{\sigma^2(y)} e^{\int_1^y\frac{-2C_0\xi}{\sigma^2(\xi)}d\xi} dy&=&
-\frac{1}{C_0} \int_0^1 F'(y) e^{F(y)} dy=  -\frac{1}{C_0}\left[e^{F(1)}-\lim_{y\to 0^+}e^{F(y)}\right]\\
&=&-\frac{1}{C_0}\left[1-\lim_{y\to 0^+}e^{\int_1^y -\frac{2C_0\xi}{\sigma^2(\xi)}d\xi}\right]=  -\frac{1}{C_0}\left[1-e^{\lim_{y\to 0^+}\int_y^1 \frac{2C_0}{C_1\xi}d\xi}\right]=+\infty.
\end{eqnarray*}}}
This shows, by \cite[Prop.~16.43]{Breiman}, that \eqref{eq:2017-05-18:00bis} holds, 
%Notice that, by uniqueness of solution to \eqref{eq:SDE},
%since $\sigma>0$, if $\xi>0$ then
%the diffusion $Z^{\tau,\xi}$ is {\color{red}nondegenerate} on $(\tau,\infty)$.
%Moreover, 
%Assumption~\ref{eq:2017-03-23:00}
%guarantees that, 
%we have 
%\begin{equation}
%  x\in\R_{++}  \ \Longrightarrow  \
%Z^{x}_{t}>0 \ \mathbb{P}\mbox{-a.s.}\  \forall  t\geq 0.
%\end{equation}
%We also provide a probabilistic proof of \eqref{eq:2017-05-18:00} in 
%Appendix {\red{SF: LASCIARLA? SE SI, TOGLIERE I DATI INIZIALI RANDOM NELLA PROOF}: \color{orange}A noi serve la \eqref{eq:2017-05-18:00} con dato iniziale random, in modo da poter concludere, attraverso la concatenazione, che gli $X^{\tau,\xi,I}$ sono strettamente positivi}}.
The fact that $0$ is not-entrance, i.e.\ that the first limit in \eqref{natural} holds, is then consequence of \cite[Prop.~16.45(a)]{Breiman}. Let us show, finally, that also $+\infty$ is not-entrance, i.e.\ that the second limit in \eqref{natural} holds. In this case, according to  \cite[Prop.~16.45(b)]{Breiman} we consider
 $\int_1^{+\infty} ym(dy)$ and see, with the same computations as above, that it is equal to $+\infty$. By the aforementioned result we conclude that $+\infty$ is not entrance.
 \end{proof}
 \begin{remark} The property
 % \eqref{monotonebis} and
 \eqref{eq:2017-05-18:00bis} can be generalized to the case of random initial data. 
Let
 $\tau$ be a (possibly infinite) $\mathbb{F}$-stopping time and let $\xi$ be an  $\mathcal{F}_\tau$-measurable random variable, clearly we have the equality in law 
$    Z^{\tau,\xi}_{t+\tau}= \left( Z^{0,x}_t \right) _{|_{x=\xi}}$.
%   From 
% \eqref{monotonebis} and \eqref{eq:2017-05-18:00bis} we obtain
% \xi,\eta\ \mathcal{F}_\tau\mbox{-measurable random variables, }
% \xi\leq \eta\ \mbox{$\mathbb{P}$-a.s.}
%   \ \Longrightarrow \ Z^{\xi,\tau}_{t+\tau}\leq Z^{\eta,\tau}_{t+\tau} \  \P\mbox{-a.s.}, \ \forall t\geq 0,
% \end{equation}
% and, by 
By
\eqref{eq:2017-05-18:00bis}, it then follows that
\begin{equation}
  \label{eq:2017-05-18:00}
 \xi\ \mathcal{F}_\tau\mbox{-measurable random variable, }
\xi>0\ \mbox{$\mathbb{P}$-a.s.}  \ \Longrightarrow  \
Z^{\tau,\xi}_{t+\tau}>0 \  \mathbb{P}\mbox{-a.s.\ on \ }\{\tau<\infty\},\  \forall  t\geq 0.
\end{equation}
\end{remark}

\begin{lemma}\label{2017-09-27:01}
% {\color{rose}MR: PULISCI LA DIMOSTRAZIONE E CORREGGI LE COSTANTI.}{\red{SF: FATTO: RIGUARDALA}}
  Let 
%$p\neq 0$ be an even natural number,
 $I\in \mathcal{I}$, $x,y\in \R_{++}$.
%satisfying 
%\eqref{eq:2017-09-17:06}. 
\begin{enumerate}[(i)]
\item\label{2017-11-13:00} We have
\begin{equation}
\label{2017-09-27:15}
       \mathbb{E} \left[ |X^{x,I}_s-X^{y,I}_s|^4
     \right] \leq
 |x-y|^4
e^{C_0
t}\qquad \forall t \geq 0,
\end{equation}
where $C_0 \coloneqq  4L_b+6L_\sigma^2.$
\item\label{2017-11-13:01}
For each $\lambda \in[0,1]$ and $x,y\in\R_{++}$, define $z_\lambda \coloneqq \lambda x+(1-\lambda)y$.
% and  $\Sigma^{\lambda, x,y,I}\coloneqq 
%\lambda X^{x,I}
%   +
% (1-  \lambda) X^{y,I}$.
  Then
  %for $0\leq t\leq T$, 
%if
%$\xi,\xi'\in L^{2p} ( ( \Omega,\mathcal{F}_t,\mathbb{P}),\mathbb{R}_{++})$,
\begin{equation}
    \label{2017-09-27:00}
    \mathbb{E}
     \left[ \left|X^{z_\lambda,I}_t - \lambda X^{x,I}_t
   -
 (1-  \lambda) X^{y,I}_t\right|^2 \right] 
     \leq
% \left( 
%     \mathbb{E}
%      \left[ 
%        |\zeta-\lambda\xi-(1-\lambda)\xi'|^p \right] 
      %+
     A_0\lambda^2(1-\lambda)^2 |x-y|^{4}      
    e^{B_0t} \ \ \ \forall \lambda\in[0,1], \ \forall t\geq 0,
\end{equation}
%and
%\begin{equation}
%       \mathbb{E} \left[ |X^{x,I}_s-X^{y,I}_s|^2
%     \right] \leq
% |x-y|^2
%e^{C
%t}\qquad \forall t \geq 0.
%\end{equation}
where $A_0>0$
% \coloneqq  \frac{\tilde{L}_b+\tilde{ L}_\sigma^2}{4L_b+2L_\sigma^2}$,
and  $B_0 \coloneqq 2L_b+2L_\sigma^2+\tilde{L}_b$.
%\begin{equation*}
%  \begin{aligned}
%  A=
% \lambda^p(1-\lambda)^p\tilde{L}
%+
%2(2p-1)
%\lambda^2(1-\lambda)^2\tilde{L}^2
%\\
%B
%%= &
%% 2pL_b+2p(2p-1)L^2_\sigma
%% +
%% pL_b+p(p-1)L^2_\sigma
%% +(p-1)\tilde{L}
%% +\frac{(p-2)(p-1)}{16}\tilde{ L}^2\\
%=&
%3pL_b
%+(5p^2-3p)L^2_\sigma
%+(p-1)\tilde{L}
%+\frac{(p-2)(p-1)}{16}\tilde{ L}^2\\
%C(p)=&
%pL_b+p(p-1)L_\sigma^2+(p-1)\tilde{ L}.
%  \end{aligned}
%\end{equation*}
\end{enumerate}
\end{lemma}
\begin{proof}
\emph{(\ref{2017-11-13:00})}
 We apply It\^o's formula to $|X^{x,I}-X^{y,I}|^4$ and then --- after a standar localization procedure with stopping times to let the stochastic integral term be a martingale and all the other expectations be well defined and finite; see e.g. the proof of Proposition \ref{prop:preB} --- we take the expectation. We get, also using Assumption \ref{eq:2017-03-23:00},
\begin{equation*}
\begin{split}
\E\left[|X^{x,I}_t-X^{y,I}_t|^4\right]&= |x-y|^4+4\E\int_0^t (X^{x,I}_u-X^{y,I}_u)^3(b(X^{x,I}_u)-b(X^{y,I}_u))du\\&+ 6\E\int_0^t (X^{x,I}_u-X^{y,I}_u)^2(\sigma(X^{x,I}_u)-\sigma(X^{y,I}_u))^2du\\
&\leq   |x-y|^4+(4L_b+6L_\sigma^2)\int_0^t \E\left[|X^{x,I}_u-X^{y,I}_u|^4\right]du.
	\end{split}
\end{equation*}
  The claim follows by Gronwall's inequality.

\noindent\emph{(\ref{2017-11-13:01})}
Define
$\Sigma^{\lambda,x,y,I}\coloneqq 
\lambda X^{x,I}
   +
 (1-  \lambda) X^{y,I}$.
%Since $I$ 
%satisfies 
%\eqref{eq:2017-09-17:06},
%we 
We apply It\^o's formula to the process 
$(X^{z_\lambda,I}-\Sigma^{\lambda,x,y,I})^2$
and then --- after a standar localization procedure with stopping times to let the stochastic integral term be a martingale and all the other expectations are well defined and finite; see e.g. the proof of Proposition \ref{prop:preB}  ---
take the expectation, obtaining,  also using Assumption \ref{eq:2017-03-23:00}, 
 \begin{equation}\label{2017-09-27:02}
   \begin{split}
     \mathbb{E} \left[ (X^{z_\lambda,I}_t-\Sigma^{\lambda,x,y,I}_t)^2
     \right] =&  2 \int_0^t \mathbb{E} \left[
       (X^{z_\lambda,I}_u-\Sigma^{\lambda,x,y,I}_u) \left(
         b(X^{z_\lambda,I}_u)- \lambda b(X^{x,I}_u) - (1-\lambda)
         b(X^{y,I}_u)
       \right) \right]  du\\
     & +  \int_0^t \mathbb{E} \left[
      \left(
         \sigma(X^{z_\lambda,I}_u)- \lambda \sigma(X^{x,I}_u) -
         (1-\lambda) \sigma(X^{y,I}_u)
       \right)^2 \right]  du\\
   \leq  &2 \int_0^t \mathbb{E} \left[
       |X^{z_\lambda,I}_u-\Sigma^{\lambda,x,y,I}_u|\cdot |         b(X^{z_\lambda,I}_u)- b(\Sigma^{\lambda,x,y,I}_u)
       | \right]  du\\
      &+ 2 \int_0^t \mathbb{E} \left[
       |X^{z_\lambda,I}_u-\Sigma^{\lambda,x,y,I}_u |
        \cdot| b(\Sigma^{\lambda,x,y,I}_u)- \lambda b(X^{t,\xi,I}_u) -
         (1-\lambda) b(X^{t,\xi',I}_u)
       | \right]  du\\
     & +2\int_0^t \mathbb{E} \left[   |    
         \sigma(X^{z_\lambda,I}_u)- \sigma(\Sigma^{\lambda,x,y,I}_u)
       |^2 \right]  du\\
& + 2\int_0^t \mathbb{E} \left[|
               \sigma(\Sigma^{\lambda,x,y,I}_u)- \lambda
         \sigma(X^{x,I}_u) - (1-\lambda) \sigma(X^{y,I}_u)
       |^2 \right]  du\\
      \leq &  2\left(L_b+L_\sigma^2 \right) \int_0^t
     \mathbb{E} \left[ |X^{z_\lambda,I}_u-\Sigma^{\lambda,x,y,I}_u|^2
     \right]  du\\
     &+ 2 \int_0^t \mathbb{E} \left[
       |X^{z_\lambda,I}_u-\Sigma^{\lambda,x,y,I}_u |
        \cdot| b(\Sigma^{\lambda,x,y,I}_u)- \lambda b(X^{t,\xi,I}_u) -
         (1-\lambda) b(X^{t,\xi',I}_u)
       | \right]  du\\
     &+2 \int_0^t \mathbb{E} \left[|
               \sigma(\Sigma^{\lambda,x,y,I}_u)- \lambda
         \sigma(X^{x,I}_u) - (1-\lambda) \sigma(X^{y,I}_u)
       |^2 \right]  du.
   \end{split}
 \end{equation} 
By doing the same computations as in 
\cite[p.\ 188]{YongZhou}
in order to obtain
\cite[p.\ 188, formulae (4.22) and (4.23)]{YongZhou},
we have
\begin{eqnarray}
 |     b(\lambda x'+(1-\lambda)x'')
    -\lambda     b(x')
    -(1-\lambda)    b(x'')|\leq 
\tilde{ L}_b\lambda(1-\lambda)|x'-x''|^2 
\qquad  \forall x',x''\in \mathbb{R}_{++},
 \label{eq:2017-09-11:01B}
\\[1em]
 |     \sigma(\lambda x'+(1-\lambda)x'')
    -\lambda     \sigma(x')
    -(1-\lambda)    \sigma(x'')|\leq 
\tilde{ L}_\sigma\lambda(1-\lambda)|x'-x''|^2
\qquad \forall x',x''\in \mathbb{R}_{++},
  %  |     \sigma(\Sigma_r^{t,\xi,\sigma',I})
    % -\lambda     \sigma(X^{t,\xi,I}_r)
    % -(1-\lambda)    \sigma(X^{t,\xi',I}_r)|\leq  \tilde{L}\lambda(1-\lambda)|X^{t,\xi,I}_r-X^{t,\xi',I}_r|^2
  \label{eq:2017-09-11:02B}
\end{eqnarray}
where $\tilde{L}_b, \tilde{L}_\sigma$ are
as in Assumption~\ref{eq:2017-03-23:00}.
Then,
by using
\eqref{eq:2017-09-11:01B}
and
\eqref{eq:2017-09-11:02B} in
\eqref{2017-09-27:02},
 we get
 \begin{equation}\label{2017-09-27:03}
   \begin{split}
     \mathbb{E} \left[ |X^{z_\lambda,I}_s-\Sigma^{\lambda,x,y,I}_s|^2
     \right] 
 \leq & 2\left(L_b+L_\sigma^2 \right) \int_0^t
     \mathbb{E} \left[ |X^{z_\lambda,I}_u-\Sigma^{\lambda,x,y,I}_u|^2
     \right]  du\\
     & + 2\lambda(1-\lambda) \tilde{ L}_b \int_0^t \mathbb{E} \left[
       |X^{z_\lambda,I}_u-\Sigma^{\lambda,x,y,I}_u|\cdot |
          X^{x,I}_u -
          X^{y,I}_u
       |^2 \right]  du\\
     & +
2 \lambda^2(1-\lambda)^2 \tilde{ L}_\sigma^2 \int_0^t \mathbb{E} \left[ |    
         X^{x,I}_u -X^{y,I}_u
       |^4 \right] du.
   \end{split}
 \end{equation}
\noindent 
Using the inequality
 \begin{equation*}\label{2017-09-27:10}
   \begin{aligned} 
2\lambda(1-\lambda) ab   &\leq a^2 + {\lambda^2(1-\lambda)^2}b^{2} \ \ \ \ \  \forall a,b\in\R,
%a^{p-2}b^4 &\leq \frac{p-2}{p}a^p+\frac{2}{p}b^{2p}.
\end{aligned}
\end{equation*}
%$ab\leq \frac{a^q}{q}+\frac{b^p}{p}$ wit $q=p/(p-1)$,
and \eqref{2017-09-27:15} into
\eqref{2017-09-27:03},
we obtain
 \begin{equation*}\label{2017-09-27:08}
   \begin{split}
     \mathbb{E} \left[ |X^{z_\lambda,I}_t-\Sigma^{\lambda,x,y,I}_t|^2
     \right]  \leq 
& \left(2L_b+2L_\sigma^2 +\tilde{L}_b\right) \int_0^t
     \mathbb{E} \left[ |X^{z_\lambda,I}_u-\Sigma^{\lambda,x,y,I}_u|^2
     \right]  du
     \\
&+
\lambda^2(1-\lambda)^2 (\tilde{L}_b+2\tilde{ L}_\sigma^2) \int_0^t \mathbb{E} \left[ |    
         X^{x,I}_u -X^{y,I}_u
       |^4 \right] du\\
         \leq
&  \left(2L_b+2L_\sigma^2 +\tilde{L}_b\right) \int_0^t
     \mathbb{E} \left[ |X^{z_\lambda,I}_u-\Sigma^{\lambda,x,y,I}_u|^2
     \right]  du\\
& +
 (\tilde{L}_b+2\tilde{ L}_\sigma^2) \lambda^2(1-\lambda)^2 \int_0^t e^{C_0u} |x-y|^4 du\\
       \leq 
& \left(2L_b+2L_\sigma^2+\tilde{L}_b \right) \int_0^t
     \mathbb{E} \left[ |X^{z_\lambda,I}_u-\Sigma^{\lambda,x,y,I}_u|^2
     \right]  du\\
& +
 \frac{\tilde{L}_b+2\tilde{ L}_\sigma^2}{C_0}(e^{C_0t}-1) \lambda^2(1-\lambda)^2|x-y|^4,
   \end{split}
 \end{equation*}
 where $C_0$ is the constant of \eqref{2017-09-27:15}.
 We conclude by Gronwall's inequality.
\end{proof}

\end{document}